\documentclass[reqno,11pt]{amsart}
\usepackage[utf8]{inputenc}

\makeatletter
\def\chaptermark#1{}

\def\chapter{%
  \if@openright\cleardoublepage\else\clearpage\fi
  \thispagestyle{plain}\global\@topnum\z@
  \@afterindenttrue \secdef\@chapter\@schapter}

\def\@chapter[#1]#2{\refstepcounter{chapter}%
  \ifnum\c@secnumdepth<\z@ \let\@secnumber\@empty
  \else \let\@secnumber\thechapter \fi
  \typeout{\chaptername\space\@secnumber}%
  \def\@toclevel{0}%
  \ifx\chaptername\appendixname \@tocwriteb\tocappendix{chapter}{#2}%
  \else \@tocwriteb\tocchapter{chapter}{#2}\fi
  \chaptermark{#1}%
  \addtocontents{lof}{\protect\addvspace{10\p@}}%
  \addtocontents{lot}{\protect\addvspace{10\p@}}%
  \@makechapterhead{#2}\@afterheading}
\def\@schapter#1{\typeout{#1}%
  \let\@secnumber\@empty
  \def\@toclevel{0}%
  \ifx\chaptername\appendixname \@tocwriteb\tocappendix{chapter}{#1}%
  \else \@tocwriteb\tocchapter{chapter}{#1}\fi
  \chaptermark{#1}%
  \addtocontents{lof}{\protect\addvspace{10\p@}}%
  \addtocontents{lot}{\protect\addvspace{10\p@}}%
  \@makeschapterhead{#1}\@afterheading}
\newcommand\chaptername{Chapter}

\def\@makechapterhead#1{\global\topskip 7.5pc\relax
  \begingroup
  \fontsize{\@xivpt}{18}\bfseries\centering
    \ifnum\c@secnumdepth>\m@ne
      \leavevmode \hskip-\leftskip
      \rlap{\vbox to\z@{\vss
          \centerline{\normalsize\mdseries
              \uppercase\@xp{\chaptername}\enspace\thechapter}
          \vskip 3pc}}\hskip\leftskip\fi
     #1\par \endgroup
  \skip@34\p@ \advance\skip@-\normalbaselineskip
  \vskip\skip@ }
\def\@makeschapterhead#1{\global\topskip 7.5pc\relax
  \begingroup
  \fontsize{\@xivpt}{18}\bfseries\centering
  #1\par \endgroup
  \skip@34\p@ \advance\skip@-\normalbaselineskip
  \vskip\skip@ }
\def\appendix{\par
  \c@chapter\z@ \c@section\z@
  \let\chaptername\appendixname
  \def\thechapter{\@Alph\c@chapter}}
\newcounter{chapter}
\newif\if@openright
\makeatother

\usepackage{a4wide}
\usepackage{color}
\usepackage{mathrsfs}
\usepackage{mathtools}
\usepackage{amsmath}
\usepackage{amssymb}
\usepackage{bbm}
\usepackage{tikz-cd}
\usepackage{esint}
\usepackage{nicefrac}
\numberwithin{equation}{section}
\usepackage[colorlinks,citecolor=green,linkcolor=red]{hyperref}

\mathtoolsset{showonlyrefs}

\newcommand{\F}{\mathcal{F}}
\newcommand{\nchi}{{\raise.3ex\hbox{\(\chi\)}}}
\newcommand{\N}{\mathbb{N}}
\newcommand{\R}{\mathbb{R}}
\renewcommand{\d}{{\mathrm d}}
\newcommand{\mm}{\mathfrak{m}}

\newcommand{\lims}{\varlimsup}
\newcommand{\eps}{\varepsilon}
\newcommand{\spt}{\mathrm{spt}\,}

\newcommand{\fr}{\penalty-20\null\hfill\(\blacksquare\)}

\newtheorem{theorem}{Theorem}[section]
\newtheorem{corollary}[theorem]{Corollary}
\newtheorem{lemma}[theorem]{Lemma}
\newtheorem{proposition}[theorem]{Proposition}
\newtheorem{definition}[theorem]{Definition}
\newtheorem{example}[theorem]{Example}
\newtheorem{remark}[theorem]{Remark}

\newtheorem{theoremAPP}{Theorem}[chapter]
\newtheorem{corollaryAPP}[theoremAPP]{Corollary}
\newtheorem{lemmaAPP}[theoremAPP]{Lemma}
\newtheorem{propositionAPP}[theoremAPP]{Proposition}
\newtheorem{definitionAPP}[theoremAPP]{Definition}
\newtheorem{remarkAPP}[theoremAPP]{Remark}

\title{Ultralimits of pointed metric measure spaces}
\author{Enrico Pasqualetto}
\address[Enrico Pasqualetto]{Scuola Normale Superiore, Piazza dei Cavalieri 7,
56126 Pisa, Italy}
\email{enrico.pasqualetto@sns.it}
\author{Timo Schultz}
\address[Timo Schultz]{Fakultät für Mathematik, Universität Bielefeld,
Postfach 100131, 33501 Bielefeld, Germany.}
\email{tschultz@math.uni-bielefeld.de}
\begin{document}
\date{\today}
\renewcommand{\thechapter}{\Roman{chapter}}
\begin{abstract}
The aim of this paper is to study ultralimits of pointed metric measure
spaces (possibly unbounded and having infinite mass). We prove that
ultralimits exist under mild assumptions and are consistent with the
pointed measured Gromov--Hausdorff convergence. We also introduce
a weaker variant of pointed measured Gromov--Hausdorff convergence,
for which we can prove a version of Gromov's compactness theorem by
using the ultralimit machinery. This compactness result shows that,
a posteriori, our newly introduced notion of convergence is equivalent
to the pointed measured Gromov one. Another byproduct of our
ultralimit construction is the identification of direct and inverse
limits in the category of pointed metric measure spaces.
\end{abstract}
\maketitle
\tableofcontents
\chapter{Introduction}

It is well-known that in the uniformly doubling framework, the Gromov--Hausdorff limit of a sequence of metric spaces is isometric to the ultralimit of that sequence. Such a relation between Gromov--Hausdorff convergence and ultralimit construction still exists in the setting of pointed metric spaces. Quite recently in \cite{Elek}, Elek generalised -- using the so-called Loeb measure construction --  ultralimits to the context of bounded metric measure spaces with finite measure, and established an analogous connection between ultralimits and the \(\underline\square\)-convergence \`a la Gromov. 

In this paper we generalise the ultralimit construction to the setting where both the metric and the measure can be unbounded, and hence we consider pointed metric measure spaces. We also give a fairly complete picture of its connection to related notions of convergence. More precisely, we prove that, under mild assumptions, pointed measured Gromov--Hausdorff limits are isomorphic to the (support of) the ultralimit. Moreover, we prove that the isomorphism holds also for limits in the pointed measured Gromov topology introduced in \cite{Gigli-Mondino-Savare}. In doing so, we study a weaker variant of pointed measured Gromov--Hausdorff convergence and prove a version of Gromov's compactness theorem in this setting. It turns out that in the end this notion of convergence is in fact equivalent to the pointed measured Gromov convergence.

With the aid of ultralimits of metric measure spaces, we
study direct and inverse limits of sequences of pointed metric
measure spaces, where the metric spaces under consideration are complete, separable, and equipped with boundedly finite Borel
measures. More specifically, we will characterise when these
limits exist in such category and explain how they look like.
\section{General overview}
The aim of this paper is to study ultralimits of metric measure spaces,
as well as their relation with other important kinds of convergence,
such as the measured Gromov--Hausdorff one and its variants.
Before passing to the contents of the paper, we describe the motivations
behind the different notions of convergence for spaces.
\subsection{The role of metric measure spaces}
The notion of a `metric measure space' certainly plays a prominent role
in many fields of mathematics, one of the reasons being that it is
used to provide natural generalisations of manifolds, where all types of
singularities are allowed. At a purely metric level, important classes
of spaces carrying a rich geometric structure are studied. In this
regard, \emph{Alexandrov spaces} constitute a significant example: introduced
by A.\ D.\ Alexandrov in \cite{Ale51}, they are metric spaces where a
synthetic lower bound on the sectional curvature is imposed via a geodesic
triangles comparison. See \cite{BBI01} and the references therein.

Nevertheless, other relevant concepts of curvature (or, rather, of curvature
bounds) require the interplay between distance and `volume measure' to be
captured; in these situations, one needs to work in the category of metric
measure spaces. This is the case, for instance, of lower Ricci curvature
bounds: after the introduction of \emph{Ricci limit spaces} in the series of
pioneering papers \cite{Cheeger-Colding97I,Cheeger-Colding97II,
Cheeger-Colding97III} by Cheeger--Colding in the nineties,
the notion of a \emph{\(\sf CD\) space} (standing for
\(\sf C\)urvature-{\sf D}imension condition) has been proposed by
Lott--Villani \cite{Lott-Villani09} and Sturm \cite{Sturm06I,Sturm06II}.
The former are obtained by
approximation from sequences of Riemannian manifolds satisfying
uniform lower Ricci bounds, while the latter is based on a purely
intrinsic description of a lower Ricci curvature bound.
More precisely, the \(\sf CD\) condition is expressed in terms of
the convexity properties of suitable entropy functionals in the
Wasserstein space of probability measures having finite second-order
moment; in particular, no smooth structure is involved in the definition.
It is then particularly evident that the nonsmooth theory of lower Ricci
bounds must be formulated in the setting of metric measure spaces,
as the entropy functionals under consideration depend both on the
distance and the measure. Let us mention that the family of \(\sf CD\)
spaces -- differently from that of Ricci limits -- contains also
(non-Riemannian) Finsler manifolds. This observation led to the
investigation of the so-called \emph{\(\sf RCD\) condition}, which aims
at selecting those `infinitesimally Hilbertian' \(\sf CD\) spaces that resemble
a Riemannian manifold. The class of \(\sf RCD\) spaces has been
widely studied as well, starting from \cite{AmbrosioGigliSavare11-2,
AmbrosioGigliMondinoRajala12,Gigli12}.
\subsection{Measured Gromov--Hausdorff convergence}
A crucial concept in metric geometry is the \emph{Gromov--Hausdorff
distance}, whose introduction dates back to \cite{Edwards75,Gromov81}.
Roughly speaking, it is a distance on the `space of compact metric spaces'
that quantifies how far two metric spaces are from being isometric.
When dealing with non-compact metric spaces, it is often useful to fix some reference
points and to work with the \emph{pointed Gromov--Hausdorff convergence}.

What often makes a notion of convergence worthwhile is its guarantee of \emph{stability} and \emph{compactness} for a suitable class of objects. 
In the framework of Gromov--Hausdorff convergence, the former refers to the fact
that many classes of spaces -- interesting from a geometric
perspective -- are closed under Gromov--Hausdorff convergence; for
instance, Alexandrov spaces have this property, as observed
by Grove--Petersen \cite{GP91}. The latter means that every sequence of spaces
(where uniform bounds are imposed) subconverges to some limit
space; in this direction, the key result in the Alexandrov framework
was achieved by Burago--Gromov--Perel'man in \cite{BGP92}. For a thorough
account on Gromov--Hausdorff convergence, we refer to the monograph
\cite{BBI01}.

When a measure enters into play -- thus moving from metric spaces
to metric measure spaces -- one proper notion is that of
\emph{(pointed) measured Gromov--Hausdorff convergence}, introduced
by Fukaya in \cite{Fukaya87} to study the behaviour of the eigenvalues
of the Laplacian along a sequence of Riemannian manifolds.
The key idea behind the measured Gromov--Hausdorff convergence is to
couple the Gromov--Hausdorff distance with a suitable weak convergence
of the involved measures. On \(\sf CD\) and \(\sf RCD\) spaces, the
measured Gromov--Hausdorff convergence turned out to be extremely useful:
indeed, both stability and compactness results have been proven in
this setting. Amongst the other notions of convergence that are used within
the lower Ricci bounds theory, we mention the
\emph{\(\mathbb{D}\)-convergence} proposed by Sturm \cite{Sturm06I}
and the \emph{pointed measured Gromov convergence} by
Gigli--Mondino--Savar\'{e} \cite{Gigli-Mondino-Savare}.
Infinite-dimensional \(\sf CD\) spaces need not be locally compact
and may be endowed with infinite (\(\sigma\)-finite) reference measures,
thus it would be quite unnatural
to consider the pointed measured Gromov--Hausdorff distance on them;
the pointed Gromov convergence has been introduced exactly due to this reason.

Notice that while the approach in the above mentioned notions differ from the approach in Gromov's $\underline\square$-convergence, there is still a natural connection between them (see e.g.\ \cite{Sturm06I, Lohr} and cf.\ \cite{Gigli-Mondino-Savare, Elek}).
\subsection{About ultralimits}
Another important approach to convergence in metric geometry is via the \emph{ultralimit} construction
of a sequence of metric spaces, which aims at detecting the asymptotic
behaviour of finite configurations of points in the approximating spaces. 
Ultralimits generalise the concept of Gromov--Hausdorff convergence (as it
was shown, for instance, in \cite{Jansen17}) and rest on a given
\emph{non-principal ultrafilter} on the set of natural numbers.
On the one hand, non-principal ultrafilters constitute a powerful technical
tool, since they automatically provide a simultaneous choice of sublimits.
On the other hand, they cannot be explicitly described, as their existence
is equivalent to a weak form of the Axiom of Choice. We do not enter
into the details of this discussion here, but we refer the interested
reader to \cite{Gromov07,Roe03,BH99}. We point out that also a ultralimit of
measure spaces is available: it is the so-called \emph{Loeb measure space}
\cite{Cutland00}.

Consequently, a natural question arises: is it possible to build ultralimits
of metric measure spaces? An affirmative answer is given in the papers
\cite{Conley-Kechris-Tucker-Drob,Elek} for finite measures, but it seems
that in the existing literature the problem has not been explicitly
investigated for (possibly infinite) \(\sigma\)-finite measures yet.
Roughly, this is the plan we pursue in the present paper:
\begin{itemize}
\item[1)] We construct ultralimits of pointed metric measure spaces
(under mild assumptions on the measures under consideration), by suitably
adapting the strategy used in \cite{Conley-Kechris-Tucker-Drob,Elek}.
\item[2)] We characterise when the ultralimit is `well-behaved', in the
sense that its measure is actually defined on the Borel
\(\sigma\)-algebra; this point will become clearer in the sequel.
\item[3)] We show consistency between ultralimits and measured
Gromov--Hausdorff distance.
\item[4)] We introduce the notion of \emph{weak pointed measured
Gromov--Hausdorff convergence}, which is better suited for a comparison
with ultralimits of metric measure spaces.
\item[5)] We prove a compactness result for the weak pointed measured
Gromov--Hausdorff convergence by appealing to its connection with the
ultralimits, in the spirit of the proof of the original Gromov's compactness
theorem. As a consequence, we obtain its equivalence with the pointed
measured Gromov convergence introduced in \cite{Gigli-Mondino-Savare}.
\item[6)] We apply the ultralimit machinery to give new insights on
\emph{direct and inverse limits} in the category of pointed Polish
metric measure spaces.
\end{itemize}
In addition, we establish in the appendices the following:
\begin{itemize} 
\item[7)] We obtain a variant of Prokhorov theorem regarding ultralimits of measures. 
\item[8)] We point out the existence of weak pointed measured Gromov--Hausdorff tangents to pointwise doubling spaces.
\end{itemize}
Below we provide a more detailed account on the results that will
be achieved in this paper.
\section{Statement of results}
Many of the definitions that we will adopt in this paper are not standard,
thus let us begin by fixing some terminology.
\medskip

By a \textbf{ball measure} on a metric space \((X,d)\) we mean a
\(\sigma\)-finite measure \(\mm\geq 0\) defined on the ball \(\sigma\)-algebra
(\emph{i.e.}, the smallest \(\sigma\)-algebra containing all open balls).
We will refer to the triple \((X,d,\mm)\) as a \textbf{metric measure space}.
It is worth pointing out that if \((X,d)\) is separable, then ball measures
and Borel measures on \(X\) coincide. The \textbf{support} of \(\mm\)
is defined as the set \(\spt(\mm)\) of all points \(x\in X\) such that
\(\mm\big(B(x,r)\big)>0\) for all \(r>0\). Again, on separable spaces
this notion coincides with the classical one. In general, the support
is closed and separable (see Lemma \ref{lem:prop_spt}), but \(\mm\) is
not necessarily concentrated on it. By `being concentrated' we mean
that \(\mm\big(X\setminus\spt(\mm)\big)=0\). A subtle point here is
that this identity is well-posed because \(\spt(\mm)\) belongs to the
ball \(\sigma\)-algebra, as a consequence of Lemma \ref{lem:prop_spt}.
\subsection{Ultralimits of metric measure spaces}
Fix a non-principal ultrafilter \(\omega\) on \(\N\).
Given a sequence \(\big((X_i,d_i,p_i)\big)\) of pointed metric spaces,
we denote by \((X_\omega,d_\omega,p_\omega)\) its ultralimit.
See Sections \ref{s:ultrafilters} and \ref{s:ultralimits_metric}
for a reminder of this terminology. In addition, suppose that each space
\((X_i,d_i)\) is equipped with a ball measure \(\mm_i\). Then we
prove -- under a suitable volume growth condition -- that it is possible
to construct the \textbf{ultralimit of metric measure spaces}
\[
(X_\omega,d_\omega,\mm_\omega,p_\omega)=\lim_{i\to\omega}(X_i,d_i,\mm_i,p_i).
\]
The growth condition we consider is called \textbf{\(\omega\)-uniform
bounded finiteness} and states that
\[
\lim_{i\to\omega}\mm_i\big(B(p_i,R)\big)<+\infty,\quad\text{ for every }R>0.
\]
See Definition \ref{def:omega-ubf}. Under this assumption, we can build
the sought ball measure \(\mm_\omega\) on \(X_\omega\).
Its construction -- which follows along the lines of \cite{Elek} and
\cite{Conley-Kechris-Tucker-Drob} --
is technically involved, thus we postpone its description directly to
Section \ref{s:constr_ultralim_mms}. We just mention that a key role
in the construction is played by the class of \textbf{internal measurable sets}
(see Definition \ref{def:int_meas_sets}), which are defined as the
ultraproducts of the measurable sets along the given sequence \((X_i,d_i,\mm_i)\).

However, even if we start with a sequence of `very good' spaces (say, compact
metric spaces with uniformly bounded diameter and endowed with a Borel
probability measure), the resulting ultralimit measure may still be extremely
wild. The reason is that the ultralimit as a metric space can easily be
non-separable, thus \(\mm_\omega\) might be a `truly' ball (\emph{i.e.},
non-Borel) measure. A pathological instance of this phenomenon is described
in Examples \ref{ex:non_sep_UL_1} and \ref{ex:non_sep_UL_2}, where the
ultralimit measure \(\mm_\omega\) is non-trivial but has empty support.
In view of the above considerations, it is important to understand when
the ultralimit measure is concentrated on its support; Section
\ref{s:sep_issue} will be devoted to this task. In these cases,
\(\mm_\omega\) can be thought of as a Borel measure on a Polish space,
since the support is always closed and separable. The philosophy
is that, even when the ultralimit metric space is extremely big,
the ultralimit measure may be capable of selecting a `more regular'
portion of the space.
More specifically, in Theorem \ref{thm:equiv_conc_spt} we will provide
necessary and sufficient conditions for the measure \(\mm_\omega\) to be
support-concentrated. The most relevant -- which will appear again
later in this introduction -- is what we will call
\textbf{asymptotic bounded \(\mm_\omega\)-total boundedness}
in Definition \ref{def:abm_omegatb}:
\begin{equation}\label{eq:abmtb_intro}
\forall R,r,\varepsilon>0\;\quad\exists M\in\N,\,(x^i_n)_{n=1}^M\subset X_i:\;
\quad\lim_{i\to\omega}\mm_i\big(\bar B(p_i,R)\setminus{\textstyle\bigcup
_{n=1}^M}B(x^i_n,r)\big)\leq\varepsilon.
\end{equation}
Finally, in Section \ref{ss:relation_with_pmGH} we will investigate the
relation between the pointed measured Gromov--Hausdorff convergence
(recalled in Section \ref{s:pmGH}) and the newly introduced notion of
ultralimit. In this case, we consider \textbf{pointed Polish metric
measure spaces} (meaning that the underlying metric space is complete
and separable), whose reference measure is boundedly-finite. Our main
result here is Theorem \ref{thm:pmGH_vs_UL}, where we show that if a
sequence \(\big((X_i,d_i,\mm_i,p_i)\big)\) of pointed Polish metric measure
spaces converges to some \((X_\infty,d_\infty,m_\infty,p_\infty)\)
in the pointed measured Gromov--Hausdorff sense, then
the ultralimit measure \(\mm_\omega\) is support-concentrated and the
limit \((X_\infty,d_\infty,\mm_\infty,p_\infty)\) can be canonically
identified with \(\big(\spt(\mm_\omega),d_\omega,\mm_\omega,p_\omega\big)\).
\subsection{Weak pointed measured Gromov--Hausdorff convergence}
By looking at the proof of the above-mentioned consistency result
(namely, Theorem \ref{thm:pmGH_vs_UL}), one can realise that it is
possible to weaken the notion of pointed measured Gromov--Hausdorff
convergence, but still maintaining a connection with ultralimits.
To be more precise, let us observe that instead of working with
\textbf{\((R,\varepsilon)\)-approximations} (recalled in Definition
\ref{def:REappr}), we can consider what we call
\textbf{weak \((R,\varepsilon)\)-approximations} (introduced in
Definition \ref{def:weakappr}). The difference is that for a weak
\((R,\varepsilon)\)-approximation the `quasi-isometry' assumption
is required to hold \emph{up to a small measure set}. It seems that
this variant better fits into the framework of metric measure spaces,
because in its formulation distance and measure cannot be decoupled
(differently from what happens in the pointed measured Gromov--Hausdorff
case, where to a `purely metric' concept of quasi-isometry, a control on
the relevant measures is imposed only afterward).

Once our notion of \textbf{weak pointed measured Gromov--Hausdorff convergence}
is introduced in Definition \ref{def:wpmGH}, we can investigate its relation
with ultralimits in Section \ref{s:wpmGH_vs_UL}. In Theorem
\ref{thm:pre-Gromov} we prove that if a sequence of pointed Polish metric
measure spaces is asymptotically boundedly \(\mm_\omega\)-totally bounded
\eqref{eq:abmtb_intro}, then it subconverges to
\(\big(\spt(\mm_\omega),d_\omega,\mm_\omega,p_\omega\big)\) with respect to
the weak pointed measured Gromov--Hausdorff convergence. This result
may be regarded as a reinforcement of Theorem \ref{thm:pmGH_vs_UL}, as
it readily implies that each weak pointed measured Gromov--Hausdorff
limit must coincide with the ultralimit (Corollary \ref{cor:conseq_Gromov}).
Having these tools at our disposal, we can then easily obtain a Gromov's
compactness theorem (see Theorem \ref{thm:Gromov_cpt_main}) for weak pointed
measured Gromov--Hausdorff convergence, stating the following: a given
sequence of spaces \(\big((X_i,d_i,\mm_i,p_i)\big)\) is precompact with
respect to the weak pointed measured Gromov--Hausdorff convergence if
and only if it is \textbf{boundedly measure-theoretically totally bounded},
which means that
\begin{equation}\label{eq:bmttb_intro}
\forall R,r,\varepsilon>0\;\quad\exists M\in\N,\,(x^i_n)_{n=1}^M\subset X_i:\;
\quad\sup_{i\in\N}\mm_i\big(\bar B(p_i,R)\setminus{\textstyle\bigcup
_{n=1}^M}B(x^i_n,r)\big)\leq\varepsilon.
\end{equation}
A significant consequence of Theorem \ref{thm:Gromov_cpt_main}
is the equivalence between weak pointed measured Gromov--Hausdorff
convergence and pointed measured Gromov convergence, that we will
achieve in Theorem \ref{thm:wpmGH_vs_pmG}. In particular, by
combining it with the results proven in \cite{Gigli-Mondino-Savare},
we can conclude that for any \(K\in\R\) the classes of \({\sf CD}(K,\infty)\)
and \({\sf RCD}(K,\infty)\) spaces are closed under weak pointed
measured Gromov--Hausdorff convergence.
\subsection{Direct and inverse limits of pointed metric measure spaces}
In Chapter \ref{ch:DL_IL_pmms} we show how ultralimits of pointed metric
measure spaces can be helpful in constructing direct and inverse limits
in the category of pointed Polish metric measure spaces. We address the
problem in a quite restricted setting: first, we just consider inverse
and direct limits, not other limits or colimits in the category; second,
instead of an arbitrary directed index set \((I,\leq)\), we just stick
to the case of natural numbers \(\N\) endowed with the natural order
(in accordance with the rest of the paper). A word on terminology:
in Chapter \ref{ch:DL_IL_pmms} we consider pointed metric measure spaces
\((X,d,\mm,p)\) where \(p\in\spt(\mm)\), in order to get a reasonable
notion of morphism (see Definition \ref{def:morphism_pmms}); this
assumption is not in force in the rest of the paper, as it would
be quite unnatural on non-separable spaces, where non-trivial measures
may have empty support.

In Theorem \ref{thm:DL} and Corollary \ref{cor:DL} we characterise which are
the direct systems of pointed Polish metric measure spaces admitting direct
limit: they are exactly the uniformly boundedly finite ones. Moreover,
in this case the sequence has a weak pointed measured Gromov--Hausdorff
limit, which also coincides with the direct limit itself.

The inverse limit case is more complex: given an inverse system
of pointed Polish metric measure spaces \(\big((X_i,d_i,\mm_i,p_i)\big)\),
one can select a subspace \(X\) of its ultralimit \(X_\omega\) having
the property that the inverse limit exists if and only if
\(p_\omega\in\spt(\mm_\omega\llcorner X)\). Moreover, in this case the
inverse limit coincides with \(X\) equipped with the restricted distance
and measure; see Theorem \ref{thm:IL}. The space \(X\) -- which is defined
in \eqref{eq:def_X_IL} -- \emph{a priori} depends on the choice of the
non-principal ultrafilter \(\omega\): its independence is granted by the
uniqueness of the inverse limit. Actually, differently from what
happens for direct limits, we are currently unaware of a characterisation
of the inverse limit that does not appeal to the theory of ultralimits.

In the context of metric measure space geometry, inverse limits
have been considered -- for instance -- in the paper \cite{CheegerKleiner13}.
\subsection{Tangents to pointwise doubling metric measure spaces}
An essential tool in metric measure geometry is given by the notion of a
\textbf{tangent cone}. For instance, in the setting of finite-dimensional
\(\sf RCD\) spaces, where the study of fine structural properties of the
spaces aroused a great deal of attention, tangent cones play a central role.
Roughly speaking, given a pointed metric measure space \((X,d,\mm,p)\),
what we call a `tangent cone' at \(p\) is any limit of the rescaled spaces
\((X,d/r_i,\mm^p_{r_i},p)\), where \(\mm^p_{r_i}\) are suitable
normalisations of \(\mm\), and \(r_i\searrow 0\). In the case of
finite-dimensional \(\sf RCD\) spaces, limits are taken with respect
to the pointed measured Gromov--Hausdorff distance. We do not enter into
details and we refer to \cite{Semola20,Brue20}.

In Appendix \ref{s:tg_cone} we observe (see Theorem \ref{thm:tg_cone})
that on \textbf{pointwise doubling} metric measure spaces, at almost every
point a weak pointed measured Gromov--Hausdorff tangent cone exists.
By pointwise doubling (or asymptotically doubling, or infinitesimally
doubling) we mean that
\[
\lims_{r\searrow 0}\frac{\mm\big(B(x,2r)\big)}{\mm\big(B(x,r)\big)}
<+\infty,\quad\text{ for }\mm\text{-a.e.\ }x\in X.
\]
In particular, all doubling spaces have this property,
but the class of pointwise doubling spaces is much larger and very well-studied in the framework of analysis on
metric spaces.

It is worth mentioning
that Theorem \ref{thm:tg_cone} applies also to some infinite-dimensional
\(\sf RCD\) spaces, such as the Euclidean space endowed with a
Gaussian measure. Anyway, we believe that this property is well-known to the experts
(in its equivalent formulation using the pointed measured
Gromov convergence) and thus we will not investigate further in that direction.
\subsection*{Acknowledgements}
The authors would like to thank the Department of Mathematics and
Statistics of the University of Jyv\"{a}skyl\"{a}, where much of
this work was done. Both authors were supported by the Academy
of Finland, project number 314789. The first named author was
also supported by the European Research Council (ERC Starting Grant
713998 GeoMeG Geometry of Metric Groups) and by the Balzan project
led by Luigi Ambrosio. The second named author was also supported by the Academy of Finland, project number 308659.
\chapter{Preliminaries}
\section{Terminology about metric/measure spaces}
Let us begin by fixing some general terminology about metric and
measure spaces, which will be used throughout the whole paper.
\subsection{Metric spaces}
An \emph{extended pseudodistance} on a set \(X\) is a
function \(d\colon X\times X\to[0,+\infty]\) such that
the following properties hold:
\[\begin{split}
d(x,x)=0,&\quad\text{ for every }x\in X,\\
d(x,y)=d(y,x),&\quad\text{ for every }x,y\in X,\\
d(x,y)\leq d(x,z)+d(z,y),&\quad\text{ for every }x,y,z\in X.
\end{split}\]
The couple \((X,d)\) is said to be an \emph{extended pseudometric space}.
An extended pseudodistance satisfying \(d(x,y)<+\infty\) for every \(x,y\in X\)
is said to be a \emph{pseudodistance} on \(X\) and the couple \((X,d)\) is
a \emph{pseudometric space}. A pseudodistance such that \(d(x,y)>0\) whenever
\(x,y\in X\) are distinct is called a \emph{distance} on \(X\) and the couple
\((X,d)\) is a \emph{metric space}. Later on, we shall need the following
two well-known results; we omit their standard proofs.
\begin{itemize}
\item Let \((X,d)\) be an extended pseudometric space.
Let \(\bar x\in X\) be given. Let us define
\[
O_{\bar x}(X)\coloneqq\big\{x\in X\,\big|\,d(x,\bar x)<+\infty\big\}.
\]
Then \(\big(O_{\bar x}(X),d\big)\) is a pseudometric space.
(For the sake of brevity, we wrote \(d\) instead of
\(d|_{O_{\bar x}(X)\times O_{\bar x}(X)}\). This abuse of
notation will sometimes appear later in this paper.)
\item Let \((X,d)\) be a pseudometric space. Let \(\sim\) be the following
equivalence relation on the set \(X\): given any \(x,y\in X\), we
declare that \(x\sim y\) provided \(d(x,y)=0\). Denote by \(\tilde X\)
the quotient set \(X/\sim\). We define the function
\(\tilde d\colon\tilde X\times\tilde X\to[0,+\infty)\) as
\[
\tilde d(\tilde x,\tilde y)\coloneqq d(x,y),\quad\text{ for every }
\tilde x,\tilde y\in\tilde X,
\]
where \(x\in X\) (resp.\ \(y\in X\)) is any representative of \(\tilde x\)
(resp.\ \(\tilde y\)). It is easy to check that the definition of
\(\tilde d\) is well-posed, \emph{i.e.}, it does not depend on the
specific choice of the representatives \(x\) and \(y\). Then it holds
that \((\tilde X,\tilde d)\) is a metric space.
\end{itemize}
Given a pseudometric space \((X,d)\), a point \(x\in X\), and a
real number \(r>0\), we denote
\[\begin{split}
B(x,r)&\coloneqq\big\{y\in X\,\big|\,d(x,y)<r\big\},\\
\bar B(x,r)&\coloneqq\big\{y\in X\,\big|\,d(x,y)\leq r\big\}.
\end{split}\]
We call \(B(x,r)\) (resp.\ \(\bar B(x,r)\)) the \emph{open ball}
(resp.\ \emph{closed ball}) of center \(x\) and radius \(r\). More generally, given any \(A\subset X\) and a real number
\(r>0\), we call \(A^r\) the \emph{open
\(r\)-neighbourhood} of the set \(A\), namely
\(A^r\coloneqq\big\{x\in X\,:\,{\rm dist}(x,A)<r\big\}\),
where \({\rm dist}(x,A)\coloneqq\inf_{y\in A}d(x,y)\).
\medskip

If \((X,d)\) is a metric space, then we denote by \(C_{bbs}(X)\)
the family of all bounded, continuous functions \(f\colon X\to\R\)
having bounded support. Recall that the support of \(f\) is defined
as the closure of the set \(\big\{x\in X\,:\,f(x)\neq 0\big\}\).
Moreover, by \(C_b(X)\) we denote the family of all real-valued,
bounded continuous functions defined on \(X\).

\subsection{Measure theory}
Let \(X\) be a non-empty set. Then a family \(\mathcal A\subset 2^X\)
is said to be an \emph{algebra of sets} provided \(X\in\mathcal A\)
and \(A\setminus B\in\mathcal A\) for every \(A,B\in\mathcal A\). If
in addition \(\bigcup_{n\in\N}A_n\in\mathcal A\) for any sequence
\((A_n)_{n\in\N}\subset\mathcal A\), then \(\mathcal A\) is called a
\emph{\(\sigma\)-algebra}. Given a family \(\mathcal F\subset 2^X\),
we denote by \(\sigma(\mathcal F)\) the \emph{\(\sigma\)-algebra
generated by \(\mathcal F\)}, which is the smallest \(\sigma\)-algebra
containing \(\mathcal F\).
By \emph{measurable space} we mean a couple \((X,\mathcal A)\),
where \(X\neq\emptyset\) is a set and \(\mathcal A\subset 2^X\)
a \(\sigma\)-algebra.

If \((X,\mathcal A)\) is a measurable space, \(Y\) a set, and
\(\varphi\colon X\to Y\) an arbitrary map, then we define the
\emph{pushforward \(\sigma\)-algebra} of \(\mathcal A\) under \(\varphi\) as
\[
\varphi_*\mathcal A\coloneqq\big\{B\subset Y\,\big|\,
\varphi^{-1}(B)\in\mathcal A\big\}.
\]
Given two measurable spaces \((X,\mathcal A_X)\), \((Y,\mathcal A_Y)\)
and a map \(\varphi\colon X\to Y\), we say that \(\varphi\) is
\emph{measurable} provided it holds that
\(\mathcal A_Y\subset\varphi_*\mathcal A_X\).
\medskip

A given function \(\mu\colon\mathcal A\to[0,+\infty]\), where \(\mathcal A\)
is an algebra on some non-empty set \(X\), is called a
\emph{set-function} on \((X,\mathcal A)\) provided \(\mu(\emptyset)=0\).
We say that a set-function \(\mu\) on \((X,\mathcal A)\) is a
\emph{finitely-additive measure} provided it satisfies the following property:
\[
\mu(A\cup B)=\mu(A)+\mu(B),\quad\text{ for every }A,B\in\mathcal A
\text{ such that }A\cap B=\emptyset.
\]
If in addition \(\mathcal A\) is a \(\sigma\)-algebra and the set-function
\(\mu\) satisfies
\[
\mu\big({\textstyle\bigcup_{n\in\N}A_n\big)=\sum_{n\in\N}\mu(A_n)},
\quad\text{ whenever }(A_n)_{n\in\N}\subset\mathcal A\text{ are pairwise disjoint},
\]
then we say that \(\mu\) is a \emph{(countably-additive) measure}
on \((X,\mathcal A)\). By \emph{measure space} we mean a triple
\((X,\mathcal A,\mu)\), where \((X,\mathcal A)\) is a measurable
space and \(\mu\) is a measure on \((X,\mathcal A)\).
We say that \(\mu\) is \emph{finite} provided \(\mu(X)<+\infty\),
while it is \emph{\(\sigma\)-finite} provided there exists a sequence
\((A_n)_{n\in\N}\subset\mathcal A\) of measurable sets satisfying
\(X=\bigcup_{n\in\N}A_n\) and \(\mu(A_n)<+\infty\) for every \(n\in\N\).
Moreover, we say that the measure \(\mu\) is \emph{concentrated} on a set
\(A\in\mathcal A\) provided \(\mu(X\setminus A)=0\).
{
Given a non-empty set \(A\in\mathcal A\) and a \(\sigma\)-algebra
\(\mathcal A'\) on \(A\) such that \(\mathcal A'\subset\mathcal A\),
we can consider the \emph{restriction} \(\mu|_{\mathcal A'}\)
of \(\mu\) to \(\mathcal A'\). It trivially holds that
\(\mu|_{\mathcal A'}\) is a measure on \((A,\mathcal A')\).
In the case where \(\mathcal A'=\mathcal A\llcorner A\coloneqq
\big\{A'\in\mathcal A\,:\,A'\subset A\big\}\), we just write
\(\mu\llcorner A\) in place of \(\mu|_{\mathcal A\llcorner A}\).
}
\medskip

If \((X,\mathcal A_X,\mu)\) is a measure space, \((Y,\mathcal A_Y)\)
a measurable space, and \(\varphi\colon X\to Y\) a measurable map,
then we define the \emph{pushforward measure} of \(\mu\) under \(\varphi\) as
\[
(\varphi_*\mu)(B)\coloneqq\mu\big(\varphi^{-1}(B)\big),
\quad\text{ for every }B\in\mathcal A_Y.
\]
Observe that \(\varphi_*\mu\) is a measure on \((Y,\mathcal A_Y)\).
If \(\mu\) is finite, then \(\varphi_*\mu\) is finite as well.
On the other hand, if \(\mu\) is \(\sigma\)-finite, then \(\varphi_*\mu\)
needs not be \(\sigma\)-finite, as easy counterexamples show.
\subsection{Ball measures and Borel measures}
Let \((X,d)\) be a pseudometric space. Then we denote by
\({\rm Ball}(X)\) the \emph{ball \(\sigma\)-algebra} on \(X\),
which is defined as the \(\sigma\)-algebra generated by the
open balls (or, equivalently, by the closed balls). A measure
on \(\big(X,{\rm Ball}(X)\big)\) is said to be a \emph{ball measure}.
Moreover, the \emph{Borel \(\sigma\)-algebra} on \(X\) is the
\(\sigma\)-algebra \(\mathscr B(X)\) generated by the topology of \((X,d)\).
A measure on \(\big(X,\mathscr B(X)\big)\) is called a \emph{Borel measure}.
With a slight abuse of notation, we shall sometimes say that a measure
is a ball measure (resp.\ a Borel measure) if it is defined on a
\(\sigma\)-algebra containing the ball \(\sigma\)-algebra (resp.\ the
Borel \(\sigma\)-algebra).

A ball measure \(\mu\) is said to be
\emph{boundedly finite} provided \(\mu\big(B(x,r)\big)<+\infty\)
for every \(x\in X\) and \(r>0\). Observe that boundedly finite
ball measures are \(\sigma\)-finite, as a consequence of the identity
\(X=\bigcup_{n\in\N}B(\bar x,n)\), where \(\bar x\in X\) is given.
Since open balls are open sets, we have
\[
{\rm Ball}(X)\subset\mathscr B(X),\quad
\text{ for every pseudometric space }(X,d).
\]
However, it might happen (unless \((X,d)\) is separable) that
\({\rm Ball}(X)\neq\mathscr B(X)\).
\begin{definition}[Metric measure space]
By \emph{metric measure space} we mean a triple \((X,d,\mm)\),
where \((X,d)\) is a metric space, while \(\mm\) is a ball measure
on \(X\).
\end{definition}
Given a ball measure \(\mu\) on a pseudometric space \((X,d)\),
we define its \emph{support} \({\rm spt}(\mu)\) as
\begin{equation}\label{eq:def_spt}\begin{split}
{\rm spt}(\mu)\coloneqq&\Big\{x\in X\,\Big|\,\mu\big(B(x,r)\big)>0
\text{ for every }r>0\Big\}\\
=&X\setminus\bigcup\Big\{B(x,r)\,\Big|\,x\in X,\,r>0,\,
\mu\big(B(x,r)\big)=0\Big\}.
\end{split}\end{equation}
A warning about our notation: if \(\mu\) is a Borel measure (and
\((X,d)\) is not separable), then the above definition of \({\rm spt}(\mu)\)
might differ from other ones that appear in the literature.
Let us now collect in the following result a few basic properties
of the support of a ball measure.
\begin{lemma}\label{lem:prop_spt}
Let \((X,d)\) be a pseudometric space. Let \(\mu\) be a ball measure
on \(X\). Then \({\rm spt}(\mu)\) is a closed set. If \(\mu\)
is \(\sigma\)-finite, then \({\rm spt}(\mu)\) is separable and in
particular \({\rm spt}(\mu)\in{\rm Ball}(X)\).
\end{lemma}
\begin{proof}
First of all, observe that \(X\setminus{\rm spt}(\mu)\) is open
(as a union of open balls), thus \({\rm spt}(\mu)\) is closed.
Now assume that \(\mu\) is \(\sigma\)-finite. We aim to show that
\({\rm spt}(\mu)\) is separable. We argue by contradiction: suppose
\({\rm spt}(\mu)\) is not separable. Then we claim that \emph{there exist
an uncountable set \(\{x_i\}_{i\in I}\subset{\rm spt}(\mu)\) and
\(\varepsilon>0\) such that \(B(x_i,\varepsilon)\cap B(x_j,\varepsilon)
=\emptyset\) for all \(i,j\in I\) with \(i\neq j\).}
To prove it, for any \(n\in\N\) pick a maximal \(\frac{2}{n}\)-separated subset
\(S_n\) of \({\rm spt}(\mu)\), whose existence follows by a standard
application of Zorn's lemma. Namely, we have that
\[\begin{split}
d(x,y)\geq\frac{2}{n},\quad&\text{for every }x,y\in S_n\text{ with }x\neq y,\\
\forall z\in{\rm spt}(\mu)\quad&\exists\,x\in S_n:\quad d(x,z)<\frac{2}{n}.
\end{split}\]
Then \(\bigcup_{n\in\N}S_n\) is dense in \({\rm spt}(\mu)\) by construction.
Being \({\rm spt}(\mu)\) non-separable, we deduce that \(S_{\bar n}\) must
be uncountable for some \(\bar n\in\N\). Call \(\{x_i\}_{i\in I}\coloneqq
S_{\bar n}\) and \(\varepsilon\coloneqq 1/\bar n\). Observe that
\begin{equation}\label{eq:prop_spt_aux}
B(x_i,\varepsilon)\cap B(x_j,\varepsilon)=\emptyset,
\quad\text{ for every }i,j\in I\text{ with }i\neq j,
\end{equation}
thus the claim is proven. Given that \(\mu\) is \(\sigma\)-finite,
we can find a sequence \((A_n)_{n\in\N} \subset{\rm Ball}(X)\)
such that \(X=\bigcup_{n\in\N}A_n\) and \(\mu(A_n)<+\infty\) for all
\(n\in\N\). Given any \(i\in I\), there exists \(n_i\in\N\) such that
\(\mu\big(A_{n_i}\cap B(x_i,\varepsilon)\big)>0\). Let us now define
\[
I_{nk}\coloneqq\Big\{i\in I\,\Big|\,n_i=n,\,\mu\big(A_n\cap
B(x_i,\varepsilon)\big)\geq 1/k\Big\},\quad\text{ for every }n,k\in\N.
\]
Given any finite subset \(F\) of \(I_{nk}\), we may estimate
\[
\frac{\# F}{k}\leq\sum_{i\in F}\mu\big(A_n\cap B(x_i,\varepsilon)\big)
\overset{\eqref{eq:prop_spt_aux}}=
\mu\big(A_n\cap{\textstyle\bigcup_{i\in F}}B(x_i,\varepsilon)\big)\leq\mu(A_n),
\]
whence it follows that \(\# I_{nk}\leq k\,\mu(A_n)\) and in particular
\(I_{nk}\) is finite. This is in contradiction with the fact that
\(I=\bigcup_{n,k\in\N}I_{nk}\) is uncountable. Therefore, \({\rm spt}(\mu)\)
is proven to be separable.

Finally, by using both the closedness and the separability of
\({\rm spt}(\mu)\), we see that
\[
{\rm spt}(\mu)=\bigcap_{k\in\N}\bigcup_{n\in\N}B(x_n,1/k),
\]
where \((x_n)_{n\in\N}\subset{\rm spt}(\mu)\) is any dense sequence.
This yields \({\rm spt}(\mu)\in{\rm Ball}(X)\), as required.
\end{proof}
\begin{remark}\label{rmk:C_in_Ball}{\rm
In the last part of the proof of Lemma \ref{lem:prop_spt}, we used the
following general fact: \emph{Let \((X,d)\) be a pseudometric space.
Then \(C\in{\rm Ball}(X)\) for every \(C\subset X\) closed and separable.}

Indeed, given any dense sequence \((x_n)_{n\in\N}\) in \(C\), it holds
\(C=\bigcap_{k\in\N}\bigcup_{n\in\N}B(x_n,1/k)\).
\fr}\end{remark}
\section{Pointed measured Gromov--Hausdorff convergence}\label{s:pmGH}
With a slight abuse of notation, we say that any complete and
separable metric space \((X,d)\) is a \emph{Polish metric space}.
(Typically, this terminology is referring to a topological space
whose topology is induced by a complete and separable distance.)
\begin{definition}[\((R,\varepsilon)\)-approximation]\label{def:REappr}
Let \((X,d_X,p_X)\), \((Y,d_Y,p_Y)\) be pointed Polish metric
spaces. Let \(R,\varepsilon>0\) be such that \(\varepsilon<R\).
Then a given Borel map \(\psi\colon B(p_X,R)\to Y\) is said to be a
\emph{\((R,\varepsilon)\)-approximation} provided it satisfies
\(\psi(p_X)=p_Y\),
\[
\sup_{x,y\in B(p_X,R)}\Big|d_X(x,y)-d_Y\big(\psi(x),\psi(y)\big)
\Big|\leq\varepsilon,\qquad B(p_Y,R-\varepsilon)\subset
\psi\big(B(p_X,R)\big)^\varepsilon.
\]
\end{definition}
\begin{remark}\label{rmk:quasi_isom}{\rm
\leavevmode
\begin{enumerate}
\item Notice that it is always possible to extend the map (in a Borel manner) to the whole space $X$. Therefore, we will often consider $(R,\varepsilon)$-approximations as maps $\psi\colon X\to Y$ to avoid unnecessary complication in the presentation.
\item Being $(R,\varepsilon)$-approximation means -- roughly speaking -- that $\psi$ is roughly isometric and roughly surjective to the corresponding ball on the image side.\fr
\end{enumerate}
}\end{remark}
\begin{lemma}\label{lem:rough_inverse}
Let \((X,d_X,p_X)\), \((Y,d_Y,p_Y)\) be pointed Polish
metric spaces. Let \(\psi\colon X\to Y\) be a given
\((R,\varepsilon)\)-approximation, for some \(R,\varepsilon>0\)
satisfying \(4\varepsilon<R\). Then there exists a
\((R-\varepsilon,3\varepsilon)\)-approximation \(\phi\colon Y\to X\)
-- that we will call a \emph{quasi-inverse} of \(\psi\) -- such that
\begin{subequations}
\begin{align}\label{eq:quasi-inverse_1}
d_X\big(x,(\phi\circ\psi)(x)\big)<3\varepsilon,&
\quad\text{ for every }x\in B(p_X,R-4\varepsilon),\\
\label{eq:quasi-inverse_2}
d_Y\big(y,(\psi\circ\phi)(y)\big)<3\varepsilon,&
\quad\text{ for every }y\in B(p_Y,R-\varepsilon).
\end{align}\end{subequations}
\end{lemma}
\begin{proof}
Let \((x_i)_i\) be a dense sequence in \(B(p_X,R)\).
Given any \(y\in B(p_Y,R-\varepsilon)\setminus\{p_Y\}\), we define
\(\phi(y)\coloneqq x_{i_y}\), where \(i_y\) is the smallest index
\(i\in\N\) for which \(d_Y\big(\psi(x_i),y\big)<\varepsilon\).
Also, we set \(\phi(y)\coloneqq\tilde x\) for all
\(y\notin B(p_Y,R-\varepsilon)\setminus\{p_Y\}\),
where \(\tilde x\) is any given element of
\(X\setminus B(p_X,R-\eps)\) if \(X\setminus B(p_X,R-\eps)\neq\emptyset\)
and \(\tilde x\coloneqq p_X\) otherwise. The resulting map
\(\phi\colon Y\to X\) is Borel and satisfies \(\phi(p_Y)=p_X\)
by construction. For every \(y,y'\in B(p_Y,R-\varepsilon)\),
we have that
\[\begin{split}
&\Big|d_X\big(\phi(y),\phi(y')\big)-d_Y(y,y')\Big|\\
\leq\,&\Big|d_X\big(\phi(y),\phi(y')\big)-
d_Y\big((\psi\circ\phi)(y),(\psi\circ\phi)(y')\big)\Big|+\Big|
d_Y\big((\psi\circ\phi)(y),(\psi\circ\phi)(y')\big)-d_Y(y,y')\Big|\\
\leq\,&\varepsilon+d_Y\big((\psi\circ\phi)(y),y\big)
+d_Y\big((\psi\circ\phi)(y'),y'\big)<3\varepsilon,
\end{split}\]
which shows that \(\sup_{y,y'\in B(p_Y,R-\varepsilon)}
\big|d_X\big(\phi(y),\phi(y')\big)-d_Y(y,y')\big|\leq 3\varepsilon\).
Moreover, given any \(x\in B(p_X,R-4\varepsilon)\), we have that
\(\psi(x)\in B(p_Y,R-3\varepsilon)\), thus from the validity
of the inequalities
\begin{equation}\label{eq:aux_quasi-inverse}
d_X\big((\phi\circ\psi)(x),x\big)\leq
d_Y\big((\psi\circ\phi\circ\psi)(x),\psi(x)\big)+\varepsilon
<2\varepsilon<3\varepsilon
\end{equation}
it follows that \(x\) belongs to the \(3\varepsilon\)-neighbourhood
of \(\phi\big(B(p_Y,R-3\varepsilon)\big)\). In particular,
it holds that \(B(p_X,R-4\varepsilon)\subset
\phi\big(B(p_Y,R-\varepsilon)\big)^{3\varepsilon}\),
whence accordingly \(\phi\) is a
\((R-\varepsilon,3\varepsilon)\)-approximation. In order to
conclude, it only remains to observe that \eqref{eq:quasi-inverse_2}
is a direct consequence of the very definition of \(\phi\),
while \eqref{eq:quasi-inverse_1} follows from the estimate
in \eqref{eq:aux_quasi-inverse}.
\end{proof}
\begin{lemma}
Let \((X,d_X,p_X)\), \((Y,d_Y,p_Y)\) be pointed Polish metric spaces.
Let \(R,r,r',\varepsilon>0\) be such that \(r+r'<R-3\varepsilon\) and
\(r>3\varepsilon\). Let \(\psi\colon X\to Y\) be a given
\((R,\varepsilon)\)-approximation, with quasi-inverse
\(\phi\colon Y\to X\). Then for every point \(y\in B(p_Y,r')\) it holds that
\begin{subequations}
\begin{align}\label{eq:tech_GH_1}
B(y,r-3\varepsilon)&\subset\psi\big(B(\phi(y),r)\big)^{3\varepsilon},\\
\label{eq:tech_GH_2}
\psi^{-1}\big(B(y,r-3\varepsilon)\big)&\subset B\big(\phi(y),r+4\varepsilon\big).
\end{align}
\end{subequations}
\end{lemma}
\begin{proof}
Let \(y\in B(p_Y,r')\) be fixed. Pick any point \(z\in B(y,r-3\varepsilon)\).
We aim to show that we have \(z\in\psi\big(B(\phi(y),r)\big)^{3\varepsilon}\),
whence \eqref{eq:tech_GH_1} would follow. Notice that \(d_Y(y,p_Y)<r'<R-\varepsilon\)
and \(d_Y(z,p_Y)\leq d_Y(z,y)+d_Y(y,p_Y)<(r-3\varepsilon)+r'<R-\varepsilon\),
thus \(y,z\in B(p_Y,R-\varepsilon)\) and
\[
d_X\big(\phi(z),\phi(y)\big)\leq d_Y(z,y)+3\varepsilon<(r-3\varepsilon)+3\varepsilon=r,
\]
where we used the fact that \(\phi\) is a \((R-\varepsilon,3\varepsilon)\)-approximation.
Then \(\phi(z)\in B\big(\phi(x),r\big)\), so that from the estimate
\(d_Y\big(z,(\psi\circ\phi)(z)\big)<3\varepsilon\) -- which is granted by
\eqref{eq:quasi-inverse_2} -- we deduce that the point \(z\) belongs to
\(\psi\big(B(\phi(y),r)\big)^{3\varepsilon}\), as required.

In order to prove \eqref{eq:tech_GH_2}, we argue by contradiction: suppose there
is \(z\in X\setminus B\big(\phi(y),r+4\varepsilon\big)\) with
\(\psi(z)\in B(y,r-3\varepsilon)\). Since \(B(y,r-3\varepsilon)\subset B(p_Y,R)\) and
\(\psi\big(X\setminus B(p_X,R)\big)\subset Y\setminus B(p_Y,R)\),
we have that \(z\in B(p_X,R)\). Moreover, \eqref{eq:tech_GH_1} grants the existence
of \(w\in B\big(\phi(y),r\big)\) such that \(d_Y\big(\psi(z),\psi(w)\big)<3\varepsilon\).
Notice that \(w\in B(p_X,R)\), as it is granted by the following estimates:
\[\begin{split}
d_X(w,p_X)&\leq d_X\big(w,\phi(y)\big)+d_X\big(\phi(y),p_X\big)
<r+d_X\big(\phi(y),\phi(p_Y)\big)\\
&\leq r+d_Y(y,p_Y)+3\varepsilon<r+r'+3\varepsilon<R.
\end{split}\]
Therefore, we deduce that \(d_X(z,w)\leq d_Y\big(\psi(z),\psi(w)\big)
+\varepsilon<4\varepsilon\). On the other hand, it holds
\[
d_X(z,w)\geq d_X\big(z,\phi(y)\big)-d_X\big(w,\phi(y)\big)
>(r+4\varepsilon)-r=4\varepsilon,
\]
which leads to a contradiction. Consequently, the claimed inclusion
\eqref{eq:tech_GH_2} is proven.
\end{proof}
\begin{definition}[Pointed Gromov--Hausdorff convergence]
Let \(\big\{(X_i,d_i,p_i)\big\}_{i\in\bar\N}\) be a sequence
of pointed Polish metric spaces. Then we say that \((X_i,d_i,p_i)\)
converges to \((X_\infty,d_\infty,p_\infty)\) in the
\emph{pointed Gromov--Hausdorff sense} (briefly, \emph{pGH sense})
as \(i\to\infty\) provided there exists a sequence
\((\psi_i)_{i\in\N}\) of \((R_i,\varepsilon_i)\)-approximations
\(\psi_i\colon X_i\to X_\infty\), for some \(R_i\nearrow+\infty\)
and \(\varepsilon_i\searrow 0\).
\end{definition}
\begin{definition}[Pointed measured Gromov--Hausdorff convergence]\label{def:pmGH}
Let \(\big\{(X_i,d_i,\mm_i,p_i)\big\}_{i\in\bar\N}\) be a sequence
of pointed Polish metric measure spaces, with \(\mm_i\) boundedly finite.
Then we say that the sequence \((X_i,d_i,\mm_i,p_i)\) converges to \((X_\infty,d_\infty,
\mm_\infty,p_\infty)\) in the \emph{pointed measured Gromov--Hausdorff sense}
(briefly, \emph{pmGH sense}) provided \((X_i,d_i,p_i)\to(X_\infty,d_\infty,p_\infty)\)
in the pGH sense and there exists a sequence \((\psi_i)_{i\in\N}\) of
\((R_i,\varepsilon_i)\)-approximations exploiting the pGH
convergence such that \((\psi_i)_*\mm_i\rightharpoonup\mm_\infty\)
in duality with \(C_{bbs}(X_\infty)\), namely,
\[
\int\varphi\circ\psi_i\,\d\mm_i\longrightarrow\int\varphi
\,\d\mm_\infty,\quad\text{ for every }\varphi\in C_{bbs}(X_\infty).
\]
\end{definition}
Let us illustrate which is the relation between pGH convergence and
pmGH convergence. On the one hand, given a sequence
\(\big((X_i,d_i,\mm_i,p_i)\big)\) that converges to
\((X_\infty,d_\infty,\mm_\infty,p_\infty)\) in the pmGH sense,
it holds that \(\big((X_i,d_i,p_i)\big)\) converges to
\((X_\infty,d_\infty,p_\infty)\) in the pGH sense. On the other
hand, if we consider a sequence \(\big((X_i,d_i,p_i)\big)\) that
pGH-converges to \((X_\infty,d_\infty,p_\infty)\) and a measure
\(\mm_\infty\) on \(X_\infty\), we can find a sequence of measures
\((\mm_i)_{i\in\N}\) such that \(\big((X_i,d_i,\mm_i,p_i)\big)\)
pmGH-converges to \((X_\infty,d_\infty,\mm_\infty,p_\infty)\);
the latter claim is the content of the following result. This shows
that -- in a sense -- the pmGH convergence is essentially a metric
concept, where a control on the behaviour of the measures is
added subsequently.
\begin{lemma}
Let \(\big\{(X_i,d_i,p_i)\big\}_{i\in\bar\N}\) be a sequence of pointed Polish
metric spaces. Suppose that \((X_i,d_i,p_i)\to(X_\infty,d_\infty,p_\infty)\) in the
pGH sense. Let \(\psi_i\colon X_i\to X_\infty\) be \((R_i,\varepsilon_i)\)-approximations
and \(\mm_\infty\) a Borel probability measure on \(X_\infty\).
Then there exists a sequence \((\mm_i)_{i\in\N}\)
of Borel probability measures \(\mm_i\in\mathscr P(X_i)\) such that
\begin{equation}\label{eq:from_pGH_to_pmGH}
(\psi_i)_*\mm_i\rightharpoonup\mm_\infty,
\quad\text{ in duality with }C_{bbs}(X_\infty).
\end{equation}
In particular, it holds that
\((X_i,d_i,\mm_i,p_i)\to(X_\infty,d_\infty,\mm_\infty,p_\infty)\)
in the pmGH sense.
\end{lemma}
\begin{proof}
Up to a diagonalisation argument, it is sufficient to prove the
claim for a fully-supported measure \(\mm_\infty\) of the form
\begin{equation}\label{eq:m_infty_delta}
\mm_\infty\coloneqq\sum_{j\in\N}\lambda_j\,\delta_{x_j},
\end{equation}
where \((x_j)_{j\in\N}\) is dense in \(X_\infty\) and
\((\lambda_j)_{j\in\N}\subset[0,1]\) satisfies
\(\sum_{j\in\N}\lambda_j=1\). Indeed, these measures are
weakly dense in the space \(\mathscr P(X_\infty)\).
Then let \(\mm_\infty\) be a fixed measure as in
\eqref{eq:m_infty_delta}. Given any \(i\in\N\) and \(x\in B(p_i,R_i)\), we define
\(j(i,x)\in\N\) as the smallest number \(j\in\N\) such that
\(d_\infty\big(\psi_i(x),x_j\big)<\varepsilon_i\). Let us
define \(J_i\coloneqq\big\{j(i,x)\,:\,x\in B(p_i,R_i)\big\}\subset\N\)
for every \(i\in\N\). Given any \(j\in J_i\), choose any point
\(x^i_j\in B(p_i,R_i)\) such that \(j(i,x^i_j)=j\). Observe that
\begin{equation}\label{eq:approx_xij}
d_\infty\big(\psi_i(x^i_j),x_j\big)<\varepsilon_i,
\quad\text{ for every }i\in\N\text{ and }j\in J_i.
\end{equation}
Now for any \(i\in\N\) let us define the Borel probability measure \(\mm_i\) on \(X_i\) as
\[
\mm_i\coloneqq\frac{1}{c_i}\sum_{j\in J_i}\lambda_j\,\delta_{x^i_j},
\quad\text{ where we set }c_i\coloneqq\sum_{j\in J_i}\lambda_j.
\]
We claim that
\begin{equation}\label{eq:claim_incl_Ji}
\forall k\in\N\quad\exists\,\bar i\in\N:\quad\forall i\geq\bar i
\quad\{1,\ldots,k\}\subset J_i.
\end{equation}
In order to prove it, let \(k\in\N\) be fixed.
Choose any \(\bar i\in\N\) so that
\begin{equation}\label{eq:choice_bar_i}\begin{split}
x_1,\ldots,x_k\in B(p_\infty,R_i-\varepsilon_i),&\quad
\text{ for every }i\geq\bar i,\\
d_\infty(x_j,x_{j'})>2\varepsilon_i,&\quad\text{ for every }
i\geq\bar i\text{ and }j,j'\in\{1,\ldots,k\}\text{ with }j\neq j'.
\end{split}\end{equation}
Given any \(j\in\{1,\ldots,k\}\), it follows from the first line in
\eqref{eq:choice_bar_i} that for every \(i\geq\bar i\) there exists a point
\(x\in B(p_i,R_i)\) such that \(d_\infty\big(\psi_i(x),x_j\big)<\varepsilon_i\).
Since the second line in \eqref{eq:choice_bar_i} yields
\[
d_\infty\big(\psi_i(x),x_{j'}\big)\geq d_\infty(x_j,x_{j'})
-d_\infty\big(\psi_i(x),x_j\big)>2\varepsilon_i-\varepsilon_i
=\varepsilon_i,\quad\text{ for every }j'<j,
\]
we deduce that \(j=j(i,x)\) and accordingly \(j\in J_i\). This proves
the validity of the claim \eqref{eq:claim_incl_Ji}.

It remains to prove that \((\psi_i)_*\mm_i\rightharpoonup\mm_\infty\)
in duality with \(C_{bbs}(X_\infty)\). Let \(f\in C_{bbs}(X_\infty)\) be
fixed and put \(M\coloneqq\sup_{X_\infty}|f|\). Fix any \(\varepsilon>0\).
Pick any \(k\in\N\) such that \(M\sum_{j>k}\lambda_j<\varepsilon\). By using
\eqref{eq:claim_incl_Ji} we can find \(\bar i\in\N\) such that
\(\{1,\ldots,k\}\subset J_i\) for every \(i\geq\bar i\). By using
\eqref{eq:approx_xij} and the continuity of the function \(f\), we
see that
\begin{equation}\label{eq:cont_f_xj}
\lim_{i\to\infty}f\big(\psi_i(x^i_j)\big)=f(x_j),
\quad\text{ for every }j\in\{1,\ldots,k\}.
\end{equation}
Moreover, for every \(i\geq\bar i\) we can estimate
\[\begin{split}
&\bigg|c_i\int f\,\d(\psi_i)_*\mm_i-\int f\,\d\mm_\infty\bigg|
=\bigg|\sum_{j\in J_i}\lambda_j\,f\big(\psi_i(x^i_j)\big)
-\sum_{j\in\N}\lambda_j\,f(x_j)\bigg|\\
\leq\,&\bigg|\sum_{j=1}^k\lambda_j\big(f\big(\psi_i(x^i_j)\big)
-f(x_j)\big)\bigg|+\sum_{\substack{j\in J_i:\\j>k}}
\lambda_j\big|f\big(\psi_i(x^i_j)\big)\big|+
\sum_{j>k}\lambda_j\big|f(x_j)\big|\\
\leq\,&\bigg|\sum_{j=1}^k\lambda_j\,\big(f\big(\psi_i(x^i_j)\big)-f(x_j)\big)\bigg|
+2M\sum_{j>k}\lambda_j\\
\leq\,&\bigg|\sum_{j=1}^k\lambda_j\big(f\big(\psi_i(x^i_j)\big)-f(x_j)\big)\bigg|+2\varepsilon.
\end{split}\]
By letting \(i\to\infty\) and using \eqref{eq:cont_f_xj}, we deduce that
\(\lims_i\big|c_i\int f\,\d(\psi_i)_*\mm_i-\int f\,\d\mm_\infty\big|
\leq 2\varepsilon\), whence by arbitrariness of \(\varepsilon>0\) we obtain
that \(\lim_i c_i\int f\,\d(\psi_i)_*\mm_i=\int f\,\d\mm_\infty\). Moreover,
it follows from \eqref{eq:claim_incl_Ji} that \(\lim_i c_i=1\), so that
\(\lim_i\int f\,\d(\psi_i)_*\mm_i=\int f\,\d\mm_\infty\) for all
\(f\in C_{bbs}(X_\infty)\), thus proving the validity of \eqref{eq:from_pGH_to_pmGH}.
\end{proof}
\section{Ultrafilters and ultraproducts}\label{s:ultrafilters}
The notion of ultralimit of pointed metric measure spaces introduced in this
paper, as well as the well-known notion of ultralimit of pointed metric spaces,
rely on the concepts of ultrafilter and ultralimits in metric spaces. We recall
the necessary definitions and fix some terminology. 
\begin{definition}[Ultrafilter]\label{def:ultrafilter}
Let $\omega\subset 2^\N$ be a collection of subsets of natural numbers.
The set $\omega$ is an \emph{ultrafilter} on $\N$ if the following three
conditions hold:
\begin{enumerate}
\item If $A,B\in\omega$, then $A\cap B\in\omega$.
\item If $A\in\omega$ and $A\subset B\subset\N$, then $B\in\omega$.
\item If $A\subset\N$, then either $A\in\omega$ or $\N\setminus A\in\omega$.
\end{enumerate}
An ultrafilter $\omega$ is called \emph{non-principal} if it does not contain
any singleton.
\end{definition}
\begin{remark}{\rm
For a set $\omega\subset 2^\N$, being an ultrafilter is equivalent to the
characteristic function $\nchi_\omega$ of $\omega$ being a non-trivial,
finitely-additive measure on $2^\N$. Due to this correspondence, we will
use the phrasing ``for $\omega$- almost every'' to mean ``outside of a
$\nchi_\omega$- null set''.
\fr}\end{remark}
\begin{definition}[Ultralimits in metric spaces]
Let $\omega\subset 2^\N$ be a given non-principal ultrafilter.
Let $(X,d)$ be a metric space and $(x_i)$ a sequence of points in $X$.
Then a point $x_\infty\in X$ is said to be the \emph{ultralimit}, or the
\emph{$\omega$-limit}, of the sequence $(x_i)$ if for every $\varepsilon>0$
we have that
\[
\big\{i\in\N\,\big|\,d(x_\infty,x_i)\le\varepsilon\big\}\in\omega.
\]
In this case, we write \(x_\infty=\lim_{i\to\omega}x_i\).
\end{definition}
We will use the following simple argument several times.
\begin{lemma}\label{lma:helppo}
For every $i\in\N$, let $\F_i\in\omega$ be so that
$\F_i\supset\F_{i+1}$ and $\F_i\subset \{k\ge i\}$.
Then for every $k\in\F_i$, there exists $i_k\in\N$ so that
$k\in\F_{i_k}\setminus \F_{i_k+1}$. Moreover, it holds that
\[
\lim_{k\to\omega}i_k=\infty.
\]
\end{lemma}
\begin{proof} The existence of $i_k$ is clearly true. To show that
$\lim_{k\to\omega}i_k=\infty$, it is enough to notice that
\(\{k\in\N\,:\,i_k\ge n\}\supset \F_n\in\omega.\)
\end{proof}
For most of the paper, \(\omega\) will be an arbitrary
non-principal ultrafilter on \(\N\). Only at the end of
Appendix \ref{s:Prokhorov}, we consider ultrafilters
having the property described in the ensuing result:
\begin{lemma}\label{lem:non_p-pt}
Let \(\{N_n\}_{n\in\N}\) be a partition of \(\N\) such
that \(N_n\) is an infinite set for all \(n\in\N\).
Then there exists a non-principal ultrafilter \(\omega\)
on \(\N\) that contains all those subsets \(S\subset\N\)
such that \(N_n\setminus S\) is finite for all but finitely
many \(n\in\N\). In particular, \(\N\setminus N_n\in\omega\)
for every \(n\in\N\).
\end{lemma}
\begin{proof}
Let \(\beta\) be the family of all those subsets \(S\subset\N\)
such that \(N_n\setminus S\) is finite for all but finitely
many \(n\in\N\). We claim that \(\beta\) is a filter on \(\N\),
\emph{i.e.}, it satisfies (1), (2) of Definition
\ref{def:ultrafilter} and does not contain \(\emptyset\).
The latter is trivially verified, as \(N_n\setminus\emptyset\)
is infinite for all \(n\in\N\) and thus \(\emptyset\notin\beta\).
Given any \(S,T\in\beta\), we have that \(N_n\setminus(A\cap B)
=(N_n\setminus A)\cup(N_n\setminus B)\) is finite for all but
finitely many \(n\), so that \(A\cap B\in\beta\). Moreover,
if \(S\in\beta\) and \(T\subset\N\) satisfy \(S\subset T\),
then the set \(N_n\setminus T\subset N_n\setminus S\) is finite
for all but finitely many \(n\), which shows that \(T\in\beta\) as
well. All in all, we have proven that \(\beta\) is a filter on
\(\N\). An application of the Ultrafilter
Lemma yields the existence of an ultrafilter
\(\omega\) on \(\N\) such that \(\beta\subset\omega\).
Since \(N_m\setminus(\N\setminus N_n)=\emptyset\) for
every \(n,m\in\N\) with \(n\neq m\), we see that
\(\{\N\setminus N_n\}_{n\in\N}\subset\beta\subset\omega\).
This implies that \(\omega\) is non-principal, as
\(\bigcap_{S\in\omega}S\subset\bigcap_{n\in\N}\N\setminus N_n
=\emptyset\). Consequently, the statement is achieved.
\end{proof}
\section{Ultralimits of metric spaces}\label{s:ultralimits_metric}
In this section, we will recall the definition of ultralimit of (pointed) metric
spaces. As an intermediate step, we will use the ultraproduct of sets -- or
rather the set-theoretic ultralimit -- in which a pseudometric is defined in a
natural manner. The ultralimit of a sequence of metric spaces is then defined by
identifying those points which are at zero distance from each other. Defining the
ultralimit of pointed metric spaces in such a way is for introducing the 
necessary terminology for the definition of ultralimits of metric measure spaces
in Chapter \ref{sec:ULmms}.
\bigskip

Let $\omega\subset 2^\N$ be a non-principal ultrafilter on \(\N\).
Actually,
\[
\omega\quad\textbf{ will remain fixed throughout the whole paper,}
\]
with the only exception of Proposition \ref{prop:Prokh_sharp}.
Let \(\big((X_i,d_i)\big)\) be a sequence of metric spaces.
Let us define the family \(\Pi^\omega_{i\in\N}X_i\) as
\[
\Pi^\omega_{i\in\N}X_i\coloneqq\bigsqcup_{S\in\omega}\prod_{i\in S}X_i.
\]
Given any \(x\in\Pi^\omega_{i\in\N}X_i\), we will denote by
\(S_x\) the element of \(\omega\) for which \(x\in\prod_{i\in S_x}X_i\).
However, often we will just write \(x=(x_i)\), omitting the explicit
reference to the index set \(S_x\). Consider the following equivalence
relation on \(\Pi^\omega_{i\in\N}X_i\): given
\((x_i),(y_i)\in\Pi^\omega_{i\in\N}X_i\), we declare
\[
(x_i)\sim_\omega(y_i)\quad\Longleftrightarrow\quad
\big\{j\in S_{(x_i)}\cap S_{(y_i)}\,\big|\,x_j=y_j\big\}\in\omega.
\]
Then we define the \emph{ultraproduct} \(\bar X_\omega\) of the sets
\(X_i\) as the quotient set
\[
\bar X_\omega\coloneqq\Pi^\omega_{i\in\N}X_i\big/\sim_\omega.
\]
The equivalence class of an element \((x_i)\in\Pi^\omega_{i\in\N}X_i\)
with respect to \(\sim_\omega\) will be denoted by \([x_i]\in\bar X_\omega\),
while by \(\pi_\omega\colon\Pi^\omega_{i\in\N}X_i\to\bar X_\omega\)
we mean the canonical projection map \(\pi_\omega\big((x_i)\big)\coloneqq[x_i]\).
\begin{definition}[Set-theoretic ultralimit of metric spaces]
\label{def:setultralim}
Let $\big((X_i,d_i)\big)$ be a sequence of metric spaces.
Then we define its \emph{set-theoretic ultralimit} as
\[
\Pi_{i\to\omega}(X_i,d_i)\coloneqq(\bar X_\omega,\bar d_\omega),
\]
where $\bar d_\omega$ is the extended pseudodistance on the ultraproduct
\(\bar X_\omega\), which is defined by
\begin{equation}\label{eq:def_bar_d_omega}
\bar d_\omega\big([x_i],[y_i]\big)\coloneqq\lim_{i\to\omega}
d_i(x_i,y_i)\in[0,+\infty],\quad\text{ for every }[x_i],[y_i]\in\bar X_\omega.
\end{equation}
\end{definition}
The definition in \eqref{eq:def_bar_d_omega} is well-posed, as the quantity
\(\lim_{i\to\omega}d_i(x_i,y_i)\) does not depend on the specific choice of
the representatives of \([x_i]\) and \([y_i]\).
\begin{definition}[Set-theoretic ultralimit of pointed metric spaces]
Let $\big((X_i,d_i,p_i)\big)$ be a sequence of pointed metric spaces.
Then we define its \emph{set-theoretic ultralimit} as
\[
\Pi_{i\to\omega}(X_i,d_i,p_i)\coloneqq
\big(O(\bar X_\omega),\bar d_\omega,[p_i]\big),
\]
where we set
\[
O(\bar X_\omega)\coloneqq\Big\{[x_i]\in\bar X_\omega\,\Big|
\,\bar d_\omega\big([x_i],[p_i]\big)<+\infty\Big\}.
\]
\end{definition}
The set-theoretic ultralimit of pointed metric spaces is a (pointed)
pseudometric space. Now the ultralimit can be defined as the natural
corresponding metric space.
\begin{definition}[Ultralimit of pointed metric spaces]
Let \(\big((X_i,d_i,p_i)\big)\) be a sequence of pointed metric spaces.
Let us consider the following equivalence relation on \(O(\bar X_\omega)\):
given any \([x_i],[y_i]\in O(\bar X_\omega)\), we declare that
\([x_i]\sim[y_i]\) if and only if \(\bar d_\omega\big([x_i],[y_i]\big)=0\).
We denote by \([[x_i]]\) the equivalence class of \([x_i]\).
Then we define the \emph{ultralimit} of \(\big((X_i,d_i,p_i)\big)\) as
\[
\lim_{i\to\omega}(X_i,d_i,p_i)\coloneqq(X_\omega,d_\omega,p_\omega),
\]
where we set \(X_\omega\coloneqq O(\bar X_\omega)/\sim\),
\[
d_\omega\big([[x_i]],[[y_i]]\big)\coloneqq\bar d_\omega\big([x_i],[y_i]\big),
\quad\text{ for every }[[x_i]],[[y_i]]\in X_\omega,
\]
and \(p_\omega\coloneqq[[p_i]]\).
\end{definition}
Observe that the value of \(d_\omega\big([[x_i]],[[y_i]]\big)\)
does not depend on the chosen representatives of \([[x_i]]\) and \([[y_i]]\).
The ultralimit \(\lim_{i\to\omega}(X_i,d_i,p_i)\) is a pointed metric space.
We denote by
\begin{equation}\label{eq:def_proj_pi}
\pi\big([x_i]\big)\coloneqq[[x_i]],\quad\text{ for every }
[x_i]\in O(\bar X_\omega),
\end{equation}
the natural projection map \(\pi\colon O(\bar X_\omega)\to X_\omega\)
on the quotient.
\chapter{Ultralimits of metric measure spaces}\label{sec:ULmms}
In this chapter, we will introduce the notion of ultralimit of pointed
metric measure spaces,  which naturally arises from the related notions.
Yet, we are not aware of any usage of it in the literature.
\section{Construction of the ultralimit}\label{s:constr_ultralim_mms}
We begin by introducing a relevant family of subsets of the
set-theoretic ultralimit of metric spaces, which will play a key
role in the sequel.
\begin{definition}[Internal measurable set]\label{def:int_meas_sets}
Let \(\big((X_i,d_i)\big)\) be a sequence of metric spaces.
Then a set \(\bar A\subset\bar X_\omega\) is said
to be an \emph{internal measurable set} provided there exist \(S\in\omega\)
and \((A_i)_{i\in S}\in\prod_{i\in S}{\rm Ball}(X_i)\) with
\(\bar A=\pi_\omega\big(\prod_{i\in S}A_i\big)\). The family of all
internal measurable subsets of \(\bar X_\omega\) is denoted by
\(\mathcal A(\bar X_\omega)\).
\end{definition}
Notice that \(\bar A\) coincides with the ultraproduct of the sets \(A_i\),
which can be seen as a subset of \(\bar X_\omega\) in a canonical way. In view
of this observation, we will denote it by \(\bar A=\Pi_{i\to\omega}A_i\).
\begin{lemma}\label{lem:equiv_iBs}
Let \(\big((X_i,d_i)\big)\) be a sequence of metric spaces.
Let \((A_i)_{i\in S}\in\prod_{i\in S}{\rm Ball}(X_i)\)
and \((B_i)_{i\in S'}\in\prod_{i\in S'}{\rm Ball}(X_i)\) be given,
for some \(S,S'\in\omega\). Then it holds that
\[
\Pi_{i\to\omega}A_i=\Pi_{i\to\omega}B_i\quad\Longleftrightarrow
\quad A_i=B_i\;\text{ for }\omega\text{-a.e.\ }i\in S\cap S'.
\]
\end{lemma}
\begin{proof}
Suppose \(A_i=B_i\) holds for \(\omega\)-a.e.\ \(i\in S\cap S'\).
Set \(T\coloneqq\{i\in S\cap S'\,:\,A_i=B_i\}\in\omega\). If
\([x_i]\in\Pi_{i\to\omega}A_i\), then
\(T\cap\{j\in S\cap S_{(x_i)}\,:\,x_j\in A_j\}\in\omega\), so
its superset \(\{j\in S'\cap S_{(x_i)}\,:\,x_j\in B_j\}\) belongs to \(\omega\)
as well, which shows that \([x_i]\in\Pi_{i\to\omega}B_i\) and
thus \(\Pi_{i\to\omega}A_i\subset\Pi_{i\to\omega}B_i\).
The converse inclusion follows by an analogous argument,
just interchanging \((A_i)_{i\in S}\) and \((B_i)_{i\in S'}\).

On the contrary, suppose \(\{i\in S\cap S'\,:\,A_i\neq B_i\}\in\omega\).
Possibly interchanging \((A_i)_{i\in S}\) and \((B_i)_{i\in S'}\), we
can suppose that \(T\coloneqq\{i\in S\cap S'\,:\,A_i\setminus B_i\neq
\emptyset\}\in\omega\). Pick any \(x_i\in A_i\setminus B_i\) for
every \(i\in T\). Then it holds \([x_i]\in\Pi_{i\to\omega}A_i
\setminus\Pi_{i\to\omega}B_i\), thus proving that
\(\Pi_{i\to\omega}A_i\neq\Pi_{i\to\omega}B_i\).
\end{proof}
The idea is to define the limit measure on a suitable algebra 
\(\mathcal A_O(\bar X_\omega)\) (obtained by restriction from
\(\mathcal A(\bar X_\omega)\)) of subsets of \(O(\bar X_\omega)\)
using the measure on internal measurable subsets defined in a natural way,
and then extend it to a \(\sigma\)-algebra containing them. Finally, the
measure on the ultralimit of pointed metric measure spaces is obtained by
pushforward under the projection map \(\pi\).
We will use the following proposition.
\begin{proposition}\label{prop:formulas_union_complement}
Let \(\big((X_i,d_i)\big)\) be a sequence of metric spaces.
Then it holds that
\begin{subequations}\begin{align}
\label{eq:A_X_algebra_claim1}
(\Pi_{i\to\omega}A_i)\cup(\Pi_{i\to\omega}B_i)
&=\Pi_{i\to\omega}A_i\cup B_i,\\
\label{eq:A_X_algebra_claim2}
(\Pi_{i\to\omega}A_i)^c&=\Pi_{i\to\omega}A_i^c,
\end{align}\end{subequations}
for every \(\Pi_{i\to\omega}A_i,\Pi_{i\to\omega}B_i\in
\mathcal A(\bar X_\omega)\). In particular,
the family \(\mathcal A(\bar X_\omega)\) is an algebra of
subsets of \(\bar X_\omega\). Moreover, if \(\mm_i\) is a given ball
measure on the space \(X_i\) for every \(i\in\N\), then the set-function
\(\tilde\mm_\omega\colon\mathcal A(\bar X_\omega)\to[0,+\infty]\) defined by
\[
\tilde\mm_\omega(\Pi_{i\to\omega}A_i)\coloneqq\lim_{i\to\omega}
\mm_i(A_i),\quad\text{ for every }\Pi_{i\to\omega}A_i
\in\mathcal A(\bar X_\omega)
\]
is a finitely-additive measure on \(\mathcal A(\bar X_\omega)\). The set-function \(\tilde\mm_\omega\) is well-defined due to
Lemma \ref{lem:equiv_iBs}.
\end{proposition}
\begin{proof}
In order to prove \eqref{eq:A_X_algebra_claim1}, first observe that,
trivially, \(\Pi_{i\to\omega}A_i,\Pi_{i\to\omega}B_i\subset
\Pi_{i\to\omega}A_i\cup B_i\). Conversely, let
\([x_i]\in\Pi_{i\to\omega}A_i\cup B_i\) be fixed. Pick any
\(S\subset S_{(x_i)}\) such that \(x_j\in A_j\) for all \(j\in S\)
and \(x_j\in B_j\) for all \(j\in S_{(x_i)}\setminus S\). Then either
\(S\in\omega\) and so \([x_i]\in\Pi_{i\to\omega}A_i\), or
\(S_{(x_i)}\setminus S\in\omega\) and so \([x_i]\in\Pi_{i\to\omega}B_i\).
We have proven \(\Pi_{i\to\omega}A_i\cup B_i
\subset(\Pi_{i\to\omega}A_i)\cup(\Pi_{i\to\omega}B_i)\),
yielding \eqref{eq:A_X_algebra_claim1}.

In order to prove \eqref{eq:A_X_algebra_claim2}, observe that
for any given \([x_i]\in\bar X_\omega\) we have that
\[\begin{split}
[x_i]\in\Pi_{i\to\omega}A_i\quad&\Longleftrightarrow\quad
\{j\in S_{(x_i)}\,:\,x_j\in A_j\}\in\omega\quad\Longleftrightarrow
\quad\{j\in S_{(x_i)}\,:\,x_j\in A_j^c\}\notin\omega\\
&\Longleftrightarrow\quad[x_i]\in(\Pi_{i\to\omega}A_i^c)^c,
\end{split}\]
proving \eqref{eq:A_X_algebra_claim2}. Given that also
\(\emptyset=\Pi_{i\to\omega}\emptyset\in\mathcal A(\bar X_\omega)\),
we thus see that \(\mathcal A(\bar X_\omega)\) is an algebra.

Let us now consider the set-function \(\tilde\mm_\omega\).
Fix pairwise disjoint sets \(\bar A^1,\ldots,\bar A^n\in
\mathcal A(\bar X_\omega)\), say \(\bar A^j=\Pi_{i\to\omega}A^j_i\)
for all \(j=1,\ldots,n\). Given \(j,j'=1,\ldots,n\) with \(j\neq j'\),
we infer from \eqref{eq:A_X_algebra_claim1} and
\eqref{eq:A_X_algebra_claim2} that
\(\Pi_{i\to\omega}A^j_i\cap A^{j'}_i=\bar A^j\cap\bar A^{j'}
=\emptyset\), whence Lemma \ref{lem:equiv_iBs} grants that
\(A^j_i\cap A^{j'}_i=\emptyset\) for \(\omega\)-a.e.\ \(i\).
Therefore, we conclude that
\[\begin{split}
\tilde\mm_\omega(\bar A^1\cup\dots\cup\bar A^n)&=
\tilde\mm_\omega\big(\Pi_{i\to\omega}A^1_i\cup\dots\cup A^n_i\big)=
\lim_{i\to\omega}\mm_i(A^1_i\cup\dots\cup A^n_i)=
\lim_{i\to\omega}\sum_{j=1}^n\mm_i(A^j_i)\\
&=\sum_{j=1}^n\lim_{i\to\omega}\mm_i(A^j_i)
=\sum_{j=1}^n\tilde\mm_\omega(\bar A^j),
\end{split}\]
thus proving that \(\tilde\mm_\omega\) is finitely-additive.
Consequently, the proof is achieved.
\end{proof}
\begin{definition}[The algebra \(\mathcal A_O(\bar X_\omega)\)]
Let \(\big((X_i,d_i,p_i)\big)\) be a sequence of pointed metric measure spaces.
Then we define the algebra \(\mathcal A_O(\bar X_\omega)\) as
\[
\mathcal A_O(\bar X_\omega)\coloneqq\big\{\bar A\cap O(\bar X_\omega)\;\big|
\;\bar A\in\mathcal A(\bar X_\omega)\big\}.
\]
\end{definition}
In general, the set \(O(\bar X_\omega)\) is not an internal measurable
subset of \(\bar X_\omega\). However, it is in any \(\sigma\)-algebra containing
\(\mathcal A(\bar X_\omega)\) due to the following proposition, see also
\cite[Lemma 3.1]{Elek}.
\begin{proposition}\label{prop:formula_balls}
Let \(\big((X_i,d_i,p_i)\big)\) be a sequence of pointed metric spaces.
Then it holds that
\begin{equation}\label{eq:formula_balls}
\bar B\big([x_i],R\big)=\bigcap_{n\in\N}\Pi_{i\to\omega}B(x_i,R+1/n),
\quad\text{ for every }[x_i]\in\bar X_\omega\text{ and }R>0.
\end{equation}
In particular, \(B\in\sigma\big(\mathcal A(\bar X_\omega)\big)\)
for every closed ball \(B\subset\bar X_\omega\), thus
\(O(\bar X_\omega)\in\sigma\big(\mathcal A(\bar X_\omega)\big)\)
as a countable union of closed balls. Moreover, for any set
\(\Pi_{i\to\omega}A_i\in\mathcal A(\bar X_\omega)\) we have that
\begin{equation}\label{eq:formula_intersect_O}
(\Pi_{i\to\omega}A_i)\cap O(\bar X_\omega)=\bigcup_{R\in\N}\Pi_{i\to\omega}
\big(A_i\cap B(p_i,R)\big).
\end{equation}
\end{proposition}
\begin{proof}
First, let us prove \eqref{eq:formula_balls}. If \([y_i]\in\bar B\big([x_i],R\big)\)
and \(n\in\N\), then \(\lim_{i\to\omega}d_i(x_i,y_i)<R+1/n\), which implies that
\[
\big\{j\in S_{(x_i)}\cap S_{(y_i)}\,\big|\,y_j\in B(x_j,R+1/n)\big\}=
\big\{j\in S_{(x_i)}\cap S_{(y_i)}\,\big|\,d_j(x_j,y_j)<R+1/n\big\}\in\omega.
\]
Therefore, by definition of internal measurable set, we have that
\([y_i]\in\Pi_{i\to\omega}B(x_i,R+1/n)\), thus
\(\bar B\big([x_i],R\big)\subset\bigcap_{n\in\N}\Pi_{i\to\omega}B(x_i,R+1/n)\).
Conversely, fix \([y_i]\in\bigcap_{n\in\N}\Pi_{i\to\omega}B(x_i,R+1/n)\).
Given any \(\varepsilon>0\), pick some \(n\in\N\) satisfying \(n>1/\varepsilon\).
Then we have that
\[\begin{split}
\omega\ni\big\{j\in S_{(x_i)}\cap S_{(y_i)}\,\big|\,y_j\in B(x_j,R+1/n)\big\}
&=\big\{j\in S_{(x_i)}\cap S_{(y_i)}\,\big|\,d_j(x_j,y_j)<R+1/n\big\}\\
&\subset\big\{j\in S_{(x_i)}\cap S_{(y_i)}\,\big|\,d_j(x_j,y_j)<R+\varepsilon\big\},
\end{split}\]
which yields \(\big\{j\in S_{(x_i)}\cap S_{(y_i)}\,\big|\,d_j(x_j,y_j)<R+\varepsilon\big\}
\in\omega\). This grants that \(\bar d_\omega\big([x_i],[y_i]\big)\leq R\) by arbitrariness
of \(\varepsilon>0\), so that \([y_i]\in\bar B\big([x_i],R\big)\) and accordingly the
identity in \eqref{eq:formula_balls} is proven.

Let us now show \eqref{eq:formula_intersect_O}. Fix any \([x_i]\in(\Pi_{i\to\omega}A_i)
\cap O(\bar X_\omega)\). Given that \([x_i]\in O(\bar X_\omega)\), there exists \(R'\in\N\)
such that \(\big\{j\in S_{(x_i)}\,:\,d_j(x_j,p_j)<R'\big\}\in\omega\), so that
\([x_i]\in\Pi_{i\to\omega}B(p_i,R')\). This ensures that
\([x_i]\in\Pi_{i\to\omega}\big(A_i\cap B(p_i,R')\big)\subset\bigcup_{R\in\N}
\Pi_{i\to\omega}\big(A_i\cap B(p_i,R)\big)\) by Proposition
\ref{prop:formulas_union_complement}. Conversely, fix
\([x_i]\in\bigcup_{R\in\N}\Pi_{i\to\omega}\big(A_i\cap B(p_i,R)\big)\).
Then there are \(S\in\omega\) and \(R\in\N\) such that
\[
\{i\in S\,|\,x_i\in A_i\}\cap\big\{i\in S\,\big|\,d_i(x_i,p_i)<R\big\}
=\big\{i\in S\,\big|\,x_i\in A_i\cap B(p_i,R)\big\}\in\omega.
\]
By exploiting the upward-closedness of the ultrafilter \(\omega\),
we infer that \(\{i\in S\,|\,x_i\in A_i\}\in\omega\) and
\(\big\{i\in S\,\big|\,d_i(x_i,p_i)<R\big\}\in\omega\), which implies that
\([x_i]\in\Pi_{i\to\omega}A_i\) and \(\bar d_\omega\big([x_i],[p_i]\big)\leq R\),
respectively. In particular, it holds
\([x_i]\in(\Pi_{i\to\omega}A_i)\cap O(\bar X_\omega)\). All in all,
\eqref{eq:formula_intersect_O} is proven.
\end{proof}
We may now give the definition of the finitely-additive measure \(\bar\mm_\omega\)
on the algebra \(\mathcal A_O(\bar X_\omega)\).
\begin{definition}[The measure \(\bar\mm_\omega\)]\label{def:bar_m_omega}
Let \(\big((X_i,d_i,\mm_i,p_i)\big)\) be a sequence of pointed metric measure spaces.
Then we define the set-function \(\bar\mm_\omega\colon\mathcal A_O(\bar X_\omega)
\to[0,+\infty]\) as
\[
\bar\mm_\omega\big((\Pi_{i\to\omega}A_i)\cap O(\bar X_\omega)\big)\coloneqq
\lim_{R\to\infty}\tilde\mm_\omega\big(\Pi_{i\to\omega}A_i\cap B(p_i,R)\big)
=\lim_{R\to\infty}\lim_{i\to\omega}\mm_i\big(A_i\cap B(p_i,R)\big)
\]
for every \(\Pi_{i\to\omega}A_i\in\mathcal A(\bar X_\omega)\). The set-function \(\bar\mm_\omega\) is well-defined due to Lemma \ref{lem:equiv_iBs} and
\eqref{eq:formula_intersect_O}.
\end{definition}
In order to show that \(\bar\mm_\omega\) is consistent with \(\tilde\mm_\omega\),
we will need the following simple result.
\begin{lemma}\label{lem:A_cap_A_O_bdd}
Let \(\big((X_i,d_i,p_i)\big)\) be a sequence of pointed metric spaces.
Let \(\bar A=\Pi_{i\to\omega}A_i\) be a given element of
\(\mathcal A(\bar X_\omega)\cap\mathcal A_O(\bar X_\omega)\).
Then the set \(\bar A\subset O(\bar X_\omega)\) is bounded with respect
to \(\bar d_\omega\).
\end{lemma}
\begin{proof}
We argue by contradiction: suppose that for every \(n\in\N\) there exists
\([x^n_i]\in\bar A\) such that \(\lim_{i\to\omega}d_i(x^n_i,p_i)=
\bar d_\omega\big([x^n_i],[p_i]\big)>n\). Define
\(S^1\coloneqq\big\{j\in S_{(x^1_i)}\,:\,d_j(x^1_j,p_j)>1\big\}\) and
\[
S^n\coloneqq\big\{j\in S^{n-1}\cap S_{(x^n_i)}\,\big|\,
d_j(x^n_j,p_j)>n\big\},\quad\text{ for every }n\geq 2.
\]
Observe that \((S^n)_{n\in\N}\subset\omega\) by construction. Also, it holds
that \(S^n\subset S^{n-1}\) for every \(n\geq 2\). Let us now set
\(y_i\coloneqq x^n_i\) for every \(n\in\N\) and \(i\in S^n\setminus S^{n+1}\).
Notice that we have \((y_i)_{i\in S^1}\in\prod_{i\in S^1}A_i\). Fix any \(n\in\N\).
Given any \(i\in S^n\), there exists a unique \(n_i\geq n\) such that
\(i\in S^{n_i}\setminus S^{n_i+1}\), thus accordingly
\(d_i(y_i,p_i)=d_i(x^{n_i}_i,p_i)>n_i\geq n\). This shows that
\(\big\{i\in S^1\,:\,d_i(y_i,p_i)>n\big\}\in\omega\) for all \(n\in\N\), in other
words \(\bar d_\omega\big([y_i],[p_i]\big)=\lim_{i\to\omega}d_i(y_i,p_i)=+\infty\).
This is in contradiction with the fact that \([y_i]\in\bar A\subset O(\bar X_\omega)\),
thus the statement is achieved.
\end{proof}
Thanks to the above lemma, we can now prove that \(\bar\mm_\omega\) and
\(\tilde\mm_\omega\) agree on \(\mathcal A(\bar X_\omega)\cap\mathcal A_O(\bar X_\omega)\).
\begin{proposition}
Let \(\big((X_i,d_i,\mm_i,p_i)\big)\) be a sequence of pointed metric measure spaces.
Then the set-function \(\bar\mm_\omega\) is a finitely-additive measure on
\(\mathcal A_O(\bar X_\omega)\). Moreover, it holds that
\begin{equation}\label{eq:consist_tilde_bar_m_omega}
\bar\mm_\omega(\bar A)=\tilde\mm_\omega(\bar A),\quad\text{ for every }
\bar A\in\mathcal A(\bar X_\omega)\cap\mathcal A_O(\bar X_\omega).
\end{equation}
\end{proposition}
\begin{proof}
Let \(\bar A,\bar B\in\mathcal A(\bar X_\omega)\) be such that
\(\big(\bar A\cap O(\bar X_\omega)\big)\cap\big(\bar B\cap O(\bar X_\omega)\big)
=\emptyset\), say that \(\bar A=\Pi_{i\to\omega}A_i\) and \(\bar B=\Pi_{i\to\omega}B_i\).
In particular, \(\big(\Pi_{i\to\omega}A_i\cap B(p_i,R)\big)\cap
\big(\Pi_{i\to\omega}B_i\cap B(p_i,R)\big)=\emptyset\) for all \(R>0\).
By finite additivity of the measure \(\tilde\mm_\omega\) on \(\mathcal A(\bar X_\omega)\),
we get that
\[\begin{split}
&\bar\mm_\omega\big(\big(\bar A\cap O(\bar X_\omega)\big)\cup
\big(\bar B\cap O(\bar X_\omega)\big)\big)\\
=\,&\bar\mm_\omega\big((\Pi_{i\to\omega}A_i\cup B_i)\cap O(\bar X_\omega)\big)
=\lim_{R\to\infty}\tilde\mm_\omega\big(\Pi_{i\to\omega}(A_i\cup B_i)\cap B(p_i,R)\big)\\
=\,&\lim_{R\to\infty}\tilde\mm_\omega\Big(\big(\Pi_{i\to\omega}A_i\cap B(p_i,R)\big)
\cup\big(\Pi_{i\to\omega}B_i\cap B(p_i,R)\big)\Big)\\
=\,&\lim_{R\to\infty}\Big[\tilde\mm_\omega\big(\Pi_{i\to\omega}A_i\cap B(p_i,R)\big)+
\tilde\mm_\omega\big(\Pi_{i\to\omega}B_i\cap B(p_i,R)\big)\Big]\\
=\,&\lim_{R\to\infty}\tilde\mm_\omega\big(\Pi_{i\to\omega}A_i\cap B(p_i,R)\big)+
\lim_{R\to\infty}\tilde\mm_\omega\big(\Pi_{i\to\omega}B_i\cap B(p_i,R)\big)\\
=\,&\bar\mm_\omega\big(\bar A\cap O(\bar X_\omega)\big)+
\bar\mm_\omega\big(\bar B\cap O(\bar X_\omega)\big),
\end{split}\]
whence it follows that \(\bar\mm_\omega\) is a finitely-additive measure.
To prove \eqref{eq:consist_tilde_bar_m_omega}, let \(\bar A=\Pi_{i\to\omega}A_i\)
be a given element of \(\mathcal A(\bar X_\omega)\cap\mathcal A_O(\bar X_\omega)\).
Lemma \ref{lem:A_cap_A_O_bdd} grants that \(\bar A\) is \(\bar d_\omega\)-bounded,
thus (recalling \eqref{eq:formula_balls}) we can find \(R>0\) such that
\(\Pi_{i\to\omega}A_i\subset\Pi_{i\to\omega}B(p_i,R)\). By using Lemma
\ref{lem:equiv_iBs} and Proposition \ref{prop:formulas_union_complement},
we deduce that there exists \(S\in\omega\) such that \(A_i\subset B(p_i,R)\)
for every \(i\in S\). Therefore, for any given \(R'>R\) we have that
\(\mm_i(A_i)=\mm_i\big(A_i\cap B(p_i,R')\big)\) for all \(i\in S\),
which yields \(\tilde\mm_\omega\big(\Pi_{i\to\omega}A_i\cap B(p_i,R')\big)
=\tilde\mm_\omega(\bar A)\). Then we can conclude that
\[
\bar\mm_\omega(\bar A)=
\lim_{R'\to\infty}\tilde\mm_\omega\big(\Pi_{i\to\omega}A_i\cap B(p_i,R')\big)=
\lim_{R'\to\infty}\tilde\mm_\omega(\bar A)=\tilde\mm_\omega(\bar A),
\]
thus obtaining \eqref{eq:consist_tilde_bar_m_omega} and accordingly the statement.
\end{proof}
Hereafter, we shall work with sequences of pointed metric measure spaces
satisfying the following growth condition.
\begin{definition}[\(\omega\)-uniform bounded finiteness]\label{def:omega-ubf}
Let \(\big((X_i,d_i,\mm_i,p_i)\big)\) be a sequence of pointed metric measure
spaces. Then we say that \(\big((X_i,d_i,\mm_i,p_i)\big)\) is a
\emph{\(\omega\)-uniformly boundedly finite} sequence provided there exists
\(\eta\colon(0,+\infty)\to(0,+\infty)\) such that for any \(R>0\) it holds that
\[
\mm_i\big(B(p_i,R)\big)\leq\eta(R),\quad
\text{ for }\omega\text{-a.e.\ }i.
\]
\end{definition}
In the framework of \(\omega\)-uniformly boundedly finite sequences of spaces,
we will extend the measure \(\bar\mm_\omega\) to a \(\sigma\)-algebra
\(\bar{\mathcal B}_\omega\) on \(O(\bar X_\omega)\) containing the algebra
\(\mathcal A_O(\bar X_\omega)\). The \(\sigma\)-algebra \(\bar{\mathcal B}_\omega\)
that we will consider is the one obtained by adding null sets to the elements of
\(\mathcal A_O(\bar X_\omega)\).
\begin{definition}[Null sets and \(\bar{\mathcal B}_\omega\)]
Let \(\big((X_i,d_i,\mm_i,p_i)\big)\) be a sequence of pointed metric measure spaces.
Let \(N\subset O(\bar X_\omega)\) be given. Then we declare that \(N\)
is a \emph{null set} provided for any \(\varepsilon>0\) there exists
\(\bar A\in\mathcal A_O(\bar X_\omega)\) such that \(N\subset\bar A\) and
\(\bar\mm_\omega(\bar A)<\varepsilon\). The family of all null sets is
denoted by \(\mathcal N_\omega\). Moreover, we denote by \(\bar{\mathcal B}_\omega\)
the family of all subsets \(A\subset O(\bar X_\omega)\) for which the following
property holds: given any \(R>0\), there exists \(\bar A_R\in\mathcal A_O(\bar X_\omega)\)
such that
\[
\big(A\cap\Pi_{i\to\omega}B(p_i,R)\big)\Delta\bar A_R\in\mathcal N_\omega.
\]
\end{definition}
\begin{definition}[The extension of \(\bar\mm_\omega\)]
Let \(\big((X_i,d_i,\mm_i,p_i)\big)\) be a sequence of pointed metric measure spaces.
Then we define \(\bar\mm_\omega\colon\bar{\mathcal B}_\omega\to[0,+\infty]\) as
follows: given any \(A\in\bar{\mathcal B}_\omega\), we set
\begin{equation}\label{eq:def_bar_m_omega_ext}
\bar\mm_\omega(A)\coloneqq\lim_{R\to\infty}\bar\mm_\omega(\bar A_R),
\end{equation}
where the sets \(\bar A_R\in\mathcal A_O(\bar X_\omega)\) are chosen so that
\(\big(A\cap\Pi_{i\to\omega}B(p_i,R)\big)\Delta\bar A_R\in\mathcal N_\omega\).
\end{definition}
It can be readily checked that \(\bar\mm_\omega\) is well-defined (\emph{i.e.}, the limit in \eqref{eq:def_bar_m_omega_ext} does not depend on the choice of the sets \(\bar A_R\))
and that it is consistent with Definition \ref{def:bar_m_omega}.
Next we will show, under the \(\omega\)-uniform bounded finiteness assumption,
that \(\bar{\mathcal B}_\omega\) is a \(\sigma\)-algebra and that \(\bar\mm_\omega\)
is a measure on \(\bar{\mathcal B}_\omega\). We will need the following lemma,
see \cite[Lemma 2.3]{Conley-Kechris-Tucker-Drob}.
\begin{lemma}\label{lem:key_lemma}
Let \(\big((X_i,d_i,\mm_i,p_i)\big)\) be a \(\omega\)-uniformly boundedly
finite sequence of pointed metric measure spaces. Let \((\bar A^n)_{n\in\N}
\subset\mathcal A_O(\bar X_\omega)\) be given. Then it holds that
\(\bigcup_{n\in\N}\bar A^n\in\bar{\mathcal B}_\omega\).
\end{lemma}
\begin{proof}
Let \(R>0\) be fixed. Then for every \(n\in\N\) we may write
\(\bar A^n=(\Pi_{i\to\omega}A^n_i)\cap O(\bar X_\omega)\), where
\(A^n_i\in{\rm Ball}(X_i)\) for all \(i\in S^n\) (for some \(S^n\in\omega\)). Notice that
\[
\bigcup_{n\in\N}\bar A^n\cap\big(\Pi_{i\to\omega}B(p_i,R)\big)=
\bigcup_{n\in\N}(\Pi_{i\to\omega}B^n_i)\cap B(p_i,R),\quad
\text{ where we set }B^n_i\coloneqq\bigcup_{k\leq n}A^k_i.
\]
Therefore, we may assume that \(\bar A^n=\Pi_{i\to\omega}A^n_i\)
are internal measurable subsets of \(O(\bar X_\omega)\) for which
\(A^n_i\subset A^m_i\subset B(p_i,R)\) whenever \(n\leq m\).
Since \(\mm_i\big(B(p_i,R)\big) \leq\eta(R)\) for \(\omega\)-a.e.\ \(i\),
we may thus also assume that \(\bar\mm_\omega(\bar A^n)
=\lim_{i\to\omega}\mm_i(A^n_i)\leq\eta(R)<+\infty\).

The rest of the proof is essentially the same as the proof of
\cite[Lemma 2.3]{Conley-Kechris-Tucker-Drob}.
We will present it here for completeness.
The procedure is a sort of diagonal argument. Denote
\(M_n\coloneqq\bar\mm_\omega(\bar A^n)\) for all \(n\in\N\).
Define the neighbourhood \(U_n\) of \(M_n\) as \(U_n\coloneqq[M_n-2^{-n},M_n+2^{-n}]\).
Define the sets \((T^n)_{n\in\N}\) recursively as follows:
\(T^1\coloneqq\big\{i\in S^1\,:\,\mm_i(A^1_i)\in U_1\big\}\) and
\[
T^n\coloneqq\big\{i\in T^{n-1}\cap S^n\,\big|\,i\geq n,\,
\mm_i(A^n_i)\in U_n\big\},\quad\text{ for every }n\geq 2.
\]
By definition of \(M_n\) (and by induction), we have that \(T^n\in\omega\)
for every \(n\in\N\). The sets \(T^n\) are also nested by definition, \emph{i.e.},
\(T^n\subset T^m\) whenever \(m\leq n\). For every index \(i\in T^1\), there exists
\(n_i\in\N\) so that \(i\in T^{n_i}\setminus T^{n_i+1}\). Define now
\(C_i\coloneqq A^{n_i}_i\) for every \(i\in T^1\). Moreover, let us set
\(\bar C\coloneqq\Pi_{i\to\omega}C_i\in\mathcal A(\bar X_\omega)\cap
\mathcal A_O(\bar X_\omega)\). We claim that
\begin{equation}\label{eq:key_lemma_aux}
\bigcup_{n\in\N}\bar A^n\subset\bar C\quad\text{ and }\quad
\bar C\setminus\bigcup_{n\in\N}\bar A^n\in\mathcal N_\omega,
\end{equation}
whence the statement would immediately follow, thanks to the arbitrariness of \(R>0\).
To prove the first part of \eqref{eq:key_lemma_aux}, let \(n\in\N\) be fixed. Notice
that if \(i\in T^n\), then \(n_i\geq n\); in turn, this implies that \(A^n_i\subset C_i\).
In other words, \(T^n=\{i\in T^n\,:\,A^n_i\subset C_i\}\in\omega\). This guarantees
that \(\bar A^n=\Pi_{i\to\omega}A^n_i\subset\Pi_{i\to\omega}C_i=\bar C\), thus
accordingly \(\bigcup_{n\in\N}\bar A^n\subset\bar C\). To prove the second part
of \eqref{eq:key_lemma_aux}, observe that \(\bar A^n=\Pi_{i\to\omega}A^n_i
\subset\Pi_{i\to\omega}A^m_i=\bar A^m\) whenever \(n\leq m\), thus \((M_n)_n
\subset\big[0,\eta(R)\big]\) is a non-decreasing sequence, which admits a limit
\(M_\infty\coloneqq\lim_{n\to\infty}M_n\in\big[0,\eta(R)\big]\). Let \(\varepsilon>0\)
be fixed. Choose any \(N\in\N\) such that \(2^{-N}<\varepsilon\) and
\(|M_n-M_\infty|<\varepsilon\) for every \(n\geq N\). Since
\[\begin{split}
T^N&=\big\{i\in T^N\,\big|\,n_i\geq N\big\}=
\big\{i\in T^N\,\big|\,|\mm_i(C_i)-M_{n_i}|<\varepsilon\big\}\cap
\big\{i\in T^N\,\big|\,|M_{n_i}-M_\infty|<\varepsilon\big\}\\
&=\big\{i\in T^N\,\big|\,|\mm_i(C_i)-M_\infty|<2\varepsilon\big\}\in\omega,
\end{split}\]
we get that \(\bar\mm_\omega(\bar C)=\lim_{i\to\omega}\mm_i(C_i)=M_\infty\).
Hence, we have that
\[
\bar\mm_\omega\big(\bar C\setminus{\textstyle\bigcup_{k\leq n}\bar A^k}\big)
=\bar\mm_\omega(\bar C\setminus\bar A^n)=\bar\mm_\omega(\bar C)-\bar\mm_\omega(\bar A^n)
=M_\infty-M_n\longrightarrow 0,\quad\text{ as }n\to\infty.
\]
Given that \(\bar C\setminus\bigcup_{k\in\N}\bar A^k\subset\bar C\setminus
\bigcup_{k\leq n}\bar A^k\) for all \(n\in\N\), we conclude that
\(\bar C\setminus\bigcup_{n\in\N}\bar A^n\in\mathcal N_\omega\).
\end{proof}
\begin{theorem}
Let \(\big((X_i,d_i,\mm_i,p_i)\big)\) be a \(\omega\)-uniformly boundedly
finite sequence of pointed metric measure spaces. Then the family
\(\bar{\mathcal B}_\omega\) is a \(\sigma\)-algebra and the set-function
\(\bar\mm_\omega\) is a (countably-additive) measure on \(\bar{\mathcal B}_\omega\).
\end{theorem}
\begin{proof}
Since \(\emptyset\in\mathcal N_\omega\cap\mathcal A_O(\bar X_\omega)\), we have that
\(\emptyset\in\bar{\mathcal B}_\omega\). To check that \(\bar{\mathcal B}_\omega\) is
a \(\sigma\)-algebra, notice that if \(A,B\in\bar{\mathcal B}_\omega\), then for any given
\(R>0\) we can find \(\bar A_R,\bar B_R\in\mathcal A_O(\bar X_\omega)\) such that
\[
A'\Delta\bar A_R,B'\Delta\bar B_R\in\mathcal N_\omega,\quad
\text{ where }A'\coloneqq A\cap\Pi_{i\to\omega}B(p_i,R)\text{ and }
B'\coloneqq B\cap\Pi_{i\to\omega}B(p_i,R).
\]
Therefore, we deduce that \((\bar A_R\setminus\bar B_R)\Delta(A'\setminus B')
\subset(\bar A_R\Delta A')\Delta(\bar B_R\Delta B')\in\mathcal N_\omega\),
thus proving that \(A\setminus B\in\bar{\mathcal B}_\omega\) by arbitrariness of \(R>0\).
Let us now show that \(A\coloneqq\bigcup_{n\in\N}A^n\in\bar{\mathcal B}_\omega\) whenever
\((A^n)_{n\in\N}\subset\bar{\mathcal B}_\omega\). Let \(R>0\) be fixed. Given any
\(n\in\N\), pick some set \(\bar A^n_R\in\mathcal A_O(\bar X_\omega)\) such that
\((C_R\cap A^n)\Delta\bar A^n_R\in\mathcal N_\omega\), where we put
\(C_R\coloneqq\Pi_{i\to\omega}B(p_i,R)\). Then by Lemma \ref{lem:key_lemma}
there exists \(\bar A_R\in\mathcal A_O(\bar X_\omega)\) such that
\(\big(C_R\cap\bigcup_{n\in\N}\bar A^n_R\big)\Delta\bar A_R\in\mathcal N_\omega\).
Now observe that
\[\begin{split}
(C_R\cap A)\Delta\bar A_R&\subset
\big((C_R\cap A)\Delta(C_R\cap{\textstyle\bigcup_{n\in\N}}\bar A^n_R)\big)\cup
\big((C_R\cap{\textstyle\bigcup_{n\in\N}}\bar A^n_R)\Delta\bar A_R\big)\\
&\subset\big({\textstyle\bigcup_{n\in\N}}(C_R\cap A^n)\Delta\bar A^n_R\big)\cup
\big((C_R\cap{\textstyle\bigcup_{n\in\N}\bar A^n_R})\Delta\bar A_R\big).
\end{split}\]
Using Lemma \ref{lem:key_lemma} it is easy to check that countable unions of
null sets are null sets, thus \((C_R\cap A)\Delta\bar A_R\in\mathcal N_\omega\)
and so \(A\in\bar{\mathcal B}_\omega\). All in all, we have proven that
\(\bar{\mathcal B}_\omega\) is a \(\sigma\)-algebra.

Let us now show that \(\bar\mm_\omega\) is a measure on \(\bar{\mathcal B}_\omega\).
Suppose that \((A^n)_{n\in\N}\subset\bar{\mathcal B}_\omega\) are pairwise disjoint sets.
Define \(A\coloneqq\bigcup_{n\in\N}A^n\). We aim to prove that
\begin{equation}\label{eq:B_omega_sigma-alg_aux}
\bar\mm_\omega(A)=\sum_{n\in\N}\bar\mm_\omega(A^n).
\end{equation}
First of all, fix any \(R>0\) and choose
\(\bar A_R,\bar A^n_R\in\mathcal A_O(\bar X_\omega)\) such that
\(\big(A\cap\Pi_{i\to\omega}B(p_i,R)\big)\Delta\bar A_R\) and
\(\big(A^n\cap\Pi_{i\to\omega}B(p_i,R)\big)\Delta\bar A^n_R\)
belong to \(\mathcal N_\omega\). Observe that
\begin{equation}\label{eq:B_omega_sigma-alg_aux2}
\bar\mm_\omega\big(\bar A_R\setminus{\textstyle\bigcup_{n\leq N}}\bar A^n_R\big)
=\bar\mm_\omega(\bar A_R)-\sum_{n=1}^N\bar\mm_\omega(\bar A^n_R)
\longrightarrow 0,\quad\text{ as }N\to\infty.
\end{equation}
By using the fact that
\(\bar\mm_\omega(\bar A_R)\geq\sum_{n=1}^N\bar\mm_\omega(\bar A^n_R)\),
we can deduce that
\[
\bar\mm_\omega(A)=\lim_{R\to\infty}\bar\mm_\omega(\bar A_R)\geq
\lim_{R\to\infty}\sum_{n=1}^N\bar\mm_\omega(\bar A^n_R)
=\sum_{n=1}^N\bar\mm_\omega(A^n)\longrightarrow\sum_{n\in\N}\bar\mm_\omega(A^n),
\quad\text{ as }N\to\infty,
\]
thus \(\bar\mm_\omega(A)\geq\sum_{n\in\N}\bar\mm_\omega(A^n)\).
Let us prove the other inequality. If \(\sum_{n\in\N}\bar\mm_\omega(A^n)=+\infty\),
then the inequality trivially holds. Suppose now that
\(\sum_{n\in\N}\bar\mm_\omega(A^n)<+\infty\). Given any \(k\in\N\),
it follows from \eqref{eq:B_omega_sigma-alg_aux2} that there exists \(N_k\in\N\) such
that \(\bar\mm_\omega(\bar A_k)-\sum_{n\leq N_k}\bar\mm_\omega(\bar A^n_k)\leq 2^{-k}\). Then
\[
\bar\mm_\omega(\bar A_k)-\frac{1}{2^k}\leq
\sum_{n\leq N_k}\bar\mm_\omega(\bar A^n_k)
\leq\sum_{n\in\N}\lim_{k'\to\infty}\bar\mm_\omega(\bar A^n_{k'})
=\sum_{n\in\N}\bar\mm_\omega(A^n).
\]
By letting \(k\to\infty\) we conclude that \eqref{eq:B_omega_sigma-alg_aux}
holds, whence the statement follows.
\end{proof}
So far we have defined a natural measure on \(O(\bar X_\omega)\). Now we
have everything that we need to define the measure on the ultralimit of
pointed metric spaces, and hence to define the ultralimit of pointed metric
measure spaces.
\begin{definition}[Ultralimit of pointed metric measure spaces]
Let \(\big((X_i,d_i,\mm_i,p_i)\big)\) be a \(\omega\)-uniformly boundedly
finite sequence of pointed metric measure spaces. Let us consider the
\(\sigma\)-algebra \(\mathcal B_\omega\coloneqq\pi_*\bar{\mathcal B}_\omega\)
on \(X_\omega\), where the projection
\(\pi\colon O(\bar X_\omega)\to X_\omega\) is defined as in
\eqref{eq:def_proj_pi}. Given that
\({\rm Ball}(X_\omega)\subset\mathcal B_\omega\) by Proposition
\ref{prop:formula_balls}, we can define the ball measure \(\mm_\omega\)
on \(X_\omega\) as
\begin{equation}\label{eq:def_m_omega}
\mm_\omega\coloneqq(\pi_*\bar\mm_\omega)|_{{\rm Ball}(X_\omega)}.
\end{equation}
Then the \emph{ultralimit} of the sequence \(\big((X_i,d_i,\mm_i,p_i)\big)\)
is defined as
\[
\lim_{i\to\omega}(X_i,d_i,\mm_i,p_i)\coloneqq
(X_\omega,d_\omega,\mm_\omega,p_\omega).
\]
\end{definition}
Notice that, by the following reasoning, the previous definition makes sense.
Given any \([x_i]\in O(\bar X_\omega)\) and \(R>0\), it clearly holds that
\(\Pi_{i\to\omega}B(x_i,R+1/n)\in\mathcal A_O(\bar X_\omega)\)
for every \(n\in\N\), thus we have \(\bar B\big([x_i],R\big)\in\sigma
\big(\mathcal A_O(\bar X_\omega)\big)\subset\bar{\mathcal B}_\omega\)
by \eqref{eq:formula_balls}. This grants that
\({\rm Ball}\big(O(\bar X_\omega)\big)\subset\bar{\mathcal B}_\omega\).
Moreover, it is immediate to check that \(B\big([x_i],R\big)=\pi^{-1}
\big(B\big([[x_i]],R\big)\big)\) for all \([x_i]\in O(\bar X_\omega)\)
and \(R>0\). In particular, \({\rm Ball}(X_\omega)\subset\pi_*
{\rm Ball}\big(O(\bar X_\omega)\big)\). Given that
\({\rm Ball}\big(O(\bar X_\omega)\big)\subset\bar{\mathcal B}_\omega\),
we finally conclude that \({\rm Ball}(X_\omega)\subset
\pi_*\bar{\mathcal B}_\omega=\mathcal B_\omega\), thus
the definition \eqref{eq:def_m_omega} makes sense.
\begin{remark}\label{rmk:ballornot}{\rm
Here we have chosen to define the ultralimit as a metric measure
space in such a way that the measure is considered as a ball
measure. This choice is to make things more consistent and clean, and
is suitable for our purposes. However, in some cases, as in Example
\ref{ex:non_sep_UL_1}, this leads to a loss of information about the
measure. We point out that one might want to consider the measure
\(\mm_\omega\) as a measure on some larger $\sigma$-algebra instead. 
\fr}\end{remark}

The following construction and its variants will play a central role
in the rest of the paper. The example shows, for instance,
that the ball \(\sigma\)-algebra of the ultralimit space
\(X_\omega\) might be much smaller than \(\mathcal B_\omega\),
thus a fortiori also of the Borel \(\sigma\)-algebra of
\(X_\omega\).
\begin{example}\label{ex:non_sep_UL_1}{\rm
Given any \(i\in\N\), consider the pointed metric measure space
\((X_i,d_i,\mm_i,p_i)\), which is defined as follows: the set \(X_i\) is made of \(i\)
distinct points \(x^i_1,\ldots,x^i_i\), the distance \(d_i\) is given by
\(d_i(x^i_j,x^i_{j'})\coloneqq 1\) for every \(j,j'=1,\ldots,i\) with \(j\neq j'\),
the measure \(\mm_i\) is the uniformly distributed probability measure
\(\mm_i\coloneqq\frac{1}{i}\sum_{j=1}^i\delta_{x^i_j}\), and \(p_i\coloneqq x^i_1\).

Being \(((X_i,d_i,\mm_i,p_i)\big)\) a \(\omega\)-uniformly boundedly finite
sequence (trivially, since each \(\mm_i\) is a probability measure), its
ultralimit \((X_\omega,d_\omega,\mm_\omega,p_\omega)=\lim_{i\to\omega}
(X_i,d_i,\mm_i,p_i)\) exists. One can check that \(X_\omega\) is an
uncountable set and \(d_\omega(x,y)=1\) for every \(x,y\in X_\omega\) with
\(x\neq y\). Then
\begin{equation}\label{eq:example_balls}
B(x,r)=\left\{\begin{array}{ll}
\{x\}\\
X_\omega
\end{array}\quad\begin{array}{ll}
\text{ if }r\leq 1,\\
\text{ if }r>1.
\end{array}\right.
\end{equation}
In particular, every singleton in \(X_\omega\) is an open ball, so that
\(\mathscr B(X_\omega)=2^{X_\omega}\) (since every subset of \(X_\omega\)
is a union of singletons). On the other hand, the ball \(\sigma\)-algebra
on \(X_\omega\) is given by
\begin{equation}\label{eq:example_ball_sigma-alg}
{\rm Ball}(X_\omega)=\big\{A\subset X_\omega\,\big|\,
A\text{ is either countable or cocountable}\big\},
\end{equation}
where by \emph{cocountable} we mean that its complement \(X_\omega\setminus A\)
is countable. This identity can be easily verified: since \({\rm Ball}(X_\omega)\)
contains all singletons and is a \(\sigma\)-algebra, it includes the family in the
right-hand side of \eqref{eq:example_ball_sigma-alg}. Being the latter a
\(\sigma\)-algebra, we conclude that \eqref{eq:example_ball_sigma-alg} holds.

It can be readily checked that the limit measure \(\mm_\omega\) is given by
\begin{equation}\label{eq:example_meas}
\mm_\omega(A)=\left\{\begin{array}{ll}
0\\
1
\end{array}\quad\begin{array}{ll}
\text{ if }A\in{\rm Ball}(X_\omega)\text{ is countable,}\\
\text{ if }A\in{\rm Ball}(X_\omega)\text{ is cocountable.}
\end{array}\right.
\end{equation}
In particular, \(\mm_\omega\) is a probability measure.
\fr}\end{example}
\begin{remark}{\rm
As pointed out in Remark \ref{rmk:ballornot}, we lose information
on the measure \(\mm_\omega\) when we restrict it to the ball
\(\sigma\)-algebra. Indeed, while \(\mm_\omega\) achieves only
values \(0\) and \(1\) as a ball measure, it achieves all the values
between \(0\) and \(1\) when regarded as a measure on
\(\mathcal B_\omega\).
\fr}\end{remark}
\begin{lemma}\label{lem:m_omega_boundedly_finite}
Let \(\big((X_i,d_i,\mm_i,p_i)\big)\) be a \(\omega\)-uniformly boundedly
finite sequence of pointed metric measure spaces. Then the ultralimit
measure \(\mm_\omega\) is boundedly finite.
\end{lemma}
\begin{proof}
It suffices to prove that \(\mm_\omega\big(B(p_\omega,R)\big)\) is finite
for every \(R>0\), since every ball is contained in \(B(p_\omega,R)\)
for \(R>0\) sufficiently large. Let \(R>0\) be fixed. Given any
\(R'>R\), we know from \eqref{eq:formula_balls} that
\(B\big([p_i],R\big)\subset\Pi_{i\to\omega}B(p_i,R')\). Therefore, we have that
\[
\mm_\omega\big(B(p_\omega,R)\big)=\bar\mm_\omega\big(B\big([p_i],R\big)\big)
\leq\bar\mm_\omega\big(\Pi_{i\to\omega}B(p_i,R')\big)=
\lim_{i\to\infty}\mm_i\big(B(p_i,R')\big)\leq\eta(R')<+\infty,
\]
whence the statement follows.
\end{proof}
\begin{remark}\label{rmk:sep_issue}{\rm
Observe that if \(\mathscr B(X_\omega)\) and \({\rm Ball}(X_\omega)\) coincide
(which happens, for instance, when \(X_\omega\) is separable),
then \(\mm_\omega\) is a Borel measure on \(X_\omega\). On the other hand,
in general it holds that \({\rm Ball}(X_\omega)\) and \(\mathcal B_\omega\)
might fail to contain \(\mathscr B(X_\omega)\).
\fr}\end{remark}
\section{On the support-concentration issue}\label{s:sep_issue}
As mentioned in Remark \ref{rmk:sep_issue}, the ultralimit of metric measure
spaces might fail to be separable. In some cases, however, the measure is
still concentrated on a separable set and hence it might be thought of as
a Borel measure on its support. In this section, we will characterise those
sequences for which the ultralimit measure is concentrated on a separable set.
For that we will need the following definitions of measure-theoretic
variants of total boundedness and asymptotic total boundedness.
\begin{definition}[Bounded \(\mm\)-total boundedness]
Let \((X,d,\mm,p)\) be a pointed, boundedly finite metric measure space.
Then we say that \((X,d,\mm,p)\) is \emph{boundedly \(\mm\)-totally bounded}
provided for every \(R,r,\varepsilon>0\) there exists a finite collection
of points \((x_n)_{n=1}^M\subset X\) such that
\[
\mm\big(\bar B(p,R)\setminus{\textstyle\bigcup_{n=1}^M}B(x_n,r)\big)
\leq\varepsilon.
\]
\end{definition}
\begin{definition}[Asymptotic bounded \(\mm_\omega\)-total boundedness]
\label{def:abm_omegatb}
Let \(\big((X_i,d_i,\mm_i,p_i)\big)\) be a \(\omega\)-uniformly boundedly
finite sequence of pointed metric measure spaces. Then we say that the sequence
\(\big((X_i,d_i,\mm_i,p_i)\big)\) is \emph{asymptotically boundedly
\(\mm_\omega\)-totally bounded} provided for every given \(R,r,\varepsilon>0\)
there exist a number \(M\in\N\) and points \((x^i_n)_{n=1}^M\subset X_i\)
such that
\[
\lim_{i\to\omega}\mm_i\big(\bar B(p_i,R)\setminus{\textstyle\bigcup_{n=1}^M}
B(x^i_n,r)\big)\leq\varepsilon.
\]
\end{definition}
With these definitions in hand, we may state the main result of this section.
\begin{theorem}\label{thm:equiv_conc_spt}
Let \(\big((X_i,d_i,\mm_i,p_i)\big)\) be a \(\omega\)-uniformly boundedly
finite sequence of pointed metric measure spaces. Then the following
conditions are equivalent:
\begin{itemize}
\item[\(\rm i)\)] \(\mm_\omega\) is concentrated on \({\rm spt}(\mm_\omega)\).
\item[\(\rm ii)\)] \(\mm_\omega\) is concentrated on a closed
and separable set \(C\subset X_\omega\).
(Observe that Remark \ref{rmk:C_in_Ball} grants
that \(C\in{\rm Ball}(X_\omega)\).)
\item[\(\rm iii)\)] The ultralimit \((X_\omega,d_\omega,\mm_\omega,p_\omega)\)
is boundedly \(\mm_\omega\)-totally bounded.
\item[\(\rm iv)\)] The sequence \(\big((X_i,d_i,\mm_i,p_i)\big)\) is
asymptotically boundedly \(\mm_\omega\)-totally bounded.
\end{itemize}
\end{theorem}
\begin{proof}
\ \\
{\color{blue}\({\rm i)}\Longrightarrow{\rm ii)}\)}
It follows from the fact that \({\rm spt}(\mm_\omega)\) is closed and
separable, as shown in Lemma \ref{lem:prop_spt}. (Observe that the
\(\sigma\)-finiteness of the measure \(\mm_\omega\) is granted by
Lemma \ref{lem:m_omega_boundedly_finite}.)\\
{\color{blue}\({\rm ii)}\Longrightarrow{\rm i)}\)}
Suppose \(\mm_\omega\) is concentrated on a closed, separable
set \(C\in{\rm Ball}(X_\omega)\). If \(C\subset{\rm spt}(\mm_\omega)\),
then \(\mm_\omega\) is concentrated on \({\rm spt}(\mm_\omega)\).
Now suppose \(C\setminus{\rm spt}(\mm_\omega)\neq\emptyset\).
Given any \(x\in C\setminus{\rm spt}(\mm_\omega)\), by definition of
support we can find a radius \(r_x>0\) such that \(\mm_\omega\big(B(x,r_x)\big)=0\).
Being \(C\) separable, we can use Lindel\"{o}f's lemma to find a sequence
\((x_n)_{n\in\N}\subset C\setminus{\rm spt}(\mm_\omega)\)
such that \(C\setminus{\rm spt}(\mm_\omega)\subset\bigcup_{n\in\N}
B(x_n,r_{x_n})\). Therefore, we can conclude that
\[
\mm_\omega\big(X\setminus{\rm spt}(\mm_\omega)\big)\leq
\mm_\omega(X\setminus C)+\mm_\omega\big(C\setminus{\rm spt}(\mm_\omega)\big)
\leq\sum_{n\in\N}\mm_\omega\big(B(x_n,r_{x_n})\big)=0,
\]
which shows that \(\mm_\omega\) is concentrated on \({\rm spt}(\mm_\omega)\).\\
{\color{blue}\({\rm ii)}\Longrightarrow{\rm iii)}\)} Suppose \(\mm_\omega\)
is concentrated on a closed, separable set \(C\in{\rm Ball}(X_\omega)\).
Let \(R,r,\varepsilon>0\) be fixed. Choose a dense sequence \((y^n)_{n\in\N}\)
in \(C\). Then it holds \(\bar B(p_\omega,R)\cap C\subset\bigcup_{n\in\N}
B(y^n,r)\). By using the continuity from above of \(\mm_\omega\) and
the fact that \(\bar B(p_\omega,R)\) has finite \(\mm_\omega\)-measure,
we deduce that
\[
\lim_{M\to\infty}\mm_\omega\big(\bar B(p_\omega,R)
\setminus{\textstyle\bigcup_{n=1}^M}B(y^n,r)\big)=
\lim_{M\to\infty}\mm_\omega\big(\big(\bar B(p_\omega,R)\cap C\big)
\setminus{\textstyle\bigcup_{n=1}^M}B(y^n,r)\big)=0.
\]
Therefore, there exists \(M\in\N\) such that
\(\mm_\omega\big(\bar B(p_\omega,R)\setminus\bigcup_{n=1}^M B(y^n,r)\big)
\leq\varepsilon\), thus proving that the ultralimit
\((X_\omega,d_\omega,\mm_\omega,p_\omega)\) is boundedly
\(\mm_\omega\)-totally bounded.\\
{\color{blue}\({\rm iii)}\Longrightarrow{\rm ii)}\)}
Suppose \((X_\omega,d_\omega,\mm_\omega,p_\omega)\) is boundedly
\(\mm_\omega\)-totally bounded. Given any \(i,j,\ell\in\N\), choose
any \(M_{ij\ell}\in\N\) and \((x_n^{ij\ell})_{n=1}^{M_{ij\ell}}\subset
X_\omega\) such that
\[
\mm_\omega\Big(\bar B(p_\omega,i)\setminus{\textstyle
\bigcup_{n=1}^{M_{ij\ell}}}B\big(x_n^{ij\ell},1/j\big)\Big)\leq\frac{1}{\ell}.
\]
Let us define the closed and separable set \(C\subset X_\omega\) as
\[
C\coloneqq\overline{\bigcup_{i,j,\ell\in\N}
\big\{x_1^{ij\ell},\ldots,x_{M_{ij\ell}}^{ij\ell}\big\}}.
\]
We claim that \(\mm_\omega\) is concentrated on \(C\).
We argue by contradiction: suppose there is \(i_0\in\N\) such
that \(\mm_\omega\big(\bar B(p_\omega,i_0)\setminus C\big)>0\).
Since \(C\) is closed and separable, every \(\varepsilon\)-neighbourhood
\(C^\varepsilon\) of \(C\) can be written as a countable union of balls,
thus in particular \(C^\varepsilon\in{\rm Ball}(X_\omega)\).
Since \(C\) is the intersection of its neighbourhoods,
there is \(j_0\in\N\) such that \(\delta\coloneqq
\mm_\omega\big(\bar B(p_\omega,i_0)\setminus C^{1/j_0}\big)>0\).
Now fix \(\ell_0\in\N\) such that \(1/\ell_0\leq\delta/2\). Then
it holds that
\[
\delta=\mm_\omega\big(\bar B(p_\omega,i_0)\setminus C^{1/j_0}\big)
\leq\mm_\omega\Big(\bar B(p_\omega,i_0)\setminus{\textstyle
\bigcup_{n=1}^{M_{i_0 j_0\ell_0}}}B\big(x^{i_0 j_0\ell_0}_n,1/j_0\big)\Big)
\leq\frac{1}{\ell_0}\leq\frac{\delta}{2}<\delta,
\]
which leads to a contradiction. Consequently, \(\mm_\omega\) is
concentrated on \(C\), as required.\\
{\color{blue}\({\rm iii)}\Longrightarrow{\rm iv)}\)}
Suppose \((X_\omega,d_\omega,\mm_\omega,p_\omega)\) is boundedly
\(\mm_\omega\)-totally bounded. Let \(R,r,\varepsilon>0\) be fixed.
Pick \(R'>R\) and \(r'\in(0,r)\). Choose finitely many
points \(\big([[x^n_i]]\big)_{n=1}^M\subset X_\omega\) such that
\[
\mm_\omega\Big(\bar B(p_\omega,R')\setminus{\textstyle\bigcup_{n=1}^M
B\big([[x^n_i]],r'\big)}\Big)\leq\varepsilon.
\]
Then it holds that
\[
\lim_{i\to\omega}\mm_i\big(\bar B(p_i,R)\setminus{\textstyle
\bigcup_{n=1}^M B(x^n_i,r)}\big)\leq\mm_\omega\Big(\bar B(p_\omega,R')
\setminus{\textstyle\bigcup_{n=1}^M B\big([[x^n_i]],r'\big)}\Big)
\leq\varepsilon,
\]
thus showing that \(\big((X_i,d_i,\mm_i,p_i)\big)\) is
asymptotically boundedly \(\mm_\omega\)-totally bounded.\\
{\color{blue}\({\rm iv)}\Longrightarrow{\rm iii)}\)} Suppose the sequence
\(\big((X_i,d_i,\mm_i,p_i)\big)\) is asymptotically boundedly
\(\mm_\omega\)-totally bounded. Let \(R,r,\varepsilon>0\) be given.
Pick any \(R'>R\) and \(r'\in(0,r)\). Then there exist a number \(M\in\N\)
and points \((x^n_i)_{n=1}^M\subset X_i\) such that
\[
\lim_{i\to\omega}\mm_i\big(\bar B(p_i,R')\setminus{\textstyle
\bigcup_{n=1}^M}B(x^n_i,r')\big)\leq\varepsilon.
\]
Define \(x^n\coloneqq[[x^n_i]]\in X_\omega\) for every
\(n=1,\ldots,M\). Therefore, it holds that
\[\begin{split}
\mm_\omega\big(\bar B(p_\omega,R)\setminus{\textstyle
\bigcup_{n=1}^M}B(x^n,r)\big)&=\inf_{k\in\N}\lim_{i\to\omega}
\mm_i\Big(\bar B\big(p_i,R+1/k\big)\setminus{\textstyle\bigcup_{n=1}^M}
B\big(x^n_i,r-1/k\big)\Big)\\
&\leq\lim_{i\to\omega}\mm_i\big(\bar B(p_i,R')\setminus{\textstyle
\bigcup_{n=1}^M}B(x^n_i,r')\big)\leq\varepsilon,
\end{split}\]
thus proving that \((X_\omega,d_\omega,\mm_\omega,p_\omega)\) is boundedly
\(\mm_\omega\)-totally bounded.
\end{proof}
\begin{example}\label{ex:non_sep_UL_2}{\rm
Let \(\big((X_i,d_i,\mm_i,p_i)\big)\) be as in Example \ref{ex:non_sep_UL_1}.
Let us show that \(\big((X_i,d_i,\mm_i,p_i)\big)\) is not asymptotically
boundedly \(\mm_\omega\)-totally bounded: observe that for any \(i\in\N\)
it holds that
\[
F\subset X_i,\;\;\mm_i\big(X_i\setminus{\textstyle\bigcup_{x\in F}}B(x,1)\big)
\leq\frac{1}{2}\quad\Longrightarrow\quad\# F\geq \frac{i}{2},
\]
which proves that \(\big((X_i,d_i,\mm_i,p_i)\big)\) fails to
be asymptotically boundedly \(\mm_\omega\)-totally bounded for the
choice \((R,r,\varepsilon)=(2,1,1/2)\). This fact is consistent with Theorem
\ref{thm:equiv_conc_spt}; indeed, since we have that
\(\mm_\omega\big(B(x,1)\big)=\mm_\omega\big(\{x\}\big)=0\) for every
\(x\in X_\omega\) by \eqref{eq:example_balls} and \eqref{eq:example_meas},
it holds that \({\rm spt}(\mm_\omega)=\emptyset\) and thus \(\mm_\omega\)
is not concentrated on \({\rm spt}(\mm_\omega)\).
\fr}\end{example}
\begin{lemma}\label{lem:pmGH_implies_omega-ubf}
Let \(\big\{(X_i,d_i,\mm_i,p_i)\big\}_{i\in\bar\N}\) be a sequence
of pointed Polish metric measure spaces, with \(\mm_i\) boundedly finite.
Suppose \((X_i,d_i,\mm_i,p_i)\to(X_\infty,d_\infty,\mm_\infty,p_\infty)\)
in the pmGH sense. Then it holds that the sequence \(\big((X_i,d_i,\mm_i,p_i)\big)\)
is \(\omega\)-uniformly boundedly finite.
\end{lemma}
\begin{proof}
Fix \((R_i,\varepsilon_i)\)-approximations \(\psi_i\colon X_i\to X_\infty\) as
in Definition \ref{def:pmGH}. We claim that
\begin{equation}\label{eq:pmGH_implies_omega-ubf_aux}
\lims_{i\to\infty}\mm_i\big(B(p_i,R)\big)\leq\mm_\infty\big(\bar B(p_\infty,R)\big),
\quad\text{ for every }R>0.
\end{equation}
Let \(R>0\) be fixed. Pick any sequence \((f_k)_k\subset C_b\big(B(p_\infty,R+1)\big)\)
with \(0\leq f_k\leq 1\) such that \({\rm dist}\big(\{f_k<1\},\bar B(p_\infty,R)\big)>0\)
and \(f_k\to\nchi_{\bar B(p_\infty,R)}\). Fix \(k\in\N\), then choose any \(j\in\N\)
such that \(R_i>R\), \(\varepsilon_i<1\), and \(f_k=1\) on \(B(p_\infty,R+\varepsilon_i)\)
for all \(i\geq j\). Since \(\psi_i\) is a \((R_i,\varepsilon_i)\)-approximation,
we see that \(\psi_i\big(B(p_i,R)\big)\subset B(p_\infty,R+\varepsilon_i)\), thus
\(B(p_i,R)\subset\psi_i^{-1}\big(B(p_\infty,R+\varepsilon_i)\big)\) and
\[\begin{split}
\int f_k\,\d\mm_\infty&=\lim_{i\to\infty}\int f_k\,\d(\psi_i)_*\mm_i
=\lim_{i\to\infty}\int f_k\circ\psi_i\,\d\mm_i\geq\lims_{i\to\infty}
\int\nchi_{B(p_\infty,R+\varepsilon_i)}\circ\psi_i\,\d\mm_i\\
&=\lims_{i\to\infty}\mm_i\big(\psi_i^{-1}\big(B(p_\infty,R+\varepsilon_i)\big)\big)
\geq\lims_{i\to\infty}\mm_i\big(B(p_i,R)\big).
\end{split}\]
By letting \(k\to\infty\) and using the dominated convergence theorem, we get
\eqref{eq:pmGH_implies_omega-ubf_aux}. Now let us define the function
\(\eta\colon(0,+\infty)\to[1,+\infty)\) as
\(\eta(R)\coloneqq\mm_\infty\big(\bar B(p_\infty,R)\big)+1\) for every \(R>0\)
(here, we are using the assumption that \(\mm_\infty\) is boundedly finite).
Then \eqref{eq:pmGH_implies_omega-ubf_aux} implies that for any \(R>0\)
we have that \(\mm_i\big(B(p_i,R)\big)\leq\eta(R)\) for \(\omega\)-a.e.\ \(i\), which
shows that \(\big((X_i,d_i,\mm_i,p_i)\big)\) is \(\omega\)-uniformly boundedly finite.
\end{proof}
\begin{proposition}\label{prop:m_omega_conc_spt}
Let \(\big\{(X_i,d_i,\mm_i,p_i)\big\}_{i\in\bar\N}\) be a sequence
of pointed Polish metric measure spaces, with \(\mm_i\) boundedly finite.
Suppose \((X_i,d_i,\mm_i,p_i)\to(X_\infty,d_\infty,\mm_\infty,p_\infty)\)
in the pmGH sense. Then the sequence \(\big((X_i,d_i,\mm_i,p_i)\big)\) is
asymptotically boundedly \(\mm_\omega\)-totally bounded.

In particular, it holds that \(\mm_\omega\) is concentrated on \({\rm spt}(\mm_\omega)\).
\end{proposition}
\begin{proof}
Fix \((R_i,\varepsilon_i)\)-approximations \(\psi_i\colon X_i\to X_\infty\) as
in Definition \ref{def:pmGH}. Let \(\phi_i\colon X_\infty\to X_i\) be a quasi-inverse
of \(\psi_i\). Fix any \(R,r,\varepsilon>0\). Being \(\mm_\infty\) boundedly finite
and \(X_\infty\) separable, we can find \(M\in\N\) and points
\((x_\infty^n)_{n=1}^M\subset \bar B(p_\infty,R)\) such that
\[
\mm_\infty(C_\infty)\leq\varepsilon,\quad\text{ where we set }
C_\infty\coloneqq\bar B(p_\infty,R)\setminus\bigcup_{n=1}^M B(x_\infty^n,r).
\]
Fix an open, bounded set \(\Omega\subset X_\infty\)
with \(C_\infty\subset\Omega\) and a sequence
\((f_k)_k\subset C_b(\Omega)\) with \(0\leq f_k\leq 1\) such that
\({\rm dist}\big(\{f_k<1\},C_\infty\big)>0\) and \(f_k\to\nchi_{C_\infty}\).
We define the sets \(C_i\subset X_i\) as
\[
C_i\coloneqq\bar B(p_i,R)\setminus
\bigcup_{n=1}^M B\big(\phi_i(x_\infty^n),r\big),\quad\text{ for every }i\in\N.
\]
Fix \(k\in\N\), then choose any \(j\in\N\) such that
\(f_k=1\) on \(\bar B(p_\infty,R+\varepsilon_i)\setminus\bigcup_{n=1}^M
B(x^n_\infty,r-7\varepsilon_i)\), \(r+R<R_i+\varepsilon_i\), and
\(r>7\varepsilon_i\) for all \(i\geq j\). Being the map \(\psi_i\) a
\((R_i,\varepsilon_i)\)-approximation, we see that
\(\bar B(p_i,R)\subset\psi_i^{-1}\big(\bar B(p_\infty,R+\varepsilon_i)\big)\).
Moreover, by using \eqref{eq:tech_GH_2} we can deduce that
\(\psi_i^{-1}\big(B(x^n_\infty,r-7\varepsilon_i)\big)\subset
B\big(\phi_i(x^n_\infty),r\big)\). All in all, we have proven that
for any \(i\geq j\) it holds
\[
C_i\subset\psi_i^{-1}(\tilde C_i),\quad\text{ where we set }
\tilde C_i\coloneqq\bar B(p_\infty,R+\varepsilon_i)\setminus
\bigcup_{n=1}^M B(x^n_\infty,r-7\varepsilon_i).
\]
Therefore, we conclude that
\[\begin{split}
\int f_k\,\d\mm_\infty&=\lim_{i\to\infty}\int f_k\,\d(\psi_i)_*\mm_i
=\lim_{i\to\infty}\int f_k\circ\psi_i\,\d\mm_i\geq
\lims_{i\to\infty}\int\nchi_{\tilde C_i}\circ\psi_i\,\d\mm_i\\
&=\lims_{i\to\infty}\mm_i\big(\psi_i^{-1}(\tilde C_i)\big)
\geq\lims_{i\to\infty}\mm_i(C_i).
\end{split}\]
By letting \(k\to\infty\) and using the dominated convergence theorem,
we thus obtain that
\[
\lim_{i\to\omega}\mm_i(C_i)\leq\lims_{i\to\infty}\mm_i(C_i)
\leq\mm_\infty(C_\infty)\leq\varepsilon,
\]
showing that \(\big((X_i,d_i,\mm_i,p_i)\big)\) is an
asymptotically boundedly \(\mm_\omega\)-totally bounded sequence, as required.
The last part of the statement now follows from Theorem
\ref{thm:equiv_conc_spt}.
\end{proof}
\section{Relation with the pmGH convergence}
\label{ss:relation_with_pmGH}
Aim of this section is to investigate the relation between the
pointed (measured) Gromov--Hausdorff convergence and the ultralimits
of pointed metric (measure) spaces.
\medskip

First of all, let us prove the following technical result,
which will be needed in the sequel.
\begin{lemma}\label{lem:non_ex_UL}
Let \((X,d)\) be a complete metric space. Let \((x_i)_{i\in\N}\subset X\) be a
sequence for which the ultralimit \(\lim_{i\to\omega}x_i\) does not exist.
Then there exists \(\varepsilon>0\) with the following property: given any
\(S\in\omega\), the set \(\{x_i\}_{i\in S}\) cannot be covered by finitely many
balls of radius \(\varepsilon\).
\end{lemma}
\begin{proof}
We argue by contradiction: suppose there exist \(\varepsilon_k\searrow 0\) and
\((S^k)_{k\in\N}\subset\omega\) such that
\[
\{x_i\}_{i\in S^k}\subset\bigcup_{n=1}^{n_k}B(x^k_n,\varepsilon_k),\quad
\text{ for some }n_k\in\N\text{ and }\{x^k_n\}_{n=1}^{n_k}\subset\{x_i\}_{i\in S^k}.
\]
In particular, by using the upward-closedness of \(\omega\), we see that
for every \(k\in\N\) it holds
\[
\bigcup_{n=1}^{n_k}\big\{i\in\N\;\big|\;x_i\in B(x^k_n,\varepsilon_k)\big\}
=\bigg\{i\in\N\;\bigg|\;x_i\in\bigcup_{n=1}^{n_k}B(x^k_n,\varepsilon_k)\bigg\}\in\omega.
\]
Therefore, there must exist \(\tilde x^k\in\{x^k_n\}_{n=1}^{n_k}\) such that
\(\big\{i\in\N\,\big|\,x_i\in B(\tilde x^k,\varepsilon_k)\big\}\in\omega\).
Now let us recursively define \((T^k)_{k\in\N}\subset\omega\) as follows:
\(T^1\coloneqq\big\{i\in\N\,\big|\,d(x_i,\tilde x^1)<\varepsilon_1\big\}\) and
\[
T^{k+1}\coloneqq\big\{i\in T^k\,\big|\,d(x_i,\tilde x^k)<\varepsilon_k\big\},
\quad\text{ for every }k\in\N.
\]
We claim that the sequence \((\tilde x^k)_{k\in\N}\subset X\) is Cauchy.
Observe that for any \(k,k_1,k_2\in\N\) such that \(k\leq k_1\leq k_2\)
and \(j\in T^{k_2}\), we have that
\[
d(\tilde x^{k_1},\tilde x^{k_2})\leq d(\tilde x^{k_1},x_j)+d(x_j,\tilde x^{k_2})
\leq\varepsilon_{k_1}+\varepsilon_{k_2}\leq 2\varepsilon_k.
\]
This implies that
\(\lim_{k\to\infty}\sup\big\{d(\tilde x^{k_1},\tilde x^{k_2})\,:\,k_1,k_2\geq k\big\}
\leq 2\lim_{k\to\infty}\varepsilon_k=0\), getting the claim. Given that \((X,d)\) is
complete, the limit \(x\coloneqq\lim_{k\to\infty}\tilde x^k\in X\) exists. Now fix any
\(\varepsilon>0\) and choose \(k_0\in\N\) such that \(\varepsilon_{k_0}<\varepsilon/2\)
and \(d(x,\tilde x^{k_0})<\varepsilon/2\). Hence, we have that
\[
\big\{i\in\N\,\big|\,d(x_i,x)<\varepsilon\big\}\supset
\big\{i\in\N\,\big|\,d(x_i,\tilde x^{k_0})<\varepsilon/2\big\}\supset
\big\{i\in\N\,\big|\,d(x_i,\tilde x^{k_0})<\varepsilon_{k_0}\big\}\in\omega.
\]
This grants that \(\big\{i\in\N\,\big|\,d(x_i,x)<\varepsilon\big\}\in\omega\)
for every \(\varepsilon>0\), which shows that the ultralimit
\(\lim_{i\to\omega}x_i\in X\) exists (and coincides with \(x\)),
thus leading to a contradiction.
\end{proof}
\begin{theorem}\label{thm:pGH_vs_UL}
Let \(\big\{(X_i,d_i,p_i)\big\}_{i\in\bar\N}\) be a sequence of pointed
Polish metric spaces. Suppose \((X_i,d_i,p_i)\to(X_\infty,d_\infty,p_\infty)\)
in the pGH sense. Let \(\psi_i\colon X_i\to X_\infty\) be
\((R_i,\varepsilon_i)\)-approximations, where \(R_i\nearrow+\infty\) and
\(\varepsilon_i\searrow 0\), with quasi-inverses \(\phi_i\colon X_\infty\to X_i\). Define \(\phi_\infty\colon X_\infty\to X_\omega\) as
\[
\phi_\infty(x)\coloneqq\big[\big[\phi_i(x)\big]\big],
\quad\text{ for every }x\in X_\infty.
\]
Then \(\phi_\infty\) is an isometric embedding such that
\(\phi_\infty(p_\infty)=p_\omega\). The map
\(\psi_\infty\colon\phi_\infty(X_\infty)\to X_\infty\), which is given by
\[
\psi_\infty\big([[x_i]]\big)\coloneqq\lim_{i\to\omega}\psi_i(x_i),
\quad\text{ for every }[[x_i]]\in\phi_\infty(X_\infty),
\]
is well-defined and is the inverse of \(\phi_\infty\colon X_\infty
\to\phi_\infty(X_\infty)\). Moreover, it holds that
\begin{equation}\label{eq:id_sigma_alg}
(\phi_\infty)_*\mathscr B(X_\infty)=\mathscr B\big(\phi_\infty(X_\infty)\big)
\subset{\rm Ball}(X_\omega).
\end{equation}
\end{theorem}
\begin{proof}
Let \(x,y\in X_\infty\) be fixed. Pick any \(j\in\N\) such that
\(x,y\in B(p_\infty,R_i-3\varepsilon_i)\) and \(4\varepsilon_i<R_i\)
for every \(i\geq j\). Then it holds
\[\begin{split}
\Big|d_\infty(x,y)-d_i\big(\phi_i(x),\phi_i(y)\big)\Big|
\leq\,&d_\infty\big(x,(\psi_i\circ\phi_i)(x)\big)+
d_\infty\big(y,(\psi_i\circ\phi_i)(y)\big)\\
&+\Big|d_\infty\big(\phi_i(x),\phi_i(y)\big)-
d_\infty\big((\psi_i\circ\phi_i)(x),(\psi_i\circ\phi_i)(y)\big)\Big|\\
\leq\,&7\varepsilon_i,
\end{split}\]
for every \(i\geq j\), whence it follows that
\[
d_\infty(x,y)=\lim_{i\to\omega}d_i\big(\phi_i(x),\phi_i(y)\big)
=d_\omega\big(\phi_\infty(x),\phi_\infty(y)\big).
\]
This shows that \(\phi_\infty\) is an isometry. Now let
\([[x_i]]\in\phi_\infty(X_\infty)\) be fixed. Pick any
\(x\in X_\infty\) for which \(\phi_\infty(x)=[[x_i]]\).
Since \([x_i]\in O(\bar X_\omega)\), one has
\(\lim_{i\to\omega}d_i(x_i,p_i)=\bar d_\omega\big([x_i],[p_i]\big)<+\infty\),
thus there exist \(R>0\) and \(S\in\omega\) such that
\(d_i(x_i,p_i)\leq R\) for every \(i\in S\). Fix any \(j\in\N\)
such that \(R<R_i-4\varepsilon_i\) and \(x\in B(p_\infty,R_i-\varepsilon_i)\)
for every \(i\geq j\). Notice that
\[
d_\infty\big(\psi_i(x_i),p_\infty\big)\leq d_i(x_i,p_i)+\varepsilon_i
\leq R+\varepsilon_i<R_i-3\varepsilon_i,\quad\text{ for every }i\in S
\text{ with }i\geq j,
\]
so that \(\psi_i(x_i)\in B(p_\infty,R_i-3\varepsilon_i)\). Therefore,
for \(\omega\)-a.e.\ \(i\) it holds that
\[\begin{split}
d_\infty\big(\psi_i(x_i),x\big)&\leq d_i\big((\phi_i\circ\psi_i)(x_i),
\phi_i(x)\big)+3\varepsilon_i\leq
d_i\big(\phi_i(x),x_i\big)+d_i\big(x_i,(\phi_i\circ\psi_i)(x_i)\big)
+3\varepsilon_i\\
&\leq d_i\big(\phi_i(x),x_i\big)+6\varepsilon_i,
\end{split}\]
thus accordingly \(\lim_{i\to\omega}d_\infty\big(\psi_i(x_i),x\big)
\leq\lim_{i\to\omega}d_i\big(\phi_i(x),x_i\big)
=\bar d_\omega\big(\big[\phi_i(x)\big],[x_i]\big)=0\).
This shows that the ultralimit \(\psi_\infty\big([[x_i]]\big)
\coloneqq\lim_{i\to\omega}\psi_i(x_i)\) exists and coincides
with \(x\), in other words \((\psi_\infty\circ\phi_\infty)(x)=x\).
Then \(\psi_\infty\colon\phi_\infty(X_\infty)\to X_\infty\)
is the inverse of \(\phi_\infty\colon X_\infty\to\phi_\infty(X_\infty)\).

Finally, let us prove \eqref{eq:id_sigma_alg}. Given that \(\phi_\infty\)
is an isometric bijection between the metric spaces \((X_\infty,d_\infty)\)
and \(\big(\phi_\infty(X_\infty),d_\omega|_{\phi_\infty(X_\infty)
\times\phi_\infty(X_\infty)}\big)\), it holds
\((\phi_\infty)_*\mathscr B(X_\infty)=\mathscr B\big(\phi_\infty(X_\infty)\big)\).
Moreover, \(\phi_\infty(X_\infty)\) is a closed separable subset of
\(X_\omega\), thus every closed subset \(C\) of \(\phi_\infty(X_\infty)\)
(which is a fortiori closed in \(X_\omega\)) belongs to \({\rm Ball}(X_\omega)\)
by Remark \ref{rmk:C_in_Ball}. Since \(\mathscr B\big(\phi_\infty(X_\infty)\big)\)
is generated by the closed subsets of \(\phi_\infty(X_\infty)\), we can
conclude that \eqref{eq:id_sigma_alg} is verified.
\end{proof}
\begin{corollary}\label{cor:pGH_vs_UL_proper}
Let \(\big\{(X_i,d_i,p_i)\big\}_{i\in\bar\N}\) be a sequence of pointed Polish metric
spaces. Suppose \((X_i,d_i,p_i)\to(X_\infty,d_\infty,p_\infty)\) in the pGH sense.
Let \(\psi_i\colon X_i\to X_\infty\) be \((R_i,\varepsilon_i)\)-approximations,
with quasi-inverses \(\phi_i\colon X_\infty\to X_i\). Let
\(\phi_\infty\colon X_\infty\to X_\omega\) be as in Theorem \ref{thm:pGH_vs_UL}.
Then an element \([[x_i]]\in X_\omega\) belongs to \(\phi_\infty(X_\infty)\) if
and only if the ultralimit \(x\coloneqq\lim_{i\to\omega}\psi_i(x_i)\in X_\infty\) exists.
In this case, it holds that \(\phi_\infty(x)=[[x_i]]\). In particular, if
\((X_\infty,d_\infty)\) is a proper metric space, then \(\phi_\infty(X_\infty)=X_\omega\)
and accordingly \(\phi_\infty\) is an isomorphism of pointed metric spaces.
\end{corollary}
\begin{proof}
Let \([[x_i]]\in X_\omega\) be given. If \([[x_i]]\in\phi_\infty(X_\infty)\),
then we know from Theorem \ref{thm:pGH_vs_UL} that the ultralimit
\(\psi_\infty\big([[x_i]]\big)=\lim_{i\to\omega}\psi_i(x_i)\in X_\infty\) exists
and satisfies \((\phi_\infty\circ\psi_\infty)\big([[x_i]]\big)=[[x_i]]\).

Conversely, suppose that \(x\coloneqq\lim_{i\to\omega}\psi_i(x_i)\in X_\infty\) exists.
Fix any \(j\in\N\). For all \(i\in\N\) sufficiently big, we have that
\[
\Big|d_i\big((\phi_i\circ\psi_j)(x_j),x_i\big)-
d_\infty\big(\psi_j(x_j),\psi_i(x_i)\big)\Big|\leq 6\varepsilon_i.
\]
In particular, it holds that
\[\begin{split}
d_\omega\big(\phi_\infty\big(\psi_j(x_j)\big),[[x_i]]\big)&=
d_\omega\big(\big[\big[(\phi_i\circ\psi_j)(x_j)\big]\big],[[x_i]]\big)=
\lim_{i\to\omega}d_i\big((\phi_i\circ\psi_j)(x_j),x_i\big)\\
&=\lim_{i\to\omega}d_\infty\big(\psi_j(x_j),\psi_i(x_i)\big)
=d_\infty\big(\psi_j(x_j),x\big).
\end{split}\]
Consequently, by using the continuity of \(\phi_\infty\),
we eventually conclude that
\[
d_\omega\big(\phi_\infty(x),[[x_i]]\big)
=\lim_{j\to\omega}d_\omega\big(\phi_\infty\big(\psi_j(x_j)\big),[[x_i]]\big)
=\lim_{j\to\omega}d_\infty\big(\psi_j(x_j),x\big)=0.
\]
This grants that \([[x_i]]=\phi_\infty(x)\in\psi_\infty(X_\infty)\),
thus proving the first part of the statement.

Now suppose \((X_\infty,d_\infty)\) is proper. Fix \([[x_i]]\in X_\omega\).
Since \(\lim_{i\to\omega}d_i(x_i,p_i)=d_\omega\big([[x_i]],p_\omega\big)\)
is finite, we can find \(R>0\) and \(S\in\omega\) such that \(x_i\in\bar B(p_i,R)\)
for all \(i\in S\). Moreover, there exists \(j\in\N\) such that
\(\bar B(p_i,R)\subset B(p_i,R_i)\) and \(\varepsilon_i\leq 1\) for every \(i\geq j\),
thus in particular we have that \(\psi_i(x_i)\in\bar B(p_\infty,R+1)\) for
\(\omega\)-a.e.\ \(i\). Being \(\bar B(p_\infty,R+1)\) compact, we conclude that
the ultralimit \(\lim_{i\to\omega}\psi_i(x_i)\in X_\infty\) exists, whence it follows
from the first part of the statement that \([[x_i]]\in\phi_\infty(X_\infty)\).
This shows that \(X_\omega=\phi_\infty(X_\infty)\), as desired.
\end{proof}
\begin{theorem}\label{thm:pmGH_vs_UL}
Let \(\big\{(X_i,d_i,\mm_i,p_i)\big\}_{i\in\bar\N}\) be a sequence of
pointed Polish metric measure spaces, with \(\mm_i\) boundedly finite.
Suppose \((X_i,d_i,\mm_i,p_i)\to(X_\infty,d_\infty,\mm_\infty,p_\infty)\)
in the pmGH sense. Then the measure \(\mm_\omega\) is concentrated on
\({\rm spt}(\mm_\omega)=\phi_\infty\big({\rm spt}(\mm_\infty)\big)\) and
\[
(\phi_\infty)_*\mm_\infty=\mm_\omega.
\]
\end{theorem}
\begin{proof}
We subdivide the proof into several steps:\\
{\color{blue}\textsc{Step 1.}}
Define \(Y\coloneqq\phi_\infty\big({\rm spt}(\mm_\infty)\big)\).
Notice that \(Y\) is a Polish metric space if endowed with the restricted
distance from \(X_\omega\). First of all, we aim to prove the inclusion
\({\rm spt}(\mm_\omega)\subset Y\).

Let \([[x_i]]\in{\rm spt}(\mm_\omega)\) be fixed.
This means that \(\bar\mm_\omega\big(B\big([x_i],r\big)\big)
=\mm_\omega\big(B\big([[x_i]],r\big)\big)>0\) for all \(r>0\).
Thanks to \eqref{eq:formula_balls}, this is equivalent to saying that
\begin{equation}\label{eq:pmGH_vs_UL_aux5}
\lim_{i\to\omega}\mm_i\big(B(x_i,r)\big)=
\bar\mm_\omega\big(\Pi_{i\to\omega}B(x_i,r)\big)>0,\quad\text{ for every }r>0.
\end{equation}
Also, we have \(\lim_{i\to\omega}d_i(x_i,p_i)=
\bar d_\omega\big([x_i],[p_i]\big)<+\infty\), so there are
\(R>0\) and \(S\in\omega\) such that
\begin{equation}\label{eq:pmGH_vs_UL_aux6}
x_i\in B(p_i,R),\quad\text{ for every }i\in S.
\end{equation}
Our aim is to show that \([[x_i]]\in Y\). To achieve this goal,
we first need to prove that
\begin{equation}\label{eq:pmGH_vs_UL_aux7}
\exists\,x\coloneqq\lim_{i\to\omega}\psi_i(x_i)\in X_\infty.
\end{equation}
We argue by contradiction: suppose \(\lim_{i\to\omega}\psi_i(x_i)\)
does not exist. By using Lemma \ref{lem:non_ex_UL}, we can find
\(\varepsilon\in(0,1)\) with the property that for every \(S'\in\omega\)
the set \(\big(\psi_i(x_i)\big)_{i\in S'}\) cannot be covered by finitely
many balls of radius \(\varepsilon\). Call \(c\coloneqq\bar\mm_\omega
\big(\Pi_{i\to\omega}B(x_i,\varepsilon/4)\big)/2>0\). Let us define
\[
S^1\coloneqq\Big\{i\in S\,\Big|\,d_\infty\big(\psi_i(x_i),\psi_1(x_1)\big)
\geq\varepsilon,\,\mm_i\big(B(x_i,\varepsilon/4)\big)>c\Big\}
\]
and, for every \(k\in\N\),
\[
S^{k+1}\coloneqq\Big\{i\in S^k\,\Big|\,d_\infty\big(\psi_i(x_i),
\psi_{i_k}(x_{i_k})\big)\geq\varepsilon\Big\},
\quad\text{ where we set }i_k\coloneqq\min(S_k).
\]
Let us denote \(T^k\coloneqq\big\{i\in S\,:\,d_\infty\big(\psi_i(x_i),
\psi_{i_k}(x_{i_k})\big)<\varepsilon\big\}\) for every \(k\in\N\).
Given that it holds that \(\big(\psi_i(x_i)\big)_{i\in T^k}\subset
B\big(\psi_{i_k}(x_{i_k}),\varepsilon\big)\), we deduce that
\(T^k\notin\omega\), so that \(\N\setminus T^k\in\omega\) and
accordingly \((S^k)_{k\in\N}\subset\omega\). In particular,
\((x_{i_k})_k\) is a subsequence of \((x_i)_i\). Observe that
\begin{equation}\label{eq:pmGH_vs_UL_aux8}
d_\infty\big(\psi_{i_k}(x_{i_k}),\psi_{i_j}(x_{i_j})\big)\geq\varepsilon,
\qquad\mm_{i_k}\big(B(x_{i_k},\varepsilon/4)\big)>c,
\end{equation}
for every \(k,j\in\N\). Since \(\sum_{k\in\N}\mm_\infty
\big(B\big(\psi_{i_k}(x_{i_k}),\varepsilon/2\big)\big)\leq
\mm_\infty\big(B(p_\infty,2R+1)\big)<+\infty\) by
\eqref{eq:pmGH_vs_UL_aux6}, there exists \(\bar k\in\N\) such that
\begin{equation}\label{eq:pmGH_vs_UL_aux9}
\mm_\infty\bigg(\bigcup_{k\geq\bar k}B\big(\psi_{i_k}(x_{i_k}),
\varepsilon/2\big)\bigg)=\sum_{k\geq\bar k}\mm_\infty
\big(B\big(\psi_{i_k}(x_{i_k}),\varepsilon/2\big)\big)\leq\frac{c}{2}.
\end{equation}
It can be readily checked that there exists a function
\(f\in C_{bbs}(X_\infty)\) with \(0\leq f\leq 1\) and
\[\begin{split}
f=1&\quad\text{ on }B\big(\psi_{i_k}(x_{i_k}),\varepsilon/3\big),
\text{ for every }k\geq\bar k,\\
f=0&\quad\text{ on }X_\infty\setminus{\textstyle\bigcup_{k\geq\bar k}}
B\big(\psi_{i_k}(x_{i_k}),\varepsilon/2\big).
\end{split}\]
Given that \(\psi_i\big(B(x_i,\varepsilon/4)\big)\subset
B\big(\psi_i(x_i),\varepsilon/3\big)\) for all \(i\) sufficiently big
and \((\psi_i)_*\mm_i\rightharpoonup\mm_\infty\) weakly,
\[\begin{split}
c&\overset{\eqref{eq:pmGH_vs_UL_aux8}}\leq
\lims_{k\to\infty}\mm_{i_k}\big(B(x_{i_k},\varepsilon/4)\big)
\leq\lims_{k\to\infty}\mm_{i_k}\Big(\psi_{i_k}^{-1}
\big(B\big(\psi_{i_k}(x_{i_k}),\varepsilon/3\big)\big)\Big)
\leq\lim_{k\to\infty}\int f\,\d(\psi_{i_k})_*\mm_{i_k}\\
&\overset{\phantom{\eqref{eq:pmGH_vs_UL_aux8}}}=
\int f\,\d\mm_\infty\leq\mm_\infty\Big({\textstyle\bigcup_{k\geq\bar k}}
B\big(\psi_{i_k}(x_{i_k}),\varepsilon/2\big)\Big)
\overset{\eqref{eq:pmGH_vs_UL_aux9}}\leq\frac{c}{2}.
\end{split}\]
This leads to a contradiction, thus accordingly the claim
\eqref{eq:pmGH_vs_UL_aux7} is proven.\\
{\color{blue}\textsc{Step 2.}} We now claim that
\begin{equation}\label{eq:pmGH_vs_UL_aux10}
x=\lim_{i\to\omega}\psi_i(x_i)\in{\rm spt}(\mm_\infty).
\end{equation}
Let \(r>0\) be fixed. We know from \eqref{eq:pmGH_vs_UL_aux5} that
\(\lambda\coloneqq\lim_{i\to\omega}\mm_i\big(B(x_i,r/4)\big)>0\).
By using this fact, \eqref{eq:pmGH_vs_UL_aux6}, and
\eqref{eq:pmGH_vs_UL_aux7}, we can find a subsequence \((i_k)_{k\in\N}\)
such that for every \(k\in\N\) it holds
\[
\varepsilon_{i_k}\leq\frac{r}{4},\qquad
\mm_{i_k}\big(B(x_{i_k},r/4)\big)\geq\frac{\lambda}{2},\qquad
x_{i_k}\in B\big(p_{i_k},R_{i_k}-r/4\big),\qquad
d_\infty\big(\psi_{i_k}(x_{i_k}),x\big)<\frac{r}{2}.
\]
Given that \(B(x_{i_k},r/4)\subset B(p_{i_k},R_{i_k})\) and the map
\(\psi_{i_k}\) is a \((R_{i_k},\varepsilon_{i_k})\)-approximation,
we have that \(\psi_{i_k}\big(B(x_{i_k},r/4)\big)\subset
B\big(\psi_{i_k}(x_{i_k}),r/4+\varepsilon_{i_k}\big)\subset
B\big(\psi_{i_k}(x_{i_k}),r/2\big)\), which in turn implies that
the inclusion \(B(x_{i_k},r/4)\subset\psi_{i_k}^{-1}
\big(B\big(\psi_{i_k}(x_{i_k}),r/2\big)\big)\) holds.
Also, \(B\big(\psi_{i_k}(x_{i_k}),r/2\big)\subset\bar B(x,r)\)
is satisfied for every \(k\in\N\). All in all, since
\((\psi_{i_k})_*\mm_{i_k}\rightharpoonup\mm_\infty\) weakly, we obtain that
\[\begin{split}
0&<\frac{\lambda}{2}\leq
\lims_{k\to\infty}\mm_{i_k}\big(B(x_{i_k},r/4)\big)
\leq\lims_{k\to\infty}(\psi_{i_k})_*\mm_{i_k}
\big(B\big(\psi_{i_k}(x_{i_k}),r/2\big)\big)\\
&\leq\lims_{k\to\infty}(\psi_{i_k})_*\mm_{i_k}\big(\bar B(x,r)\big)
\leq\mm_\infty\big(\bar B(x,r)\big).
\end{split}\]
By arbitrariness of \(r>0\), we can conclude that \(x\in{\rm spt}(\mm_\infty)\),
which yields \eqref{eq:pmGH_vs_UL_aux10}. Finally, it follows from
\eqref{eq:pmGH_vs_UL_aux10} and Corollary \ref{cor:pGH_vs_UL_proper}
that \([[x_i]]\in Y\). Hence, we showed that \({\rm spt}(\mm_\omega)\subset Y\).\\
{\color{blue}\textsc{Step 3.}} The next step is to prove that
\((\phi_\infty)_*\mm_\infty(B)=\mm_\omega(B)\) is satisfied for every given Borel set
\(B\subset\phi_\infty(X_\infty)\). By inner regularity of \((\phi_\infty)_*\mm_\infty\)
and \(\mm_\omega\), it is sufficient to show that
\begin{equation}\label{eq:claim_incl_gen}
\mm_\omega\big(\phi_\infty(K)\big)=\mm_\infty(K),
\quad\text{ for every }K\subset X_\infty\text{ compact.}
\end{equation}
Fix any \(\varepsilon>0\). Then we can find \(0<\delta<\delta'<\varepsilon\) such that
\(\mm_\infty(\partial U)=\mm_\infty(\partial V)=0\)
and \(\mm_\infty(V)\leq\mm_\infty(K)+\varepsilon\), where \(U\)
and \(V\) stand for the \(\delta\)-neighbourhood and the \(\delta'\)-neighbourhood
of \(K\), respectively. Call \(C_\varepsilon\) the closure of \(U\). Note
that \(K\subset U\subset C_\varepsilon\subset V\). We claim that
\begin{equation}\label{eq:claim_incl}
\pi^{-1}\big(\phi_\infty(X_\infty)\big)\cap\Pi_{i\to\omega}\psi_i^{-1}(U)\subset
\pi^{-1}\big(\phi_\infty(C_\varepsilon)\big)\subset\Pi_{i\to\omega}\psi_i^{-1}(V).
\end{equation}
First of all, let
\([x_i]\in\pi^{-1}\big(\phi_\infty(X_\infty)\big)\cap\Pi_{i\to\omega}\psi_i^{-1}(U)\)
be given. Consider the point \(x\in X_\infty\) for which
\(\big[\big[\phi_i(x)\big]\big]=\phi_\infty(x)=[[x_i]]\). We have
that \(\phi_i(x)\in\psi_i^{-1}(U)\) for \(\omega\)-a.e.\ \(i\),
thus in particular \((\psi_i\circ\phi_i)(x)\in C_\varepsilon\) for \(\omega\)-a.e.\ \(i\).
Since \(\lim_{i\to\omega}d_\infty\big((\psi_i\circ\phi_i)(x),x\big)=0\) and
\(C_\varepsilon\) is closed, we conclude that \(x\in C_\varepsilon\) as well. Therefore,
we have that \(\pi\big([x_i]\big)=[[x_i]]=\phi_\infty(x)\in\phi_\infty(C_\varepsilon)\).

Suppose \([y_i]\in\pi^{-1}\big(\phi_\infty(C_\varepsilon)\big)\).
Then there exists \(y\in C_\varepsilon\) for which \(\big[\big[\phi_i(y)\big]\big]
=\phi_\infty(y)=[[y_i]]\). Given that \(\lim_{i\to\omega}d_\infty\big(
(\psi_i\circ\phi_i)(y),y\big)=0\) and \(V\) is an open neighbourhood of \(y\),
we know that \(\psi_i(y_i)=(\psi_i\circ\phi_i)(y)\in V\) for \(\omega\)-a.e.\ \(i\).
This means that \([y_i]\in\Pi_{i\to\omega}\psi_i^{-1}(V)\), as desired.
\medskip

Recall that \(\mm_\omega\) is concentrated on \({\rm spt}(\mm_\omega)\) by
Proposition \ref{prop:m_omega_conc_spt} and that
\({\rm spt}(\mm_\omega)\subset\phi_\infty(X_\infty)\) by \textsc{Steps} 1 and 2.
Hence, \eqref{eq:claim_incl} gives
\begin{equation}\label{eq:conseq_claim_incl}\begin{split}
\lim_{i\to\omega}(\psi_i)_*\mm_i(U)&=\bar\mm_\omega\big(\Pi_{i\to\omega}\psi_i^{-1}(U)\big)
=\bar\mm_\omega\Big(\pi^{-1}\big(\phi_\infty(X_\infty)\big)
\cap\Pi_{i\to\omega}\psi_i^{-1}(U)\Big)\\
&\leq\mm_\omega\big(\phi_\infty(C_\varepsilon)\big)\leq
\bar\mm_\omega\big(\Pi_{i\to\omega}\psi_i^{-1}(V)\big)=
\lim_{i\to\omega}(\psi_i)_*\mm_i(V).
\end{split}\end{equation}
Given that \(U\) and \(V\) have \(\mm_\infty\)-negligible boundary,
we know that \(\lim_{i\to\infty}(\psi_i)_*\mm_i(U)=\mm_\infty(U)\)
and \(\lim_{i\to\infty}(\psi_i)_*\mm_i(V)=\mm_\infty(V)\). Therefore,
we deduce from \eqref{eq:conseq_claim_incl} that
\[\begin{split}
\mm_\omega\big(\phi_\infty(C_\varepsilon)\big)&\leq\lim_{i\to\infty}(\psi_i)_*\mm_i(V)
=\mm_\infty(V)\leq\mm_\infty(U)+\varepsilon=
\lim_{i\to\infty}(\psi_i)_*\mm_i(U)+\varepsilon\\
&\leq\mm_\omega\big(\phi_\infty(C_\varepsilon)\big)+\varepsilon.
\end{split}\]
This implies that
\begin{equation}\label{eq:same_meas}
\Big|\mm_\omega\big(\phi_\infty(C_\varepsilon)\big)-\mm_\infty(C_\varepsilon)\Big|
\leq\varepsilon.
\end{equation}
Now fix any sequence \(\varepsilon_j\searrow 0\). We can assume without loss
of generality that \(C_{\varepsilon_{j+1}}\subset C_{\varepsilon_j}\) for all \(j\in\N\).
Since \(K=\bigcap_{j\in\N}C_{\varepsilon_j}\) by construction, we infer that
\(\phi_\infty(K)=\bigcap_{j\in\N}\phi_\infty(C_{\varepsilon_j})\),
thus the continuity from above of \(\mm_\omega\) and \(\mm_\infty\) gives
\[
\Big|\mm_\omega\big(\phi_\infty(K)\big)-\mm_\infty(K)\Big|=\lim_{j\to\infty}
\Big|\mm_\omega\big(\phi_\infty(C_{\varepsilon_j})\big)-\mm_\infty(C_{\varepsilon_j})\Big|
\overset{\eqref{eq:same_meas}}=0.
\]
This shows that \(\mm_\omega\big(\phi_\infty(K)\big)=\mm_\infty(K)\),
thus proving \eqref{eq:claim_incl_gen}.\\
{\color{blue}\textsc{Step 4.}} In order to achieve the statement, it only
remains to show that \(Y\subset{\rm spt}(\mm_\omega)\). Fix any point \(y\in Y\).
Take \(x\in{\rm spt}(\mm_\infty)\) with \(y=\phi_\infty(x)\).
Fix \(r>0\). Given that \(\mm_\infty\big(B(x,r)\big)>0\) and
\(B(x,r)\subset\phi_\infty^{-1}\big(B(y,r)\big)\) -- the latter holds
because \(\phi_\infty\) is an isometry -- we have that
\[\begin{split}
\mm_\omega\big(B(y,r)\big)&=\mm_\omega\big(B(y,r)\cap Y\big)
=\mm_\infty\big(\phi_\infty^{-1}\big(B(y,r)\big)\cap\phi_\infty^{-1}(Y)\big)\\
&=\mm_\infty\big(\phi_\infty^{-1}\big(B(y,r)\big)\cap{\rm spt}(\mm_\infty)\big)
=\mm_\infty\big(\phi_\infty^{-1}\big(B(y,r)\big)\big)\\
&\geq\mm_\infty\big(B(x,r)\big)>0.
\end{split}\]
By arbitrariness of \(r>0\), we conclude that \(y\in{\rm spt}(\mm_\omega)\),
proving the inclusion \(Y\subset{\rm spt}(\mm_\omega)\).
Consequently, the proof is complete.
\end{proof}
\chapter{Weak pointed measured Gromov--Hausdorff convergence}
\section{Definition of wpmGH convergence}
In this section, we will give a definition of a weaker version of pointed measured Gromov--Hausdorff convergence, and investigate its connection with the notion of ultralimit of metric measure spaces. We will generalise Gromov's compactness theorem to this setting, and with that, we will show that this new notion of convergence is in fact equivalent to the notion of the so-called pointed measured Gromov convergence, introduced and studied in \cite{Gigli-Mondino-Savare}.

\begin{definition}[Weak $(R,\varepsilon)$-approximation]\label{def:weakappr}
Let \((X,d_X,\mm_X,p_X)\) and \((Y,d_Y,\mm_Y,p_Y)\) be pointed
Polish metric measure spaces. Let \(R,\varepsilon>0\) be such
that \(\varepsilon<R\). Then a given Borel map \(\psi\colon B(p_X,R)\to Y\)
is said to be a \emph{weak \((R,\varepsilon)\)-approximation} provided
it satisfies \(\psi(p_X)=p_Y\), and there exists a Borel set
$\tilde X\subset B(p_X,R)$ containing \(p_X\) such that
\begin{enumerate}
\item \label{wyy}$d_X(x,y)-\varepsilon\leq d_Y\big(\psi(x),\psi(y)\big)\leq d_X(x,y)+\varepsilon$ for every $x,y\in\tilde X$,
\item \label{wkaa}$\mm_X\big(B(p_X,R)\setminus \tilde X\big)\le\varepsilon$ and
$\mm_Y\big(B(p_Y,R-\varepsilon)\setminus\psi(\tilde X)^\varepsilon\big)\leq \varepsilon$.
\end{enumerate}
\end{definition}
Remark \ref{rmk:quasi_isom} applies also here: the map $\psi$ can be viewed as a map defined globally on the whole space $X$. Notice that the definition is weaker than but analogous to Definition \ref{def:REappr}; the condition \eqref{wyy} says that the map $\psi$ is $(1,\varepsilon)$-quasi-isometric in the ball of radius $R$ after neglecting a small set in measure, while the condition \eqref{wkaa} states that the map is roughly surjective up to a small error in measure.
\medskip

The following lemma grants that also in the setting of weak $(R,\varepsilon)$-approximations, we have the existence of rough inverses. Its proof is very
similar to the one of Lemma \ref{lem:rough_inverse}; we sketch the
argument for completeness.
\begin{lemma}\label{lem:rough_inverse_weak_approx}
Let \((X,d_X,\mm_X,p_X)\), \((Y,d_Y,\mm_Y,p_Y)\) be pointed Polish
metric measure spaces. Let \(\psi\colon X\to Y\) be a given
weak \((R,\varepsilon)\)-approximation, for some \(R,\varepsilon>0\)
satisfying \(4\varepsilon<R\). Then there exists a weak
\((R-\varepsilon,3\varepsilon)\)-approximation \(\phi\colon Y\to X\)
-- that we will call a \emph{rough inverse}
of \(\psi\) -- such that
\begin{subequations}
\begin{align}\label{eq:wquasi-inverse_1}
d_X\big(x,(\phi\circ\psi)(x)\big)<3\varepsilon,&
\quad\text{ for every }x\in B(p_X,R-4\varepsilon)\cap\tilde X,\\
\label{eq:wquasi-inverse_2}
d_Y\big(y,(\psi\circ\phi)(y)\big)<3\varepsilon,&
\quad\text{ for every }y\in B(p_Y,R-\varepsilon)\cap \tilde Y.
\end{align}\end{subequations}
\end{lemma}
\begin{proof}
Define $\tilde Y\coloneqq \psi(\tilde X)^\varepsilon\cap B(p_Y,R-\varepsilon)$.
Let $\{x_i\}_{i\in\N}$ be dense in \(\tilde X\).
Given any \(y\in \tilde Y\setminus\{p_Y\}\), define
\(\phi(y)\coloneqq x_{i_y}\), where \(i_y\) is the smallest index
\(i\in\N\) for which \(d_Y\big(\psi(x_i),y\big)<\varepsilon\).
We also define \(\phi(p_Y)\coloneqq p_X\), while we set
\(\phi(y)\coloneqq x_0\) for all \(y\in Y\setminus\tilde Y\),
where \(x_0\in X\setminus B(p_X,R)\) is any given point
(when \(X\setminus B(p_X,R)\neq\emptyset\), otherwise \(x_0\coloneqq p_X\)).
The resulting map \(\phi\colon Y\to X\) is Borel
and satisfies \(\phi(p_Y)=p_X\) by construction. For every
\(y,y'\in \tilde Y\), we have that
\[\begin{split}
&\Big|d_X\big(\phi(y),\phi(y')\big)-d_Y(y,y')\Big|\\
\leq\,&\Big|d_X\big(\phi(y),\phi(y')\big)-
d_Y\big((\psi\circ\phi)(y),(\psi\circ\phi)(y')\big)\Big|+\Big|
d_Y\big((\psi\circ\phi)(y),(\psi\circ\phi)(y')\big)-d_Y(y,y')\Big|\\
\leq\,&\varepsilon+d_Y\big((\psi\circ\phi)(y),y\big)
+d_Y\big((\psi\circ\phi)(y'),y'\big)<3\varepsilon,
\end{split}\]
which shows that \(\sup_{y,y'\in \tilde Y}
\big|d_X\big(\phi(y),\phi(y')\big)-d_Y(y,y')\big|\leq 3\varepsilon\).
Moreover, given any point \(x\) in \(B(p_X,R-4\varepsilon)\cap\tilde X\),
we have that \(\psi(x)\in B(p_Y,R-3\varepsilon)\cap\tilde Y\), thus
from the validity of the inequalities
\begin{equation}\label{eq:aux_quasi-inverse_2}
d_X\big((\phi\circ\psi)(x),x\big)\leq
d_Y\big((\psi\circ\phi\circ\psi)(x),\psi(x)\big)+\varepsilon
<2\varepsilon<3\varepsilon
\end{equation}
it follows that \(x\) belongs to the \(3\varepsilon\)-neighbourhood
of \(\phi\big(B(p_Y,R-3\varepsilon)\cap\tilde Y\big)\). In particular,
it holds that \(B(p_X,R-4\varepsilon)\setminus\phi(\tilde Y)^{3\varepsilon}\subset B(p_X,R-4\varepsilon)\setminus \tilde X\),
hence
\[\mm_X\big(B(p_X,R-4\varepsilon)\setminus\phi(\tilde Y)^{3\varepsilon}\big)\le\mm_X\big(B(p_X,R-4\varepsilon)\setminus\tilde X\big)\le\varepsilon.\]
Accordingly, \(\phi\) is a weak
\((R-\varepsilon,3\varepsilon)\)-approximation. In order to
conclude, it only remains to observe that \eqref{eq:quasi-inverse_2}
is a direct consequence of the very definition of \(\phi\),
while \eqref{eq:quasi-inverse_1} follows from the estimate
in \eqref{eq:aux_quasi-inverse}.
\end{proof}
\begin{definition}[Weak pointed measured Gromov--Hausdorff convergence]
Let us consider a sequence \(\big\{(X_i,d_i,\mm_i,p_i)\big\}_{i\in\bar\N}\)
of pointed Polish metric measure spaces. Then \((X_i,d_i,\mm_i,p_i)\)
is said to converge to \((X_\infty,d_\infty,\mm_\infty,p_\infty)\) in the
\emph{weak pointed measured Gromov--Hausdorff sense}\label{def:wpmGH}
(briefly, in the \emph{wpmGH sense}) as \(i\to\infty\) provided there exists
a sequence \((\psi_i)_{i\in\N}\) of weak \((R_i,\varepsilon_i)\)-approximations, for some \(R_i\nearrow+\infty\)
and \(\varepsilon_i\searrow 0\), so that
\[
(\psi_i)_*\mm_i\rightharpoonup \mm_\infty,\quad
\text{ in duality with }C_{bbs}(X_\infty).
\]
\end{definition}
\section{Relation between wpmGH convergence and ultralimits}
\label{s:wpmGH_vs_UL}
In the following proof of Gromov compactness theorem for weak pointed
measured Gromov--Hausdorff convergence, we will need the variant of Prokhorov
theorem for ultralimits introduced in Appendix \ref{s:Prokhorov}.
\begin{theorem}\label{thm:pre-Gromov}
Let $\big\{(X_i,d_i,\mm_i,p_i)\big\}_{i\in\N}$ be a given asymptotically
uniformly boundedly $\mm_\omega$-totally bounded sequence of pointed Polish
metric measure spaces. Then it holds that the ultralimit $\lim_{i\to\omega}(X_i,d_i,\mm_i,p_i)$ with respect to the wpmGH-topology exists and is isomorphic to $\big(\spt(\mm_\omega),d_\omega,\mm_\omega,p_\omega\big)$.
\end{theorem}
\begin{remark}{\rm
Technically speaking, we have given the definition of an ultralimit
of points only when the underlying space is a (pseudo)metric space.
The ultralimit in the above theorem can be regarded as such
\emph{after} proving the equivalence of wpmGH- and pmG-convergence
in Theorem \ref{thm:wpmGH_vs_pmG}, and noticing that the pmG-topology
is in fact metrisable. 

For the reader who is not comfortable with such an undirect
reasoning, we suggest to think of Theorem \ref{thm:pre-Gromov} more
directly in terms of weak \((R,\varepsilon)\)-approximations:
it simply states that there exist sequences
\(\tilde R_i,\tilde\varepsilon_i>0\) and weak
\((\tilde R_i,\tilde\varepsilon_i)\)-approximations \(\psi_i\) from
\(X_i\) to \(\spt(\mm_\omega)\) so that \(\lim_{i\to\omega}R_i=\infty\),
\(\lim_{i\to\omega}\varepsilon_i=0\), and
\[
\lim_{i\to\omega}(\psi_i)_*\mm_i=\mm_\omega.
\]
Here, the last ultralimit makes sense, since the weak convergence
is metrisable.
\fr}\end{remark}
\begin{proof}[Proof of Theorem \ref{thm:pre-Gromov}]
The proof goes along the same lines as the proof of the classical
Gromov's compactness theorem using ultralimits. We subdivide the proof into
several steps:\\
{\color{blue}\textsc{Step 1.}} Let $(X_i,d_i,\mm_i,p_i)$ be an asymptotically
uniformly boundedly $\mm_\omega$-totally bounded sequence of pointed Polish
metric measure spaces. Let $R_i\nearrow\infty$ and $\varepsilon_i\searrow 0$
be fixed. Let $\tilde Y_i\subset B(p_\omega,R_i)\cap\spt(\mm_\omega)$ be a
compact set, with $\tilde Y_i\subset \tilde Y_{i+1}$, and let $N_i\in \N$
be so that
\begin{equation}\label{eq:Gromov_wpmGH1}
\mm_\omega\big(B(p_\omega,R_i)\setminus \tilde Y_i\big)\le \varepsilon_i
\end{equation}
and
\begin{equation}\label{eq:Gromov_wpmGH2}
\tilde Y_i\subset \bigcup_{j=1}^{N_i}B(x^{i,j},\varepsilon_i)
\end{equation}
for some $(x^{i,j})_{j=1}^{N_i}\subset\tilde Y_i$ such that
$d_\omega(x^{i,j},x^{i,m})>\varepsilon_i$ whenever \(j,m=1,\ldots,N_i\)
satisfy \(j\neq m\). Furthermore, we can additionally require that for any
\(i<i'\) it holds that \(N_i\leq N_{i'}\) and \(x^{i',j}=x^{i,j}\) for every
\(j=1,\ldots,N_i\). For each $i,j$, we can write
\[x^{i,j}=[[x^{i,j}_k]]_k,\quad\text{ where }
x^{i,j}_k\in X_k\text{ for all }k\in\N.
\]
Observe that we can further require that \(x^{i,j}_k=x^{i',j}_k\)
for every \(i<i'\), \(j=1,\ldots,N_i\), and \(k\in\N\).
We now define for each $i\in \N$ the sets $\F_i$ recursively by setting
\(\F_1\coloneqq \mathcal{A}_1\cap\mathcal{B}_1\cap\mathcal{C}_1
\cap\mathcal{D}_1\) and
\begin{align}\F_{i+1}\coloneqq &\F_i\cap\mathcal{A}_{i+1}\cap\mathcal{B}_{i+1}\cap\mathcal{C}_{i+1}\cap\mathcal{D}_{i+1},
\end{align}
where we define
\begin{align}
\mathcal{A}_{i}&\coloneqq\Big\{k\ge i\;\Big|\;\mm_k\big(B(p_k,R_i)\setminus {\textstyle\bigcup_{j=1}^{N_i}}B(x^{i,j}_k,2\varepsilon_i)\big)
\le 2\varepsilon_i\Big\},\\
\mathcal{B}_{i}&\coloneqq\Big\{k\in\N\;\Big|\;\big|d_k(x^{i,j}_k,x^{i,m}_k)
-d_\omega(x^{i,j},x^{i,m})\big|<\varepsilon_i\text{ for all }j,m\leq N_i\Big\},\\
\mathcal{C}_{i}&\coloneqq\Big\{k\in\N\;\Big|\;\big|d_k(x^{i,j}_k,p_k)
-d_\omega(x^{i,j},p_\omega)\big|<\varepsilon_i\text{ for all }j\leq N_i\big\},\\
\mathcal{D}_i&\coloneqq\Big\{k\in\N\;\Big|\;d_k(x^{i,j}_k,x^{i,m}_k)
>\varepsilon_i\text{ for all }j,m\leq N_i\text{ with }j\neq m\Big\}.
\end{align}
By asymptotic uniform bounded $\mm_\omega$-total boundedness (and by definition
of ultralimit), we have that $\F_i\in\omega$ holds for every $i\in\N$.
For every $k\in\F_1$, there exists a unique $i_k\in\N$ so that
$k\in \F_{i_k}\setminus \F_{i_k+1}$. For $k\in\F_1$, define the Borel map
$\psi_k\colon X_{k}\to X_\omega$ as follows: $\psi_k(p_k)\coloneqq p_\omega$,
\[\begin{split}
\psi_k(x)\coloneqq x^{i_k,1},&\quad\text{ for every }x\in B(x^{i_k,1}_k,
2\varepsilon_{i_k}),\\
\psi_k(x)\coloneqq x^{i_k,j+1},&\quad\text{ for every }x\in
B(x^{i_k,j+1}_k,2\varepsilon_{i_k})\setminus{\textstyle\bigcup_{n=1}^j}
B(x^{i_k,j}_k,2\varepsilon_{i_k}),
\end{split}\]
while we can assume without loss of generality that
\begin{equation}\label{eq:Gromov_wpmGH3}
\psi_k\big(X_k\setminus{\textstyle{\bigcup}_{j=1}^{N_{i_k}}}
B(x^{i_k,j}_k,2\varepsilon_{i_k})\big)\subset
X_\omega\setminus B(p_\omega,R_{i_k}),
\end{equation}
provided the latter set is not empty. Put
$\tilde R_k\coloneqq R_{i_k}$ and $\tilde \varepsilon_k\coloneqq 
6\varepsilon_{i_k}$ for all \(k\in\N\). Then we have (by Lemma
\ref{lma:helppo}) that
\(\lim_{k\to\omega}\tilde R_k=\infty\) and
\(\lim_{k\to\omega}\tilde \varepsilon_k=0\).\\
{\color{blue}\textsc{Step 2.}} We claim that
\begin{equation}\label{eq:Gromov_wpmGH4}
\psi_k\;\text{ is a weak }(\tilde R_k,\tilde\varepsilon_k)
\text{-approximation,}\quad\text{ for every }k\in\mathcal F_1.
\end{equation}
In order to prove condition (1) of Definition \ref{def:weakappr}, define
\[
\tilde X_k\coloneqq \bigcup_{j=1}^{N_{i_k}}B(x^{i_k,j}_k,2\varepsilon_{i_k}).
\]
Let the points  $x,y\in \tilde X_k$ be fixed. Then we have that
$x\in B(x^{i_k,j_x+1}_k,2\varepsilon_{i_k})
\setminus \bigcup_{j=1}^{j_x}B(x^{i_k,j}_k,2\varepsilon_{i_k})$ and
$y\in B(x^{i_k,j_y+1}_k,2\varepsilon_{i_k})\setminus \bigcup_{j=1}^{j_y}
B(x^{i_k,j}_k,2\varepsilon_{i_k})$ for some $j_x,j_y\le N_{i_k}-1$.
If $j_x=j_y$, we have
\begin{align}
d_\omega\big(\psi_k(x),\psi_k(y)\big)-
\tilde\varepsilon_k=-6\varepsilon_{i_k}\le
d_k(x,y)\le 4\varepsilon_{i_k}\le d_\omega\big(\psi_k(x),\psi_k(y)\big)
+\tilde\varepsilon_k.
\end{align}
If $j_x\neq j_y$, we have
\begin{align}
d_\omega\big(\psi_k(x),\psi_k(y)\big)-\tilde\varepsilon_k&=
d_\omega(x^{i_k,j_x+1},x^{i_k,j_y+1})-6\varepsilon_{i_k}\le
d_k(x^{i_k,j_x+1}_k,x^{i_k,j_y+1}_k)-5\varepsilon_{i_k}
\\ &\le d_k(x^{i_k,j_x+1}_k,x)+d_k(x,y)+d_k(y,x^{i_k,j_y+1}
_k)-5\varepsilon_{i_k}\le d_k(x,y)
\\ &\le d_k(x^{i_k,j_x+1}_k,x^{i_k,j_y+1}_k)+4\varepsilon_{i_k}\le
d_\omega\big(\psi_k(x),\psi_k(y)\big)+5\varepsilon_{i_k}
\\ &\le d_\omega\big(\psi_k(x),\psi_k(y)\big)+\tilde\varepsilon_{k}.
\end{align}
Hence, the property (1) is proven. For the property (2), first notice
that the inequality
\[
\mm_k\big(B(p_k,\tilde R_k)\setminus \tilde X_k\big)\le \tilde\varepsilon_k
\]
is true due to the definitions of $\tilde X_k$ and $\mathcal F_{i_k}$.
Moreover, given that
\(\{x^{i_k,j}\}_{j=1}^{N_{i_k}}\subset\psi_k(\tilde X_k)\)
by the very construction of \(\psi_k\), we deduce from
\eqref{eq:Gromov_wpmGH2} that \(\tilde Y_{i_k}\subset
\bigcup_{j=1}^{N_{i_k}}B(x^{i_k,j},\tilde\varepsilon_k)
\subset\psi_k(\tilde X_k)^{\tilde\varepsilon_k}\), whence
it follows from \eqref{eq:Gromov_wpmGH1} that
\(\mm_\omega\big(B(p_\omega,\tilde R_k)\setminus\psi_k(\tilde X_k)
^{\tilde\varepsilon_k}\big)\leq\tilde\varepsilon_k\),
yielding (2). All in all, we have proven that $\psi_k$ is a weak
$(\tilde R_k,\tilde \varepsilon_k)$-approximation for every $k\in\F_1$,
\emph{i.e.}, the claim \eqref{eq:Gromov_wpmGH3}.\\
{\color{blue}\textsc{Step 3.}} Next, we claim that
there exists a Borel measure \(\tilde\mm\geq 0\)
on \(\spt(\mm_\omega)\) such that
\begin{equation}\label{eq:Gromov_wpmGH5}
\tilde\mm=\lim_{k\to\omega}(\psi_k)_*\mm_k,
\quad\text{ in duality with }C_{bbs}\big(\spt(\mm_\omega)\big).
\end{equation}
Let \(n\in\N\) and \(\varepsilon>0\) be fixed. Pick \(k_0\in\N\)
such that \(\tilde R_k\geq\tilde R_{k_0}\), \(\tilde R_k-\tilde\varepsilon_k\geq n\), and
\(2\tilde\varepsilon_k\leq\varepsilon\) for all
\(k\in\mathcal F_{i_{k_0}}\) with \(k\geq k_0\). Call
\(\mu_k^n\coloneqq\big((\psi_k)_*\mm_k\big)
\llcorner\big(B(p_\omega,n)\cap\spt(\mm_\omega)\big)\) for every
\(k\in\mathcal F_{i_{k_0}}\) with \(k\geq k_0\) and
\(T\coloneqq B(\tilde Y_{i_{k_0}},2\varepsilon_{i_{k_0}})\).
Since \(\tilde Y_{i_{k_0}}\) is compact, one has that \(T\)
is \(4\varepsilon_{i_{k_0}}\)-totally bounded, thus in particular
it is \(\varepsilon\)-totally bounded. Moreover, we claim that
\begin{equation}\label{eq:Gromov_wpmGH6}
\psi_k\bigg(\bigcup_{j=1}^{N_{i_{k_0}}}B\big(x^{i_{k_0},j}_k,
2\varepsilon_{i_{k_0}}\big)\bigg)\subset T,\quad\text{ for every }
k\in\mathcal F_{i_{k_0}}\text{ with }k\geq k_0.
\end{equation}
In order to prove it, fix any \(j=1,\ldots,N_{i_{k_0}}\) and
\(x\in\tilde X_k\cap B\big(x^{i_{k_0},j}_k,2\varepsilon_{i_{k_0}}\big)
=B\big(x^{i_{k_0},j}_k,2\varepsilon_{i_{k_0}}\big)\).
Being the map \(\psi_k\) a weak \((\tilde R_k,\tilde\eps_k)\)-approximation, we deduce that
\[
d_\omega\big(\psi_k(x),\psi_k(x^{i_{k_0},j}_k)\big)
\leq d_k(x,x^{i_{k_0},j}_k)+\tilde\varepsilon_k
<2\varepsilon_{i_{k_0}}+\tilde\varepsilon_k\leq 2\tilde\varepsilon_{k_0}.
\]
Given that \(\psi_k(x^{i_{k_0},j}_k)=\psi_k(x^{i_k,j}_k)
=x^{i_k,j}=x^{i_{k_0},j}\in\tilde Y_{i_{k_0}}\), we have proven that
\(\psi_k(x)\in T\), whence \eqref{eq:Gromov_wpmGH6} follows. Therefore,
since for any \(k\in\mathcal F_{i_{k_0}}\) with \(k\geq k_0\) it holds that
\[\begin{split}
\psi_k^{-1}\Big(\big(B(p_\omega,n)\cap\spt(\mm_\omega)\big)\setminus T\Big)
&\overset{\eqref{eq:Gromov_wpmGH6}}\subset
\psi_k^{-1}\big(B(p_\omega,\tilde R_{k_0}-\tilde\varepsilon_{k_0})
\big)\setminus\bigcup_{j=1}^{N_{i_{k_0}}}B\big(x^{i_{k_0},j}_k,
2\varepsilon_{i_{k_0}}\big)\\
&\overset{\phantom{\eqref{eq:Gromov_wpmGH6}}}\subset
B(p_k,\tilde R_{k_0})\setminus\bigcup_{j=1}^{N_{i_{k_0}}}
B\big(x^{i_{k_0},j}_k,2\varepsilon_{i_{k_0}}\big),
\end{split}\]
we conclude that
\[\begin{split}
\mu_k^n\big(\spt(\mm_\omega)\setminus T\big)&=
\mm_k\Big(\psi_k^{-1}\Big(\big(B(p_\omega,n)\cap\spt(\mm_\omega)\big)
\setminus T\Big)\Big)\\
&\leq\mm_k\Big(B(p_k,\tilde R_k)\setminus\bigcup_{j=1}^{N_{i_{k_0}}}
B\big(x^{i_{k_0},j}_k,2\varepsilon_{i_{k_0}}\big)\Big)\\
&\leq 2\varepsilon_{i_{k_0}}\leq\varepsilon.
\end{split}\]
In particular, \(\mu_k^n\big(\spt(\mm_\omega)\setminus T\big)
\leq\varepsilon\) holds for \(\omega\)-a.e.\ \(k\). Hence,
by using Theorem \ref{thm:Prokhorov} we obtain that there
exists a finite Borel measure \(\tilde\mm_n\) on \(\spt(\mm_\omega)\)
for which \(\tilde\mm_n=\lim_{k\to\omega}\mu_k^n\) with respect to
the weak convergence. It can be readily checked that \(\tilde\mm_n=
\tilde\mm_{n'}\llcorner\big(B(p_\omega,n)\cap\spt(\mm_\omega)\big)\)
whenever \(n<n'\), so that it makes sense to define the measure
\(\tilde\mm\) on \(\spt(\mm_\omega)\) as
\[
\tilde\mm(E)\coloneqq\lim_{n\to\infty}\tilde\mm_n\big(E\cap B(p_\omega,n)
\big),\quad\text{ for every }E\in\mathscr B\big(\spt(\mm_\omega)\big).
\]
Finally, given any \(f\in C_{bbs}\big(\spt(\mm_\omega)\big)\) and
\(n\in\N\) with \(\spt(f)\subset B(p_\omega,n)\), we have that
\[
\int f\,\d\tilde\mm=\int f\,\d\tilde\mm_n=\lim_{k\to\omega}
\int f\,\d\mu_k^n=\lim_{k\to\omega}\int f\,\d(\psi_k)_*\mm_k.
\]
Therefore, we have proven that \eqref{eq:Gromov_wpmGH5} is satisfied,
as desired.\\
{\color{blue}\textsc{Step 4.}}
In order to conclude, it only remains to show that there exists
an isometric bijection \(\phi_\infty\colon\spt(\mm_\omega)\to
\spt(\mm_\omega)\) such that \((\phi_\infty)_*\tilde\mm=
\mm_\omega|_{\mathscr B(\spt(\mm_\omega))}\). Since the arguments
are very similar to those in the proofs of the results in
Section \ref{ss:relation_with_pmGH}, we will omit some details.

First of all, we define \(\phi_\infty\colon\spt(\mm_\omega)\to
\spt(\mm_\omega)\). Given any $y\in\spt(\mm_\omega)$, there exist
$y_i\in\tilde Y_i$ so that $y=\lim_{i\to\omega}y_i$. Then let us define
\(\phi_\infty(y)\coloneqq\big[\big[\phi_i(y_i)\big]\big]_i\in X_\omega\),
where \(\phi_i\) is a rough inverse of \(\psi_i\) (whose existence is
granted by Lemma \ref{lem:rough_inverse_weak_approx}). The map
\(\phi_\infty\) is well-defined and isometric by Definition
\ref{def:wpmGH} of wpmGH-convergence and by (2) of Definition
\ref{def:weakappr}. Moreover, given any point \(x=\phi_\infty(y)
\in\phi_\infty\big(\spt(\mm_\omega)\big)\), we define
\[
\psi_\infty(x)\coloneqq\lim_{i\to\omega}\psi_i(x_i)\in\spt(\mm_\omega),
\]
where \(x=[[x_i]]_i\) for some \(x_i\in\tilde X_i\). Then \(\psi_\infty\)
is well-defined and is the inverse of \(\psi_\infty\), since
\[
\lim_{i\to\omega}d_\omega\big(y,\psi_i(x_i)\big)
\leq\lim_{i\to\omega}\Big[d_\omega(y,y_i)+
d_\omega\big(y_i,(\psi_i\circ\phi_i)(y_i)\big)+
d_\omega\big((\psi_i\circ\phi_i)(y_i),\psi_i(x_i)\big)\Big]=0.
\]
As in \textsc{Step 1} of the proof of Theorem \ref{thm:pmGH_vs_UL},
we get $\spt(\mm_\omega)\subset\phi_\infty\big(\spt(\tilde\mm)\big)
\subset\phi_\infty\big(\spt(\mm_\omega)\big)$. We will show that
$\mm_\omega\llcorner\phi_\infty\big(\spt(\mm_\omega)\big)=
(\phi_\infty)_*\tilde\mm$. Since $\tilde \mm$ and $\mm_\omega$ are both
boundedly finite, it suffices to prove that $\tilde \mm(K)=\mm_\omega
\big(\phi_\infty(K)\big)$ for every $K\subset\spt(\mm_\omega)$ compact.
Let $K\subset\spt(\mm_\omega)$ be a compact set and $\varepsilon>0$.
Then there exist $0<\delta<\delta'<\varepsilon$ so that
$\tilde\mm(\partial U)=\tilde\mm(\partial V)=0$ and
$\tilde\mm(V)\le\tilde\mm(K)+\varepsilon$, where $U$ and $V$ are the
$\delta$- and the $\delta'$-neighbourhood of $K$, respectively.
Denote by $C_\varepsilon$ the closure of $U$. Notice that
$K\subset U\subset C_\varepsilon\subset V$. Let us show that
\begin{align} \label{tuplaosa}
\pi^{-1}\big(\phi_\infty\big(\spt(\mm_\omega)\big)\big)\cap
\Pi_{i\to\omega}\psi_i^{-1}(U)\subset \pi^{-1}\big(
\phi_\infty(C_\varepsilon)\big)\subset\Pi_{i\to\omega}\psi_i^{-1}(V).
\end{align}
Let $[x_i]\in\Pi_{i\to\omega}\psi_i^{-1}(U)$ be so that
$[[x_i]]=\phi_\infty(y)$ for some $y\in\spt(\mm_\omega)$. Then
$\phi_i(y_i)\in\psi_i^{-1}(U)$ for $\omega$-almost every $i\in\N$,
where $\phi_\infty(y)=\big[\big[\phi_i(y_i)\big]\big]$. Thus,
$(\psi_i\circ\phi_i)(y_i)\in C_\varepsilon$ for $\omega$-almost every
\(i\in\N\), and therefore, since $C_\varepsilon$ is closed, we have that
\[
y=\lim_{i\to\omega}\psi_i(x_i)=\lim_{i\to\omega}
(\psi_i\circ\phi_i)(y_i)\in C_\varepsilon.
\]
Thus, we have proven that
$[x_i]\in\pi^{-1}\big(\phi_\infty(C_\varepsilon)\big)$.
Now suppose that \([z_i]\in\pi^{-1}\big(\phi_\infty(C_\varepsilon)\big)\).
Then there exists \(z\in C_\varepsilon\) for which
\(\big[\big[\phi_i(z)\big]\big]=\phi_\infty(z)=[[z_i]]\). Since
\(\lim_{i\to\omega}d_\infty\big((\psi_i\circ\phi_i)(z),z\big)=0\) and
\(V\) is an open neighbourhood of \(z\), we know that
\(\psi_i(z_i)=(\psi_i\circ\phi_i)(z)\in V\) for \(\omega\)-a.e.\ \(i\in\N\).
This means that \([y_i]\in\Pi_{i\to\omega}\psi_i^{-1}(V)\), as desired.
All in all, the claim \eqref{tuplaosa} is proven.

Since $\mm_\omega$ is concentrated on $\spt(\mm_\omega)$, and
$\spt(\mm_\omega)\subset\phi_\infty\big(\spt(\mm_\omega)\big)$,
we have by \eqref{tuplaosa} that
\begin{equation}\label{epayht}\begin{split}
\lim_{i\to\omega}(\psi_i)_*\mm_i(U)&=
\bar\mm_\omega\big(\Pi_{i\to\omega}\psi^{-1}(U)\big)=
\bar\mm_\omega\big(\pi^{-1}\big(\phi_\infty(\spt(\mm_\omega))\big)\cap
\Pi_{i\to\omega}\psi^{-1}(U)\big)\\
&\le\mm_\omega\big(\psi_\infty(C_\varepsilon)\big)\le
\bar\mm_\omega\big(\Pi_{i\to\omega}\psi_i^{-1}(V)\big)=
\lim_{i\to\omega}(\psi_i)_*\mm_i(V).
\end{split}\end{equation}
Due to the fact that $U$ and $V$ are continuity sets for $\tilde\mm$,
we obtain from \eqref{epayht} that
\begin{align}
\mm_\omega\big(\phi_\infty(C_\varepsilon)\big)&
\le\lim_{i\to\omega}(\psi_i)_*\mm_i(V)=\tilde\mm(V)\le
\tilde\mm(K)+\varepsilon\le\tilde\mm(U)+\varepsilon\\
&=\lim_{i\to\omega}(\psi_i)_*\mm_i(U)+\varepsilon\le
\mm_\omega\big(\phi_\infty(C_\varepsilon)\big)+\varepsilon.
\end{align}
By using the continuity from above of $\mm_\omega$ and letting
$\varepsilon\to 0$, we obtain $\tilde \mm(K)=\mm_\omega
\big(\phi_\infty(K)\big)$.

Finally, we conclude by observing that, arguing as in \textsc{Step 4}
of the proof of Theorem \ref{thm:pmGH_vs_UL}, we get that
$\phi_\infty\big(\spt(\tilde\mm)\big)\subset\spt(\mm_\omega)$.
Consequently, the statement is achieved.
\end{proof}
\begin{corollary}\label{cor:conseq_Gromov}
Let $\big\{(X_i,d_i,\mm_i,p_i)\big\}_{i\in\N}$ be a sequence of pointed Polish metric measure spaces that wpmGH-converges to a limit space
$(X_\infty,d_\infty,\mm_\infty,p_\infty)$. Then
$\big\{(X_i,d_i,\mm_i,p_i)\big\}_{i\in\N}$ is asymptotically
boundedly $\mm_\omega$-totally bounded and its wpmGH-limit
$(X_\infty,d_\infty,\mm_\infty,p_\infty)$ is isomorphic to
$\big(\spt(\mm_\omega),d_\omega,\mm_\omega,p_\omega\big)$.
\end{corollary}
\begin{proof}
Thanks to Theorem \ref{thm:pre-Gromov}, it suffices to prove that the
sequence $\big\{(X_i,d_i,\mm_i,p_i)\big\}_{i\in\N}$ is asymptotically
boundedly $\mm_\omega$-totally bounded. This can be easily achieved by suitably adapting the arguments in the
proof of Proposition \ref{prop:m_omega_conc_spt}; we omit the
details.
\end{proof}
Before stating the main result of this section, we need to introduce a
couple of definitions. Let \(\big\{(X_i,d_i,\mm_i,p_i)\big\}_{i\in\N}\)
be a sequence of pointed Polish metric measure spaces. Then we say that
\(\big\{(X_i,d_i,\mm_i,p_i)\big\}_{i\in\N}\) is \emph{uniformly boundedly
finite} provided for any \(R>0\) it holds
\[
\sup_{i\in\N}\mm_i\big(B(p_i,R)\big)<+\infty.
\]
In addition, we say that \(\big\{(X_i,d_i,\mm_i,p_i)\big\}_{i\in\N}\)
is \emph{boundedly measure-theoretically totally bounded} provided
for every \(R,r,\varepsilon>0\) there exist a number \(M\in\N\) and
points \((x^i_n)_{n=1}^M\subset X_i\) such that
\[
\sup_{i\in\N}\mm_i\big(\bar B(p_i,R)\setminus{\textstyle\bigcup
_{n=1}^M}B(x^i_n,r)\big)\leq\varepsilon.
\]
\begin{theorem}[Gromov's compactness for wpmGH-convergence]
\label{thm:Gromov_cpt_main}
Let $\big\{(X_i,d_i,\mm_i,p_i)\big\}_{i\in\N}$ be a sequence
of boundedly finite pointed Polish metric measure spaces. Then it holds that
$\big\{(X_i,d_i,\mm_i,p_i)\big\}_{i\in\N}$ is precompact with
respect to the wpmGH-convergence (\emph{i.e.}, any of its subsequences
has a wpmGH-converging subsequence) if and only if it is
uniformly boundedly finite and  boundedly measure-theoretically totally bounded.
\end{theorem}
\begin{remark}{\rm
We do not claim that the limit point $p_\infty$ is in the support of the limit measure nor that the limit measure is non-trivial. To guarantee that, one can obtain another compact class of metric measure spaces by requiring also that $\mm_i\big(B(p_i,R)\big)$ is uniformly bounded away from $0$ for every $R>0$.
\fr}\end{remark}
\begin{proof}[Proof of Theorem \ref{thm:Gromov_cpt_main}]\ \\
{\color{blue}\textsc{Necessity.}} Suppose
\(\big\{(X_i,d_i,\mm_i,p_i)\big\}_i\) is wpmGH-precompact. We argue by contradiction: suppose that there exists a subsequence
\(\big\{(X_{i_j},d_{i_j},\mm_{i_j},p_{i_j})\big\}_j\) that does
not admit any boundedly measure-theoretically totally bounded
subsequence.

Now we can pick any wpmGH-converging subsequence
\(\big\{(X_{i_{j_k}},d_{i_{j_k}},\mm_{i_{j_k}},p_{i_{j_k}})\big\}_k\),
which is asymptotically boundedly \(\mm_\omega\)-totally bounded
by Corollary \ref{cor:conseq_Gromov} and accordingly has a
boundedly measure-theoretically totally bounded subsequence.
This leads to a contradiction, thus 
$\big\{(X_i,d_i,\mm_i,p_i)\big\}_i$ is a
boundedly measure-theoretically totally bounded sequence, as desired.\\
{\color{blue}\textsc{Sufficiency.}} Suppose the sequence
$\big\{(X_i,d_i,\mm_i,p_i)\big\}_i$ is
boundedly measure-theoretically totally bounded. Fix an arbitrary
subsequence \(\big\{(X_{i_j},d_{i_j},\mm_{i_j},p_{i_j})\big\}_j\),
which is boundedly measure-theoretically totally bounded as well.
In particular, it is asymptotically boundedly \(\mm_\omega\)-totally
bounded, thus it has a wpmGH-converging subsequence by
Theorem \ref{thm:pre-Gromov}. This shows that
$\big\{(X_i,d_i,\mm_i,p_i)\big\}_i$ is wpmGH-precompact,
yielding the sought conclusion.
\end{proof}
\section{Equivalence between wpmGH convergence and pmG convergence}
We will use Theorem \ref{thm:Gromov_cpt_main} to prove that the weak pointed
measured Gromov--Hausdorff convergence is actually equivalent
to the so-called \emph{pointed measured Gromov convergence}
(briefly, \emph{pmG convergence}),
which was introduced in \cite{Gigli-Mondino-Savare}.
Amongst the several equivalent ways to define the pmG convergence,
we just need to recall the `extrinsic approach'
\cite[Definition 3.9]{Gigli-Mondino-Savare}.
\begin{definition}[Pointed measured Gromov convergence]
Let \(\big\{(X_i,d_i,\mm_i,p_i)\big\}_{i\in\bar\N}\) be a given sequence
of pointed Polish metric measure spaces. Then \((X_i,d_i,\mm_i,p_i)\)
is said to converge to \((X_\infty,d_\infty,\mm_\infty,p_\infty)\)
in the \emph{pointed measured Gromov sense} (briefly, in the
\emph{pmG sense}) provided there exist a Polish metric space
\((Y,d)\) and isometric embeddings \(\iota_i\colon X_i\to Y\),
\(i\in\bar\N\), such that
\(\lim_{i\to\infty}d\big(\iota_i(p_i),\iota_\infty(p_\infty)\big)=0\) and
\[
(\iota_i)_*\mm_i\rightharpoonup(\iota_\infty)_*\mm_\infty,
\quad\text{ in duality with }C_{bbs}(Y).
\]
\end{definition}
\begin{theorem}[wpmGH and pmG]\label{thm:wpmGH_vs_pmG}
Let $\big\{(X_i,d_i,\mm_i,p_i)\big\}_{i\in\bar\N}$
be given pointed Polish metric measure spaces. Then it holds that
\[
(X_i,d_i,\mm_i,p_i)\longrightarrow(X_\infty,d_\infty,\mm_\infty,p_\infty),
\quad\text{ in the wpmGH-sense}
\]
if and only if
\[
(X_i,d_i,\mm_i,p_i)\longrightarrow(X_\infty,d_\infty,\mm_\infty,p_\infty),
\quad\text{ in the pmG-sense.}
\]
\end{theorem}
\begin{proof}
Suppose $(X_i,d_i,\mm_i,p_i)$ is wpmGH-converging to
$(X_\infty,d_\infty,\mm_\infty,p_\infty)$. We aim to show that
$(X_i,d_i,\mm_i,p_i)\to(X_\infty,d_\infty,\mm_\infty,p_\infty)$
with respect to the pmG-topology. The proof goes along the same
lines as the proof of \cite[Proposition 3.30]{Gigli-Mondino-Savare}.
Let us define a distance $d$ on the set
$Y\coloneqq\bigsqcup_{i\in\bar\N}X_i$ so that the spaces $X_i$'s
are isometrically embedded into $Y$ and $\mm_i\rightharpoonup\mm_\infty$
weakly. Let $\psi_i\colon X_i\to X_\infty$ be the weak
$(R_i,\varepsilon_i)$-approximations given by the definition
of wpmGH-convergence. Then we define $d\colon Y\times Y\to[0,+\infty)$ as
\[
d(y_1,y_2)\coloneqq\left\{\begin{array}{ll}
d_i(y_1,y_2),\\
\Phi_i(y_1,y_2),\\
\Psi_{ij}(y_1,y_2),
\end{array}\quad\begin{array}{ll}
\text{ if }y_1,y_2\in X_i\text{ for some }i\in\bar\N,\\
\text{ if }y_1\in X_i\text{ for some }i\in\N\text{ and }y_2\in X_\infty,\\
\text{ if }y_1\in X_i\text{ and }y_2\in X_j\text{ for some }i,j\in\N
\text{ with }i\neq j,
\end{array}\right.
\]
where the functions \(\Phi_i\colon X_i\times X_\infty\to[0,+\infty)\)
and \(\Psi_{ij}\colon X_i\times X_j\to[0,+\infty)\) are given by
\[\begin{split}
\Phi_i(y_1,y_2)&\coloneqq\inf\Big\{d_i(x_i,y_1)
+d_\infty\big(\psi_i(x_i),y_2\big)+\eps_i\;\Big|\;x_i\in\tilde X_i\Big\},\\
\Psi_{ij}(y_1,y_2)&\coloneqq\inf\Big\{d_i(x_i,y_1)+d_j(x_j,y_2)
+d_\infty\big(\psi_i(x_i),\psi_j(x_j)\big)+\eps_i+\eps_j\;\Big|\;
x_i\in\tilde X_i,\,x_j\in\tilde X_j\Big\}.
\end{split}\]
Standard verifications show that \(d\) is a distance on \(Y\).
Note that for any $x\in\tilde X_i$ one has
\begin{align}\label{aux1}
d\big(x,\psi_i(x)\big)\le\eps_i.
\end{align}
In particular, $d(p_i,p_\infty)\to 0$ as \(i\to\infty\).
By weak convergence of $(\psi_i)_*\mm_i$, we have that
\[
\eta(R)\coloneqq\limsup_{i\to\infty}\mm_i\big(B(p_i,R)\big)<+\infty,
\quad\text{ for every }R>0.
\]
We need to show that $\mm_i\rightharpoonup\mm_\infty$ weakly in duality
with $C_{bbs}(Y)$.
First, we claim that
\begin{equation}\label{eq:against_Lip}
\int f\,\d\mm_i\longrightarrow\int f\,\d\mm_\infty,
\quad\text{ for every }f\in C_{bbs}(Y)\text{ Lipschitz.}
\end{equation}
Let any such \(f\) be fixed, say that \(f\) is $L$-Lipschitz.
Pick $R>0$ with $f=0$ on $Y\setminus B(p_i,R/2)$ for all $i$.
Since $(\psi_i)_*\mm_i\rightharpoonup \mm_\infty$, we have
\(\int f\circ\psi_i\,\d\mm_i\to\int f\,\d\mm_\infty\). Thus, it suffices
to prove
\[
\bigg|\int f\,\d\mm_i-\int f\circ\psi_i\,\d\mm_\infty\bigg|
\longrightarrow 0,\quad\text{ as }i\to\infty.
\]
By \eqref{aux1}, we have (for every $i$ with $R_i\ge R$) that
\[\begin{split}
&\bigg|\int f\,\d\mm_i-\int f\circ\psi_i\,\d\mm_\infty\bigg|\\
\leq\,&\bigg|\int_{\tilde X_i\cap B(p_i,R)}f\,\d\mm_i-
\int_{\tilde X_i\cap B(p_i,R)}f\circ\psi_i\,\d\mm_\infty\bigg|
+\bigg|\int_{B(p_i,R_i)\setminus \tilde X_i}f\,\d\mm_i-
\int_{X_i\setminus\tilde X_i}f\circ\psi_i\,\d\mm_\infty\bigg|\\
\leq\,&L\,\varepsilon_i\,\mm_i\big(B(p_i,R)\big)+
2\,\varepsilon_i\sup_Y|f|,
\end{split}\]
whence the claimed property \eqref{eq:against_Lip} follows.
Now, fix any \(f\in C_{bbs}(Y)\) and \(\eps>0\).
Choose any function \(\tilde f\in C_{bbs}(Y)\) Lipschitz with
\(\sup_Y|f-\tilde f|\leq\eps\). For \(R>0\) sufficiently big,
one has
\[\begin{split}
&\lims_{i\to\infty}\bigg|\int f\,\d\mm_i-\int f\,\d\mm_\infty\bigg|\\
\overset{\phantom{\eqref{eq:against_Lip}}}\leq\,&
\lims_{i\to\infty}\int|f-\tilde f|\,\d\mm_i
+\lims_{i\to\infty}\bigg|\int\tilde f\,\d\mm_i-\int\tilde f\,\d\mm_\infty\bigg|
+\int|\tilde f-f|\,\d\mm_\infty\\
\overset{\eqref{eq:against_Lip}}\leq\,&\Big(\eta(R)+\mm_\infty
\big(B(p_\infty,R)\big)\Big)\eps.
\end{split}\]
By the arbitrariness of \(\eps>0\), we deduce that
\(\int f\,\d\mm_\infty=\lim_{i\to\infty}\int f\,\d\mm_i\).
This proves that \(\mm_i\rightharpoonup\mm_\infty\) in duality
with \(C_{bbs}(Y)\), so $(X_i,d_i,\mm_i,p_i)\to(X_\infty,d_\infty,\mm_\infty,p_\infty)$
in the pmG-sense.

To prove the converse implication, suppose that
$(X_i,d_i,\mm_i,p_i)\to(X_\infty,d_\infty,\mm_\infty,p_\infty)$
in the pmG-sense. By Gromov's compactness for
pmG-convergence \cite[Corollary 3.22]{Gigli-Mondino-Savare} and
by Theorems \ref{thm:pre-Gromov} and \ref{thm:Gromov_cpt_main},
we have that (up to passing to a not relabeled subsequence)
$(X_i,d_i,\mm_i,p_i)\to\big(\spt(\mm_\omega),d_\omega,\mm_\omega,p_\omega\big)$
in the wpmGH-sense. Then by using the first part of the proof we obtain that
$(X_i,d_i,\mm_i,p_i)\to\big(\spt(\mm_\omega),d_\omega,\mm_\omega,p_\omega\big)$
in the pmG-topology. By uniqueness of the limit, we conclude.
Therefore, the statement is finally achieved.
\end{proof}
\chapter{Direct and inverse limits of pointed metric measure spaces}
\label{ch:DL_IL_pmms}
\section{The category of pointed metric measure spaces}
In this section, when we consider a pointed Polish metric measure space
\((X,d,\mm,p)\), we implicitly assume that \(p\in{\rm spt}(\mm)\).
\begin{definition}\label{def:morphism_pmms}
Let \((X,d_X,\mm_X,p_X)\) and \((Y,d_Y,\mm_Y,p_Y)\) be pointed Polish metric measure spaces.
Then a map \(\varphi\colon X\to Y$ is said to be a \emph{morphism of pointed Polish
metric measure spaces} if:
\begin{itemize}
\item[\(\rm i)\)] The map \(\varphi\) is Borel measurable
and satisfies \(\varphi(p_X)=p_Y\).
\item[\(\rm ii)\)] The restricted map
\(\varphi|_{{\rm spt}(\mm_X)}\colon{\rm spt}(\mm_X)\to Y\) is \(1\)-Lipschitz.
\item[\(\rm iii)\)] It holds that \(\varphi_*\mm_X\leq\mm_Y\).
\end{itemize}
\end{definition}
With the above notion of morphism at our disposal, we can consider the
category of pointed Polish metric measure spaces. Observe that a morphism \(\varphi\)
from \((X,d_X,\mm_X,p_X)\) to \((Y,d_Y,\mm_Y,p_Y)\) is an isomorphism if and
only if \(\varphi|_{{\rm spt}(\mm_X)}\colon{\rm spt}(\mm_X)\to Y\) is an isometry
and \(\varphi_*\mm_X=\mm_Y\). Below, we briefly remind the notions of
direct and inverse limits of a sequence of pointed Polish
metric measure spaces, referring to the monograph
\cite{MacLane98} for a general treatment of these topics.
\medskip

A couple \(\big(\{(X_i,d_i,\mm_i,p_i)\}_{i\in\N},\{\varphi_{ij}\}_{i\leq j}\big)\)
is said to be a \emph{direct system of pointed Polish metric measure
spaces} provided each \((X_i,d_i,\mm_i,p_i)\) is a pointed
Polish metric measure space and the maps \(\varphi_{ij}\colon X_i\to X_j\)
are morphisms of pointed Polish metric measure spaces satisfying
\[\begin{split}
\varphi_{ii}={\rm id}_{X_i},&\quad\text{ for every }i\in\N,\\
\varphi_{ik}=\varphi_{jk}\circ\varphi_{ij},&\quad\text{ for every }
i,j,k\in\N\text{ with }i\leq j\leq k,
\end{split}\]
where \({\rm id}_{X_i}\colon X_i\to X_i\) stands for the identity
map \(X_i\ni x\mapsto x\in X_i\). Moreover, by a \emph{target} of
the direct system
\(\big(\{(X_i,d_i,\mm_i,p_i)\}_{i\in\N},\{\varphi_{ij}\}_{i\leq j}\big)\)
we mean a couple \(\big((Y,d_Y,\mm_Y,p_Y),\{\psi_i\}_{i\in\N}\big)\),
where \((Y,d_Y,\mm_Y,p_Y)\) is a pointed Polish metric measure space,
while the maps \(\psi_i\colon X_i\to Y\) are morphisms of pointed
Polish metric measure spaces satisfying
\[\begin{tikzcd}[column sep=small]
X_i \arrow[rr,"\varphi_{ij}"] \arrow[rd,swap,"\psi_i"] & &
X_j \arrow[dl,"\psi_j"] \\
& Y &
\end{tikzcd}\]
for every \(i,j\in\N\) with \(i\leq j\). Finally,
we say that a target \(\big((X,d_X,\mm_X,p_X),\{\varphi_i\}_{i\in\N}\big)\)
is the \emph{direct limit} of
\(\big(\{(X_i,d_i,\mm_i,p_i)\}_{i\in\N},\{\varphi_{ij}\}_{i\leq j}\big)\) 
if it satisfies the following universal property: given any
target \(\big((Y,d_Y,\mm_Y,p_Y),\{\psi_i\}_{i\in\N}\big)\), there
exists a unique morphism \(\Phi\colon X\to Y\) of pointed Polish
metric measure spaces such that
\[\begin{tikzcd}
X_i \arrow[r,"\varphi_i"] \arrow[rd,swap,"\psi_i"] &
X \arrow[d,dashed,"\Phi"] \\
& Y
\end{tikzcd}\]
for every \(i\in\N\). Whenever the direct limit exists, it is
uniquely determined up to unique isomorphism.
With an abuse of notation, we will occasionally say that
\((X,d_X,\mm_X,p_X)\) is the direct limit, omitting the
reference to the associated morphisms \(\{\varphi_i\}_{i\in\N}\).
\medskip

A couple \(\big(\{(X_i,d_i,\mm_i,p_i)\}_{i\in\N},\{P_{ij}\}_{i\leq j}\big)\)
is said to be an \emph{inverse system of pointed Polish metric measure
spaces} provided each \((X_i,d_i,\mm_i,p_i)\) is a pointed
Polish metric measure space and the maps \(P_{ij}\colon X_j\to X_i\)
are morphisms of pointed Polish metric measure spaces satisfying
\[\begin{split}
P_{ii}={\rm id}_{X_i},&\quad\text{ for every }i\in\N,\\
P_{ik}=P_{ij}\circ P_{jk},&\quad\text{ for every }
i,j,k\in\N\text{ with }i\leq j\leq k.
\end{split}\]
Moreover, we say that a given couple \(\big((Y,d_Y,\mm_Y,p_Y),\{Q_i\}_{i\in\N}\big)\) \emph{projects on} the inverse system
\(\big(\{(X_i,d_i,\mm_i,p_i)\}_{i\in\N},\{P_{ij}\}_{i\leq j}\big)\)
provided \((Y,d_Y,\mm_Y,p_Y)\) is a pointed Polish metric measure space,
while the maps \(Q_i\colon Y\to X_i\) are morphisms of pointed
Polish metric measure spaces such that
\[\begin{tikzcd}[column sep=small]
& Y \arrow[dl,swap,"Q_j"] \arrow[dr,"Q_i"] & \\
X_j \arrow[rr,swap,"P_{ij}"] & & X_i
\end{tikzcd}\]
holds for every \(i,j\in\N\) with \(i\leq j\). The morphisms
\(\{Q_i\}_{i\in\N}\) are typically called \emph{projections} or
\emph{bonding maps}. Finally, a couple
\(\big((X,d_X,\mm_X,p_X),\{P_i\}_{i\in\N}\big)\) that projects on
the direct system
\(\big(\{(X_i,d_i,\mm_i,p_i)\}_{i\in\N},\{P_{ij}\}_{i\leq j}\big)\)
is said to be its \emph{inverse limit} if
it satisfies the following universal property: given any
couple \(\big((Y,d_Y,\mm_Y,p_Y),\{Q_i\}_{i\in\N}\big)\) that
projects on \(\big(\{(X_i,d_i,\mm_i,p_i)\}_{i\in\N},\{P_{ij}\}_{i\leq j}\big)\),
there exists a unique morphism \(\Phi\colon Y\to X\) of pointed Polish
metric measure spaces such that
\[\begin{tikzcd}
Y \arrow[d,dashed,swap,"\Phi"] \arrow[dr,"Q_i"] & \\
X \arrow[r,swap,"P_i"] & X_i
\end{tikzcd}\]
holds for every \(i\in\N\). Whenever the inverse limit exists,
it is uniquely determined up to unique isomorphism.
\section{Direct limits of pointed metric measure spaces}
To begin with, let us study direct limits in the category of
pointed Polish metric measure spaces.
\begin{theorem}[Direct limits of pointed metric measure spaces]\label{thm:DL}
Consider a direct system
\begin{equation}\label{eq:direct_system}
\Big(\big\{(X_i,d_i,\mm_i,p_i)\big\}_{i\in\N},\{\varphi_{ij}\}_{i\leq j}\Big)
\end{equation}
of pointed Polish metric measure spaces.
Then the direct system in \eqref{eq:direct_system}
admits a direct limit \((X,d_X,\mm_X,p_X)\) if and only if
\(\big((X_i,d_i,\mm_i,p_i)\big)\) is \(\omega\)-uniformly boundedly finite.
In this case,
\[
(X,d_X,\mm_X,p_X)=\big({\rm spt}(\mm_\omega),d_\omega|_{{\rm spt}(\mm_\omega)
\times{\rm spt}(\mm_\omega)},\mm_\omega|_{\mathscr B({\rm spt}(\mm_\omega))},
p_\omega\big),
\]
where \((X_\omega,d_\omega,\mm_\omega,p_\omega)\) stands for the
ultralimit \(\lim_{i\to\omega}(X_i,d_i,\mm_i,p_i)\).
\end{theorem}
\begin{proof}
We subdivide the proof into several steps:\\
{\color{blue}\textsc{Step 1.}} First of all, we claim that if
\(\big((X_i,d_i,\mm_i,p_i)\big)\) is not \(\omega\)-uniformly boundedly finite,
then the direct system in \eqref{eq:direct_system} does not admit any target
in the category of pointed Polish metric measure spaces, thus in particular it
does not have a direct limit in such category. To this aim, suppose
\(\lim_{i\to\omega}\mm_i\big(B(p_i,R)\big)=+\infty\) for some \(R>0\).
Then we can find a subsequence \((i_j)_j\) such that \(\lim_{j\to\infty}\mm_{i_j}\big(B(p_{i_j},R)\big)=+\infty\). We argue by contradiction: suppose to have a
target \(\big((Y,d_Y,\mm_Y,p_Y),\{\psi_i\}_{i\in\N}\big)\) of the direct system
in \eqref{eq:direct_system}. Since \(\psi_i|_{{\rm spt}(\mm_i)}\) is
\(1\)-Lipschitz and \((\psi_i)_*\mm_i\leq\mm_Y\) for all \(i\in\N\), we have that
\(\psi_i\big({\rm spt}(\mm_i)\cap B(p_i,R)\big)\subset B(p_Y,R)\) and thus
\[
\mm_{i_j}\big(B(p_{i_j},R)\big)\leq(\psi_{i_j})_*\mm_{i_j}
\big(\psi_{i_j}\big({\rm spt}(\mm_{i_j})\cap B(p_{i_j},R)\big)\big)
\leq\mm_Y\big(B(p_Y,R)\big),\quad\text{ for every }j\in\N.
\]
By letting \(j\to\infty\), we conclude that \(\mm_Y\big(B(p_Y,R)\big)=+\infty\),
which is contradiction with the fact that \(\mm_Y\) is boundedly finite.
Therefore, no target exists and thus the claim is proven.\\
{\color{blue}\textsc{Step 2.}} Hereafter, we shall consider the case
where \(\big((X_i,d_i,\mm_i,p_i)\big)\) is a \(\omega\)-uniformly
boundedly finite sequence. Let us define
\(\hat X_i\coloneqq{\rm spt}(\mm_i)\subset X_i\) for every \(i\in\N\).
Notice that the inclusion \(\varphi_{ij}(\hat X_i)\subset\hat X_j\)
holds for every \(i,j\in\N\) with \(i\leq j\).
Let \(i\in\N\) be fixed. We define the map
\(\bar\varphi_i\colon\hat X_i\to O(\bar X_\omega)\) as
\(\bar\varphi_i(x_i)\coloneqq\big[\varphi_{ik}(x_i)\big]_{k\geq i}\) for
every \(x_i\in\hat X_i\), while \(\varphi_i\colon\hat X_i\to X_\omega\) is
given by \(\varphi_i\coloneqq\pi\circ\bar\varphi_i\), namely,
\[
\varphi_i(x_i)=\big[\big[\varphi_{ik}(x_i)\big]\big]_{k\geq i},
\quad\text{ for every }x_i\in\hat X_i.
\]
We also extend \(\varphi_i\) to a measurable map defined on the whole \(X_i\),
by declaring that \(\varphi_i(x_i)\coloneqq p_\omega\) for every
\(x_i\in X_i\setminus\hat X_i\). Observe that \(\bar\varphi_i\) is well-posed,
as a consequence of the estimate
\[
\bar d_\omega\big(\bar\varphi_i(x_i),{[p_k]}_k\big)=
\lim_{k\to\omega}d_k\big(\varphi_{ik}(x_i),p_k\big)=
\lim_{k\to\omega}d_k\big(\varphi_{ik}(x_i),\varphi_{ik}(p_i)\big)
\leq d_i(x_i,p_i)<+\infty.
\]
Notice also that \(\varphi_i(p_i)=\big[\big[\varphi_{ik}(p_i)\big]\big]_{k\geq i}
=[[p_k]]=p_\omega\). Moreover, we have that \(\bar\varphi_i\) (and thus also
\(\varphi_i\)) is \(1\)-Lipschitz when restricted to \(\hat X_i\), as shown by
the following estimate:
\[
\bar d_\omega\big(\bar\varphi_i(x_i),\bar\varphi_i(y_i)\big)=
\lim_{k\to\omega}d_k\big(\varphi_{ik}(x_i),\varphi_{ik}(y_i)\big)
\leq d_i(x_i,y_i),\quad\text{ for every }x_i,y_i\in\hat X_i.
\] 
{\color{blue}\textsc{Step 3.}} Let us now define the closed,
separable subset \(X\) of \(X_\omega\) as
\[
X\coloneqq\text{closure of }\bigcup_{i\in\N}\varphi_i(\hat X_i)\,\text{ in }X_\omega.
\]
We call \(d_X\coloneqq d_\omega|_{X\times X}\),
\(\mm_X\coloneqq\mm_\omega|_{\mathscr B(X)}\), and \(p_X\coloneqq p_\omega\).
Note that \(\varphi_i(X_i)\subset X\) for all \(i\in\N\).
We aim to prove that each \(\varphi_i\colon X_i\to X\)
satisfies \((\varphi_i)_*\mm_i\leq\mm_X\). To do so, we first claim that
\begin{equation}\label{eq:DL_aux1}
\Pi_{k\to\omega}\varphi_{ik}(K)\subset{\rm cl}_{O(\bar X_\omega)}\big(
\bar\varphi_i(K)\big),\quad\text{ for every }K\subset\hat X_i\text{ compact.}
\end{equation}
Observe that the set \(\bar\varphi_i(K)\) is compact (since the map \(\bar\varphi_i\)
is continuous), but is not necessarily closed; this is due to the fact that
\(\big(O(\bar X_\omega),\bar d_\omega\big)\) is a pseudometric space
and not a metric space. Fix any \(\varepsilon>0\) and denote by
\(C_\varepsilon\) the closed \(\varepsilon\)-neighbourhood of \(\bar\varphi_i(K)\).
Being the set \(K\) compact, we can find points \(z^1,\ldots,z^n\in K\) such that
\(K\subset\bigcup_{j=1}^n B(z^j,\varepsilon)\). Now fix any
\([y_k]\in\Pi_{k\to\omega}\varphi_{ik}(K)\). Pick some \(S\in\omega\) such that
\(k\geq i\) and \(y_k\in\varphi_{ik}(K)\) for every \(k\in S\). Given any
\(k\in S\), we can choose \(x^k\in K\) such that \(y_k=\varphi_{ik}(x^k)\).
Then let us define
\[
S^j\coloneqq\big\{k\in S\;\big|\;d_i(x^k,z^j)<\varepsilon\big\},
\quad\text{ for every }j=1,\ldots,n.
\]
Since \(S=\bigcup_{j=1}^n S^j\) and \(S\in\omega\), we deduce that
\(S^{j_0}\in\omega\) for some \(j_0=1,\ldots,n\). In particular, we have that
\(d_i(x^k,z^{j_0})<\varepsilon\) holds for \(\omega\)-a.e.\ \(k\geq i\),
thus accordingly
\[
\bar d_\omega\big({[y_k]}_k,\bar\varphi_i(z^{j_0})\big)
=\lim_{k\to\omega}d_k\big(\varphi_{ik}(x^k),\varphi_{ik}(z^{j_0})\big)
\leq\lim_{k\to\omega}d_i(x^k,z^{j_0})\leq\varepsilon.
\]
This shows that
\([y_k]_k\in\bar B\big(\bar\varphi_i(z^{j_0}),\varepsilon\big)\subset C_\varepsilon\).
Hence, we have proven that
\(\Pi_{k\to\omega}\varphi_{ik}(K)\subset C_\varepsilon\).
Since the intersection \(\bigcap_{\eps>0}C_\eps\) coincides with
the closure of \(\bar\varphi_i(K)\) in \(O(\bar X_\omega)\),
we obtain \eqref{eq:DL_aux1}.

Let us now pass to the verification of the inequality \((\varphi_i)_*\mm_i\leq\mm_X\).
Fix any \(C\subset X\) closed. Since \(\mm_i\) is inner regular and
\(\varphi_i^{-1}(C)\in\mathscr B(X_i)\), we can find a sequence \((K_n)_n\) of
compact subsets of \(\varphi_i^{-1}(C)\cap\hat X_i\) such that
\(\mm_i\big(\varphi_i^{-1}(C)\big)=\lim_{n\to\infty}\mm_i(K_n)\).
Observe that we have
\(\Pi_{k\to\omega}\varphi_{ik}(K_n)\subset{\rm cl}_{O(\bar X_\omega)}\big(\bar\varphi_i(K_n)\big)\subset{\rm cl}_{O(\bar X_\omega)}\big(
\bar\varphi_i\big(\varphi_i^{-1}(C)\cap\hat X_i\big)\big)\subset\pi^{-1}(C)\),
where we used \eqref{eq:DL_aux1} and the fact that \(\pi^{-1}(C)\) is
closed thanks to the continuity of \(\pi\). Then it holds that
\[\begin{split}
\mm_X(C)&=\mm_\omega(C)=\bar\mm_\omega\big(\pi^{-1}(C)\big)\geq
\bar\mm_\omega\big({\textstyle\Pi_{k\to\omega}}\varphi_{ik}(K_n)\big)
=\lim_{k\to\omega}\mm_k\big(\varphi_{ik}(K_n)\big)\\
&\geq\lim_{k\to\omega}(\varphi_k)_*\mm_i\big(\varphi_{ik}(K_n)\big)
=\lim_{k\to\omega}\mm_i\big(\varphi_{ik}^{-1}\big(\varphi_{ik}(K_n)\big)\big)\geq\mm_i(K_n).
\end{split}\]
By letting \(n\to\infty\), we thus deduce that
\(\mm_X(C)\geq\mm_i\big(\varphi_i^{-1}(C)\big)=(\varphi_i)_*\mm_i(C)\).
By exploiting the inner regularity of \(\mm_X\) and \((\varphi_i)_*\mm_i\),
we finally conclude that \((\varphi_i)_*\mm_i\leq\mm_X\), as desired.
\\
{\color{blue}\textsc{Step 4.}} Let us now prove that
\begin{equation}\label{eq:DL_aux4}
X={\rm spt}(\mm_\omega).
\end{equation}
Given that \(\varphi_i|_{\hat X_i}\) is \(1\)-Lipschitz (by \textsc{Step 2})
and \((\varphi_i)_*\mm_i\leq\mm_\omega\) (by \textsc{Step 3}), we can deduce that
\(\varphi_i(\hat X_i)\subset{\rm spt}(\mm_\omega)\) for every \(i\in\N\),
thus obtaining that \(X\subset{\rm spt}(\mm_\omega)\). In order to prove the
converse inclusion, it is enough to show that \(x\notin{\rm spt}(\mm_\omega)\)
for any given point \(x\in X_\omega\setminus X\). Fix any \(r>0\) such that
\(X\cap B(x,3r)=\emptyset\). Choose any \(R>0\) for which
\(B(x,3r)\subset B(p_\omega,R/2)\).
Given that \(\varphi_{ik}|_{\hat X_i}\) is a \(1\)-Lipschitz map and
\((\varphi_{ik})_*\mm_i\leq\mm_k\) for all \(i\leq k\), we have that
\begin{equation}\label{eq:DL_aux5}
\varphi_{ik}\big(\hat X_i\cap B(p_i,R)\big)\subset B(p_k,R),
\quad\text{ for every }i,k\in\N\text{ with }i\leq k.
\end{equation}
In particular, if \(i,k\in\N\) satisfy \(i\leq k\), then it holds that
\[
\mm_i\big(B(p_i,R)\big)\overset{\eqref{eq:DL_aux5}}\leq
\mm_i\big(\varphi_{ik}^{-1}\big(B(p_k,R)\big)\big)=
(\varphi_{ik})_*\mm_i\big(B(p_k,R)\big)\leq\mm_k\big(B(p_k,R)\big),
\]
which shows that \(\N\ni i\mapsto\mm_i\big(B(p_i,R)\big)\) is non-decreasing.
Thanks to the \(\omega\)-uniform bounded finiteness assumption, we have that
\(M\coloneqq\lim_{i\to\omega}\mm_i\big(B(p_i,R)\big)<+\infty\). Given any
\(\varepsilon>0\), we can thus find a family \(S\in\omega\) such that
\begin{equation}\label{eq:DL_aux6}
M-\varepsilon\leq\mm_i\big(B(p_i,R)\big)\leq M,\quad\text{ for every }i\in S.
\end{equation}
For any \(i\in S\), pick a compact set \(K_i\subset\hat X_i\cap B(p_i,R)\)
such that
\begin{equation}\label{eq:DL_aux7}
\mm_i\big(B(p_i,R)\setminus K_i\big)\leq\varepsilon.
\end{equation}
Set \(i_0\coloneqq\min(S)\).
Note that
\(\Pi_{k\to\omega}\varphi_{i_0 k}(K_{i_0})\subset
{\rm cl}_{O(\bar X_\omega)}\big(\bar\varphi_{i_0}(K_{i_0})\big)\)
by \eqref{eq:DL_aux1}. Calling \(x=[[x_k]]\), we have that
\(\Pi_{k\to\omega}B(x_k,2r)\subset B\big([x_k],4r\big)\). Since
\(X\cap B(x,3r)=\emptyset\), we deduce that
\[\begin{split}
\Pi_{k\to\omega}\big(\varphi_{i_0 k}(K_{i_0})\cap B(x_k,2r)\big)&=
\big(\Pi_{k\to\omega}\varphi_{i_0 k}(K_{i_0})\big)\cap
\big(\Pi_{k\to\omega}B(x_k,2r)\big)\\
&\subset{\rm cl}_{O(\bar X_\omega)}\big(\bar\varphi_{i_0}(K_{i_0})\big)
\cap B\big([x_k],3r\big)\\
&\subset{\rm cl}_{O(\bar X_\omega)}\Big(\pi^{-1}\big(\varphi_{i_0}(K_{i_0})
\big)\Big)\cap\pi^{-1}\big(B(x,3r)\big)\\
&\subset\pi^{-1}\Big({\rm cl}_{X_\omega}\big(\varphi_{i_0}(K_{i_0})\big)
\big)\cap B(x,3r)\Big)\\
&\subset\pi^{-1}\big(X\cap B(x,3r)\big)=\emptyset,
\end{split}\]
so that there is a family \(T\in\omega\) such that \(T\subset S\) and
\(\varphi_{i_0 k}(K_{i_0})\cap B(x_k,2r)=\emptyset\) for all \(k\in T\).
Since \(\Pi_{k\to\omega}B(x_k,2r)\subset B\big([x_k],3r\big)\subset
B\big([p_k],R/2\big)\subset\Pi_{k\to\omega}B(p_k,R)\), we can additionally
require that \(B(x_k,2r)\subset B(p_k,R)\) for every \(k\in T\).
Therefore, for any \(k\in T\) it holds that
\[\begin{split}
\mm_k\big(B(x_k,2r)\big)&\overset{\phantom{\eqref{eq:DL_aux6}}}\leq
\mm_k\big(\big(\hat X_k\cap B(p_k,R)\big)\setminus\varphi_{i_0 k}(K_{i_0})\big)
\overset{\eqref{eq:DL_aux5}}=\mm_k\big(B(p_k,R)\big)-
\mm_k\big(\varphi_{i_0 k}(K_{i_0})\big)\\
&\overset{\eqref{eq:DL_aux6}}\leq M-\mm_k\big(\varphi_{i_0 k}(K_{i_0})\big)
\leq M-(\varphi_{i_0 k})_*\mm_{i_0}\big(\varphi_{i_0 k}(K_{i_0})\big)
\leq M-\mm_{i_0}(K_{i_0})\\
&\overset{\eqref{eq:DL_aux7}}\leq M-\mm_{i_0}\big(B(p_{i_0},R)\big)+\varepsilon
\overset{\eqref{eq:DL_aux6}}\leq M-(M-\varepsilon)+\varepsilon=2\varepsilon.
\end{split}\]
Consequently, by using the fact that
\(B\big([x_k],r\big)\subset\Pi_{k\to\omega}B(x_k,2r)\), we deduce that
\[
\mm_\omega\big(B(x,r)\big)\leq\bar\mm_\omega\big(\Pi_{k\to\omega}B(x_k,2r)\big)
=\lim_{k\to\omega}\mm_k\big(B(x_k,2r)\big)\leq 2\varepsilon.
\]
By the arbitrariness of \(\varepsilon>0\), we conclude that
\(\mm_\omega\big(B(x,r)\big)=0\), which shows that
\(x\) does not belong to \({\rm spt}(\mm_\omega)\). Hence, the claimed identity
\eqref{eq:DL_aux4} is finally proven.\\
{\color{blue}\textsc{Step 5.}} Observe that \(\varphi_i=\varphi_j\circ\varphi_{ij}\)
holds for every \(i,j\in\N\) such that \(i\leq j\). Indeed, we have
\[
\varphi_i(x_i)=\big[\big[\varphi_{ik}(x_i)\big]\big]_{k\geq i}
=\big[\big[(\varphi_{jk}\circ\varphi_{ij})(x_i)\big]\big]_{k\geq j}
=\varphi_j\big(\varphi_{ij}(x_i)\big),\quad\text{ for every }x_i\in\hat X_i.
\]
Therefore, \(\big((X,d_X,\mm_X,p_X),\{\varphi_i\}_{i\in\N}\big)\) is a target of
the direct system in \eqref{eq:direct_system}. In order to prove that it is actually
the direct limit, we have to show that is satisfies the universal property. To this aim,
fix any target \(\big((Y,d_Y,\mm_Y,p_Y),\{\psi_i\}_{i\in\N}\big)\) of the direct system
in \eqref{eq:direct_system}. We want to prove that there exists a unique morphism
\(\Phi\colon X\to Y\) of pointed Polish metric measure spaces such that
\(\psi_i=\Phi\circ\varphi_i\) holds for every \(i\in\N\). First, notice that
for any \(x\in\bigcup_{i\in\N}\varphi_i(\hat X_i)\) the choice of \(\Phi(x)\)
is forced: given \(i\in\N\) and \(x_i\in\hat X_i\) with \(x=\varphi_i(x_i)\),
we must set \(\Phi(x)\coloneqq\psi_i(x_i)\). We need to check that this definition
is well-posed. Namely, we have to prove that
\begin{equation}\label{eq:DL_aux2}
i,j\in\N,\;x_i\in\hat X_i,\;x_j\in\hat X_j,\;\varphi_i(x_i)=\varphi_j(x_j)
\quad\Longrightarrow\quad\psi_i(x_i)=\psi_j(x_j).
\end{equation}
By using the fact that
\(\psi_i(x_i)=\psi_k\big(\varphi_{ik}(x_i)\big)\) and
\(\psi_j(x_j)=\psi_k\big(\varphi_{jk}(x_j)\big)\) for all \(k\geq i,j\), we get
\[\begin{split}
d_Y\big(\psi_i(x_i),\psi_j(x_j)\big)&=\lim_{k\to\omega}
d_Y\big(\psi_k\big(\varphi_{ik}(x_i)\big),\psi_k\big(\varphi_{jk}(x_j)\big)\big)
\leq\lim_{k\to\omega}d_k\big(\varphi_{ik}(x_i),\varphi_{jk}(x_j)\big)\\
&=d_X\big(\varphi_i(x_i),\varphi_j(x_j)\big)=0,
\end{split}\]
whence \eqref{eq:DL_aux2} follows. A similar computation shows that
\(\Phi\colon\bigcup_{i\in\N}\varphi_i(\hat X_i)\to Y\) is \(1\)-Lipschitz:
if \(x=\varphi_i(x_i)\) and \(y=\varphi_j(y_j)\) for some \(i,j\in\N\),
\(x_i\in\hat X_i\), and \(y_j\in\hat X_j\), then we have that
\[\begin{split}
d_Y\big(\Phi(x),\Phi(y)\big)&=d_Y\big(\psi_i(x_i),\psi_j(y_j)\big)
=\lim_{k\to\omega}d_Y\big(\psi_k\big(\varphi_{ik}(x_i)\big),
\psi_k\big(\varphi_{jk}(y_j)\big)\big)\\
&\leq\lim_{k\to\omega}d_k\big(\varphi_{ik}(x_i),\varphi_{jk}(y_j)\big)
=d_X\big(\varphi_i(x_i),\varphi_j(y_j)\big)=d_X(x,y).
\end{split}\]
Consequently, it holds that \(\Phi\) can be uniquely extended to a
\(1\)-Lipschitz mapping \(\Phi\colon X\to Y\). Moreover, notice
that \(\Phi(p_X)=\psi_1(p_1)=p_Y\). To prove that \(\Phi\) is a morphism
of pointed Polish metric measure spaces, it suffices to show that
\(\Phi_*\mm_X\leq\mm_Y\). To achieve this goal, we need the following fact:
\begin{equation}\label{eq:DL_aux3}
\pi^{-1}\big(\Phi^{-1}(B)\big)\subset\Pi_{k\to\omega}\psi_k^{-1}(B),
\quad\text{ for every }B\in\mathscr B(Y).
\end{equation}
Let \([y_k]\in\pi^{-1}\big(\Phi^{-1}(B)\big)\) be given. Choose any
\(i\in\N\) and \(x_i\in\hat X_i\) such that \([[y_k]]=\varphi_i(x_i)\). For every
\(k\geq i\), we know that \(\psi_k\big(\varphi_{ik}(x_i)\big)=\psi_i(x_i)=
\Phi\big([[y_k]]\big)\in B\) by definition of \(\Phi\). This implies that
\(y_k=\varphi_{ik}(x_i)\in\psi_k^{-1}(B)\) holds for \(\omega\)-a.e.\ \(k\geq i\),
which is equivalent to saying that the element \([y_k]\) belongs to
\(\Pi_{k\to\omega}\psi_k^{-1}(B)\). Therefore, the claim in \eqref{eq:DL_aux3} is proven.

Finally, given any Borel set \(B\subset Y\), we have that
\[\begin{split}
\Phi_* \mm_X(B)&=\mm_\omega\big(\Phi^{-1}(B)\big)
=\bar\mm_\omega\big(\pi^{-1}(\Phi^{-1}(B))\big)
\overset{\eqref{eq:DL_aux3}}\leq\bar\mm_\omega\big(\Pi_{k\to\omega}\psi_k^{-1}(B)\big)=\lim_{k\to\omega}\mm_k\big(\psi_k^{-1}(B)\big)\\
&=\lim_{k\to\omega}(\psi_k)_*\mm_k(B)\leq\mm_Y(B).
\end{split}\]
This yields \(\Phi_*\mm_X\leq\mm_Y\), thus the universal property
is verified. The statement follows.
\end{proof}
We preferred to formulate Theorem \ref{thm:DL} in the language of ultralimits,
as they constitute the main topic of this paper, but we can readily provide
an alternative (and more explicit) characterisation using the wpmGH
convergence.
\begin{corollary}\label{cor:DL}
A given direct system \eqref{eq:direct_system} of pointed Polish metric
measure spaces admits a direct limit \((X,d_X,\mm_X,p_X)\) if and only if
\(\big((X_i,d_i,\mm_i,p_i)\big)\) is uniformly boundedly finite.

In this case, \((X,d_X,\mm_X,p_X)\) coincides with the
wpmGH limit of \(\big((X_i,d_i,\mm_i,p_i)\big)\).
\end{corollary}
\begin{proof}
For a direct system of pointed Polish metric measure spaces,
uniform bounded finiteness and \(\omega\)-uniform bounded finiteness
are equivalent, thanks to the monotonicity of the function
\(\N\ni i\mapsto\mm_i\big(B(p_i,R)\big)\) for any \(R>0\)
(observed in \textsc{Step 4} of Theorem \ref{thm:DL}).
Then the statement can be proven by combining Theorem
\ref{thm:DL} with the results in Section \ref{s:wpmGH_vs_UL}.
\end{proof}
\section{Inverse limits of pointed metric measure spaces}
We now pass to the study of inverse limits in the category of
pointed Polish metric measure spaces.
\begin{theorem}[Inverse limits of pointed metric measure spaces]\label{thm:IL}
Consider an inverse system
\begin{equation}\label{eq:inverse_system}
\Big(\big\{(X_i,d_i,\mm_i,p_i)\big\}_{i\in\N},\{P_{ij}\}_{i\leq j}\Big)
\end{equation}
of pointed Polish metric measure spaces. Then \(\big((X_i,d_i,\mm_i,p_i)\big)\)
is a uniformly boundedly finite sequence, so that the ultralimit
\((X_\omega,d_\omega,\mm_\omega,p_\omega)=\lim_{i\to\omega}(X_i,d_i,\mm_i,p_i)\)
exists. Let us define
\begin{equation}\label{eq:def_X_IL}
X\coloneqq\hat X\cap{\rm spt}(\mm_\omega),\quad\text{ where }
\hat X\coloneqq\Big\{[[x_i]]\in X_\omega\,\Big|\,P_{ij}(x_j)=
x_i\in{\rm spt}(\mm_i)\text{ for all }i\leq j\Big\},
\end{equation}
while \(d_X\coloneqq d_\omega|_{X\times X}\),
\(\mm_X\coloneqq\mm_\omega|_{\mathscr B(X)}\), and \(p_X\coloneqq p_\omega\).
Then \((X,d_X,\mm_X)\) is a Polish metric measure space. Moreover, it holds
that the inverse system in \eqref{eq:inverse_system} admits an inverse limit
if and only if \(p_X\in{\rm spt}(\mm_X)\). In this case, the inverse limit
coincides with \((X,d_X,\mm_X,p_X)\), the natural projection maps
\(P_i\colon X\to X_i\) being given by \(P_i\big([[x_k]]\big)\coloneqq x_i\)
for every \([[x_k]]\in X\).
\end{theorem}
\begin{proof}
We subdivide the proof into several steps:\\
{\color{blue}\textsc{Step 1.}} First of all, let us prove that the sequence
\(\big((X_i,d_i,\mm_i,p_i)\big)\) is uniformly boundedly finite,
thus the ultralimit \((X_\omega,d_\omega,\mm_\omega,p_\omega)=
\lim_{i\to\omega}(X_i,d_i,\mm_i,p_i)\) exists.
Given any \(i,j\in\N\) with \(i\leq j\), it follows from the fact that
\(P_{ij}|_{{\rm spt}(\mm_j)}\) is \(1\)-Lipschitz that
\begin{equation}\label{eq:IL_aux1}
P_{ij}\big(B(x_j,r)\cap{\rm spt}(\mm_j)\big)\subset B\big(P_{ij}(x_j),r\big),
\quad\text{ for every }x_j\in X_j\text{ and }r>0.
\end{equation}
Therefore, since we know that \((P_{ij})_*\mm_j\leq\mm_i\), we deduce that
\begin{equation}\label{eq:IL_aux2}\begin{split}
\mm_j\big(B(x_j,r)\big)&=\mm_j\big(B(x_j,r)\cap{\rm spt}(\mm_j)\big)
\overset{\eqref{eq:IL_aux1}}\leq\mm_j\big(P_{ij}^{-1}\big(B(P_{ij}(x_j),r)
\big)\big)\\&\leq\mm_i\big(B(P_{ij}(x_j),r)\big).
\end{split}\end{equation}
As a consequence of \eqref{eq:IL_aux2}, we have that
\(\sup_{i\in\N}\mm_i\big(B(p_i,R)\big)\leq\mm_1\big(B(p_1,R)\big)<+\infty\)
holds for every \(R>0\), which proves that \(\big((X_i,d_i,\mm_i,p_i)\big)\)
is a uniformly boundedly finite sequence.\\
{\color{blue}\textsc{Step 2.}} We claim that
\begin{equation}\label{eq:IL_aux3}
\hat X\;\text{ is a closed subset of }X_\omega.
\end{equation}
In order to prove it, we will show that any \(d_\omega\)-Cauchy sequence
\(({\bf x}^k)_k\subset\hat X\) converges to some element of \(\hat X\).
Say that \({\bf x}^k={[[x^k_i]]}_i\), where \(x^k_i=P_{ij}(x^k_j)\) for all
\(i\leq j\). Given any \(k,k'\in\N\), we have that
\(d_i(x^k_i,x^{k'}_i)=d_i\big(P_{ij}(x^k_i),P_{ij}(x^{k'}_i)\big)
\leq d_j(x^k_j,x^{k'}_j)\) for all \(i,j\in\N\) with \(i\leq j\).
This shows that \(\N\ni i\mapsto d_i(x^k_i,x^{k'}_i)\) is non-decreasing,
thus in particular
\begin{equation}\label{eq:IL_aux6}
d_\omega({\bf x}^k,{\bf x}^{k'})=\sup_{i\in\N}d_i(x^k_i,x^{k'}_i),
\quad\text{ for every }k,k'\in\N.
\end{equation}
Hence, given any \(i\in\N\), we know from \eqref{eq:IL_aux6} and the
fact that \(({\bf x}^k)_k\) is \(d_\omega\)-Cauchy that \((x^k_i)_k\)
is \(d_i\)-Cauchy. Let us denote by \(x_i\in{\rm spt}(\mm_i)\) its limit.
Define \({\bf x}\coloneqq[[x_i]]\in X_\omega\). Since \(P_{ij}\) is
continuous, by letting \(k\to\infty\) in the identity \(x_i^k=P_{ij}(x_j^k)\)
we get \(x_i=P_{ij}(x_j)\) for all \(i\leq j\), whence it follows that
\({\bf x}\in\hat X\). It only remains to show that \(d_\omega({\bf x}^k,{\bf x})
\to 0\) as \(k\to\infty\). To this aim, fix \(\varepsilon>0\). Pick any
\(\bar k\in\N\) such that \(d_\omega({\bf x}^k,{\bf x}^{k'})<\varepsilon\)
for all \(k,k'\geq\bar k\). Given any \(i\in\N\), we deduce from
\eqref{eq:IL_aux6} that \(d_i(x^k_i,x^{k'}_i)<\varepsilon\) for all
\(k,k'\geq\bar k\). Hence, by letting \(k'\to\infty\) we deduce that
\(d_i(x^k_i,x_i)\leq\varepsilon\) for all \(k\geq\bar k\), so that
\(d_\omega({\bf x}^k,{\bf x})=\lim_{i\to\omega}d_i(x^k_i,x_i)\leq\varepsilon\)
for every \(k\geq\bar k\). This proves that
\(\lim_{k\to\infty}d_\omega({\bf x}^k,{\bf x})=0\), thus accordingly
the claimed property \eqref{eq:IL_aux3} holds.

In particular, since \({\rm spt}(\mm_\omega)\) is complete
and separable, we deduce that \((X,d_X)\) is a Polish metric space
and thus \(\mm_X\) is a (well-defined) non-negative Borel measure on \(X\).\\
{\color{blue}\textsc{Step 3.}} Now, for any given \(i\in\N\) let us define the
projection map \(P_i\colon X\to X_i\) as
\[
P_i\big([[x_k]]\big)\coloneqq x_i,\quad\text{ for every }[[x_k]]\in X.
\]
Suppose that \({\rm spt}(\mm_X)\neq\emptyset\). Observe that
\(P_i=P_{ik}\circ P_k\) for every \(i,k\in\N\) with \(i\leq k\)
by construction. We now aim to prove that for every \(i\in\N\) it holds that
\begin{subequations}\begin{align}\label{eq:IL_aux4a}
P_i|_{{\rm spt}(\mm_X)}\colon{\rm spt}(\mm_X)\to X_i&,
\quad\text{ is }1\text{-Lipschitz},\\
\label{eq:IL_aux4b}(P_i)_*\mm_X\leq\mm_i&.
\end{align}\end{subequations}
We begin with the verification of \eqref{eq:IL_aux4a}. Fix
\([[x_k]],[[y_k]]\in{\rm spt}(\mm_X)\). Since \(x_k,y_k\in{\rm spt}(\mm_k)\)
and \(P_{ik}|_{{\rm spt}(\mm_k)}\) is \(1\)-Lipschitz for every \(k\geq i\),
we get \(d_i(x_i,y_i)\leq d_k(x_k,y_k)\) for every \(k\geq i\) and thus
\(d_i(x_i,y_i)\leq\lim_{k\to\omega}d_k(x_k,y_k)=d_X\big([[x_k]],[[y_k]]\big)\),
which gives \eqref{eq:IL_aux4a}.

In order to prove \eqref{eq:IL_aux4b}, we need to show that
\(\mm_X\big(P_i^{-1}(A)\big)\leq\mm_i(A)\) for every given
Borel set \(A\subset X_i\). Notice that if
\([x_k]\in\pi^{-1}\big(P_i^{-1}(A)\big)\), then \(P_{ik}(x_k)=x_i\in A\) for
\(\omega\)-a.e.\ \(k\geq i\). This shows that \(\pi^{-1}\big(P_i^{-1}(A)\big)
\subset\Pi_{k\to\omega}P_{ik}^{-1}(A)\), so that accordingly
\[
\mm_X\big(P_i^{-1}(A)\big)=\bar\mm_\omega\big(\pi^{-1}\big(P_i^{-1}(A)\big)\big)
\leq\bar\mm_\omega\big(\Pi_{k\to\omega}P_{ik}^{-1}(A)\big)=
\lim_{k\to\omega}(P_{ik})_*\mm_k(A)\leq\mm_i(A).
\]
Therefore, \eqref{eq:IL_aux4b} is proven.
Notice also that if \(p_X\in{\rm spt}(\mm_X)\), then \(P_i(p_X)=p_i\)
for all \(i\in\N\).\\
{\color{blue}\textsc{Step 4.}} Now suppose to have a pointed Polish metric measure
space \((Y,d_Y,\mm_Y,p_Y)\) and projection maps \(Q_i\colon Y\to X_i\) for
every \(i\in\N\); namely, the maps \(Q_i\) are morphisms of pointed Polish
metric measure spaces and \(Q_i=P_{ij}\circ Q_j\) holds whenever \(i\leq j\).
Without loss of generality, we can assume that \({\rm spt}(\mm_Y)=Y\).
Define \(\Phi\colon Y\to\hat X\) as \(\Phi(y)\coloneqq\big[\big[Q_i(y)\big]\big]\)
for every \(y\in Y\). By arguing as in \textsc{Step 1}, we can see that the
map \(\Phi\colon Y\to\hat X\) is \(1\)-Lipschitz and
\(\mm_Y\big(B(y,r)\big)\leq\mm_i\big(B(Q_i(y),r)\big)\) for all \(i\in\N\),
\(y\in Y\), and \(r>0\). In particular, one has
\[\begin{split}
\mm_\omega\big(\bar B(\Phi(y),r)\big)&=\inf_{\ell\in\N}\bar\mm_\omega
\big(\Pi_{i\to\omega}B(Q_i(y),r+1/\ell)\big)=
\inf_{\ell\in\N}\lim_{i\to\omega}\mm_i\big(B(Q_i(y),r+1/\ell)\big)\\
&\geq\mm_Y\big(B(y,r)\big)>0,\quad\text{ for every }r>0.
\end{split}\]
This grants that \(\Phi(y)\in{\rm spt}(\mm_\omega)\) and thus
\(\Phi(Y)\subset X\). Observe that \(\Phi\) is the unique map from \(Y\)
to \(X\) satisfying \(Q_i=P_i\circ\Phi\) for every \(i\in\N\).
Moreover, let us define \(\bar\Phi\colon Y\to O(\bar X_\omega)\) as
\(\bar\Phi(y)\coloneqq\big[Q_i(y)\big]\) for every \(y\in Y\).
Notice that \(\Phi=\pi\circ\bar\Phi\). By arguing as in \textsc{Step 3}
of the proof of Theorem \ref{thm:DL}
-- specifically, where we proved the inclusion \eqref{eq:DL_aux1} --
one can show that
\begin{equation}\label{eq:IL_aux5}
\Pi_{i\to\omega}Q_i(K)\subset{\rm cl}_{O(\bar X_\omega)}
\big(\bar\Phi(K)\big),\quad\text{ for every }K\subset Y\text{ compact.}
\end{equation}
Consequently, given any closed set \(C\subset X\) and any compact set
\(K\subset\Phi^{-1}(C)\), it holds that 
\[\begin{split}
\mm_Y(K)&\leq\lim_{i\to\omega}\mm_Y\big(Q_i^{-1}\big(Q_i(K)\big)\big)
=\lim_{i\to\omega}(Q_i)_*\mm_Y\big(Q_i(K)\big)
\leq\lim_{i\to\omega}\mm_i\big(Q_i(K)\big)\\
&=\bar\mm_\omega\big(\Pi_{i\to\omega}Q_i(K)\big)
\overset{\eqref{eq:IL_aux5}}\leq
\bar\mm_\omega\big({\rm cl}_{O(\bar X_\omega)}\big(\bar\Phi(K)\big)\big)
\leq\mm_\omega(C),
\end{split}\]
where in the last inequality we used the fact that \(\pi^{-1}(C)\)
is closed (by continuity of \(\pi\)) and thus
\({\rm cl}_{O(\bar X_\omega)}\big(\bar\Phi(K)\big)\subset
{\rm cl}_{O(\bar X_\omega)}\big(\bar\Phi\big(\Phi^{-1}(C)\big)\big)
\subset\pi^{-1}(C)\). Thanks to the inner regularity of \(\mm_Y\),
we deduce that \(\mm_Y\big(\Phi^{-1}(C)\big)\leq\mm_X(C)\),
which in turn yields \(\Phi_*\mm_Y\leq\mm_X\) as a consequence of the
inner regularity of the measures \(\Phi_*\mm_Y\) and \(\mm_X\).
\\
{\color{blue}\textsc{Step 5.}} We are now in a position to conclude the
proof of the statement. First, we claim that if \(p_X\notin{\rm spt}(\mm_X)\),
then no pointed Polish metric measure space projects on the inverse system in
\eqref{eq:inverse_system}, thus in particular such inverse system does not admit
an inverse limit in the category of pointed Polish metric measure spaces. We
argue by contradiction: suppose there exists a pointed Polish metric measure
space \((Y,d_Y,\mm_Y,p_Y)\) that projects on the given inverse system, via
some morphisms \(Q_i\colon Y\to X_i\). Since \(p_X\notin{\rm spt}(\mm_X)\),
we can choose a radius \(r>0\) such that \(\mm_X\big(B(p_X,r)\big)=0\).
By using what we proved in \textsc{Step 4}, we obtain
\[
\mm_Y\big(B(p_Y,r)\big)\leq\mm_Y\big(\Phi^{-1}\big(B(p_X,r)\big)\big)
\leq\mm_X\big(B(p_X,r)\big)=0.
\]
This implies that \(p_Y\notin{\rm spt}(\mm_Y)\), which leads to a contradiction.
Hence, the claim is proven.

Conversely, suppose \(p_X\in{\rm spt}(\mm_X)\). Then \textsc{Step 3}
grants that \(\big((X,d_X,\mm_X,p_X),\{P_i\}_{i\in\N}\big)\) projects on
the inverse system in \eqref{eq:inverse_system}, while \textsc{Step 4}
shows the validity of the universal property. All in all, we have proven
that \(\big((X,d_X,\mm_X,p_X),\{P_i\}_{i\in\N}\big)\) is the inverse limit
of the inverse system in \eqref{eq:inverse_system}. Therefore, the
statement is finally achieved.
\end{proof}
\begin{remark}{\rm
Under the assumption of Theorem \ref{thm:IL}, the following implication
trivially holds:
\begin{equation}\label{eq:implic_belong_spt}
p_X\in{\rm spt}(\mm_X)\quad\Longrightarrow\quad p_\omega\in{\rm spt}(\mm_\omega).
\end{equation}
Note that the fact that \(p_\omega\) belongs to the support of \(\mm_\omega\)
can be equivalently characterised as:
\[
p_\omega\in{\rm spt}(\mm_\omega)\quad\Longleftrightarrow\quad
\lim_{i\to\omega}\mm_i\big(B(p_i,r)\big)>0,\text{ for every }r>0.
\]
It follows from the inclusions
\(B\big([p_i],r\big)\subset\Pi_{i\to\omega}B(p_i,2r)\subset
B\big([p_i],3r\big)\), which hold for all \(r\).

On the other hand, we are not aware of any alternative characterisation
of the fact that \(p_X\) belongs to \({\rm spt}(\mm_X)\). Nevertheless,
in the ensuing Example \ref{ex:p_X_not_spt_m_X} we will see that the converse implication of the one in \eqref{eq:implic_belong_spt} might fail.
\fr}\end{remark}
\begin{example}\label{ex:p_X_not_spt_m_X}{\rm
Let us denote by \((e_j)_{j\in\N}\) the canonical basis of \(\ell^\infty\),
namely, \(e_j\coloneqq(\delta_{ij})_{i\in\N}\). Given any \(i\in\N\), we
define the subset \(X_i\subset\ell^\infty\) in the following way:
\[
X_i\coloneqq\bigcup_{k=1}^\infty X_i^k,\quad\text{ where we set }
X_i^k\coloneqq\bigg\{\frac{e_j}{k}\;\bigg|\;j=1,\ldots,2^i\bigg\}
\text{ for every }k\in\N.
\]
Moreover, we define \(p_i\coloneqq e_1/i\), while we set
\(d_i(x,y)\coloneqq\|x-y\|_{\ell^\infty}\) for every \(x,y\in X_i\) and
\[
\mm_i\coloneqq\sum_{k=1}^\infty\frac{1}{2^k}\bigg(
\sum_{j=1}^{2^i}\frac{1}{2^i}\,\delta_{e_j/k}\bigg)\in\mathscr P(X_i).
\]
The bonding maps \(\{P_{ij}\}_{i\leq j}\) are given as follows:
for any \(i\in\N\), we set \(P_{i,i+1}\colon X_{i+1}\to X_i\) as
\[
P_{i,i+1}\big(e_{2j-1}/k\big)=P_{i,i+1}\big(e_{2j}/k\big)\coloneqq
e_j/k,\quad\text{ for every }k\in\N\text{ and }j=1,\ldots,2^i,
\]
while we set \(P_{ij}\coloneqq P_{i,i+1}\circ\ldots\circ P_{j-1,j}
\colon X_j\to X_i\) for every \(i,j\in\N\) with \(i<j\). Observe that
the maps \(P_{ij}\) are \(1\)-Lipschitz and satisfy
\((P_{ij})_*\mm_j=\mm_i\). Hence, the sequence
\(\big\{(X_i,d_i,\mm_i,p_i)\big\}_{i\in\N}\) together with the maps
\(\{P_{ij}\}_{i\leq j}\) form an inverse system in the category of
pointed Polish metric measure spaces. Let \((X,d_X,\mm_X,p_X)\) be
as in Theorem \ref{thm:IL}. Then we claim that
\begin{equation}\label{eq:p_X_not_spt_m_X_claim}
{\rm spt}(\mm_\omega)=\{p_\omega\},\qquad{\rm spt}(\mm_X)=\emptyset.
\end{equation}

Let us first show that \(p_\omega\in{\rm spt}(\mm_\omega)\). Given any \(r>0\),
choose \(\bar k\in\N\) such that \(2/\bar k<r/2\). Then for any
\(i\geq\bar k\) we have \(\mm_i\big(B(p_i,r/2)\big)\geq\sum_{k\geq\bar k}2^{-k}
=2^{-\bar k+1}\), so that
\[
\mm_\omega\big(B(p_\omega,r)\big)\geq
\lim_{i\to\omega}\mm_i\big(B(p_i,r/2)\big)\geq\frac{1}{2^{\bar k-1}}>0.
\]
By the arbitrariness of \(r>0\), we deduce that
\(p_\omega\in{\rm spt}(\mm_\omega)\). Now fix any
\(x=[[x_i]]\in X_\omega\setminus\{p_\omega\}\). We aim to show that
\(p_\omega\notin{\rm spt}(\mm_\omega)\). Since
\(\lim_{i\to\omega}d_i(x_i,p_i)=d_\omega(x,p_\omega)>0\), we can find
\(\bar k\in\N\) with \(x_i\in\bigcup_{k\leq\bar k}X_i^k\) for
\(\omega\)-a.e.\ \(i\). Then there exist \(k_0\in\{1,\ldots,\bar k\}\)
and \(S\in\omega\) such that \(x_i\in X_i^{k_0}\) for every \(i\in S\).
Given that \(B\big(x_i,1/k_0(k_0+1)\big)=\{x_i\}\) and
\(\mm_i\big(\{x_i\}\big)=2^{-k_0-i}\) for every \(i\in S\),
we conclude that \(\mm_\omega\big(B\big(x,1/k_0(k_0+2)\big)\big)\leq
\lim_{i\to\omega}\mm_i\big(B\big(x_i,1/k_0(k_0+1)\big)\big)=0\),
thus proving that \(x\notin{\rm spt}(\mm_\omega)\). All in all,
the first part of \eqref{eq:p_X_not_spt_m_X_claim} is proven.
To prove the second one, it is sufficient to show that
\(\mm_\omega\big(\{p_\omega\}\big)=0\), since this implies that
\(\mm_X=0\) and thus \({\rm spt}(\mm_X)=\emptyset\). To this aim,
fix any \(r>0\). Given that \(\mm_i\big(B(p_i,r)\big)\leq 2^{2-1/r}\)
for all \(i\in\N\) with \(1/i<r\), we deduce that
\(\mm_\omega\big(\{p_\omega\}\big)\leq\lim_{i\to\omega}
\mm_i\big(B(p_i,r)\big)\leq 2^{2-1/r}\). By letting \(r\searrow 0\),
we can finally conclude that \(\mm_\omega\big(\{p_\omega\}\big)=0\),
as claimed above.
\fr}\end{example}
\appendix
\numberwithin{equation}{chapter}
\chapter{Prokhorov theorem for ultralimits}\label{s:Prokhorov}
In this appendix, we obtain a variant for ultralimits of the celebrated
Prokhorov theorem; see Theorem \ref{thm:Prokhorov}.
Even though we only needed one of the two implications
(in the proof of Theorem \ref{thm:pre-Gromov}), we prove
the full result, since we believe it might be
of independent interest.
\begin{lemmaAPP}\label{lem:asympt_Cauchy}
Let \((X,d)\) be a complete metric space.
Let \((x_i)_{i\in\N}\subset X\) be an
\emph{asymptotically Cauchy} sequence, meaning that
for every \(\varepsilon>0\) there exists \(S\in\omega\)
such that \(d(x_i,x_j)\leq\varepsilon\) holds for every \(i,j\in S\).
Then the ultralimit \(x\coloneqq\lim_{i\to\omega}x_i\in X\) exists.
\end{lemmaAPP}
\begin{proof}
Fix any sequence \(\varepsilon_k\searrow 0\). Given any \(k\in\N\),
we can find \(S'_k\in\omega\) such that \(d(x_i,x_j)\leq\varepsilon_k\)
for every \(i,j\in S'_k\). Define \(S_1\coloneqq S'_1\) and
\[
S_k\coloneqq S'_k\cap S_{k-1}\cap
\big\{i\in\N\,:\,i>\min(S_{k-1})\big\}\in\omega,
\quad\text{ for every }k\geq 2.
\]
Then the sequence \((i_k)_{k\in\N}\), given by
\(i_k\coloneqq\min(S_k)\), is strictly increasing.
Since \(d(x_{i_\ell},x_{i_k})\leq\varepsilon_k\) for all
\(\ell\geq k\), we deduce that \((x_{i_k})_{k\in\N}\) is
a Cauchy sequence, thus it admits a limit \(x\in X\). We
claim that \(x=\lim_{i\to\omega}x_i\). In order to prove it,
let \(\varepsilon>0\) be fixed. Pick any \(k\in\N\) such that
\(\varepsilon_k\leq\varepsilon\) and \(d(x_{i_k},x)\leq\varepsilon\).
Since \(d(x_i,x_{i_k})\leq\varepsilon_k\leq\varepsilon\) for
all \(i\in S_k\), we have that \(d(x_i,x)\leq 2\varepsilon\) for
\(\omega\)-a.e.\ \(i\). By arbitrariness of \(\varepsilon\),
we conclude that \(x=\lim_{i\to\omega}x_i\), as required.
\end{proof}

Given a Polish metric space \((X,d)\) and a constant \(\lambda>0\),
we define the family \(\mathcal M_\lambda(X)\) as
\[
\mathcal M_\lambda(X)\coloneqq\big\{\mu\geq 0\text{ Borel measure on }
X\;\big|\;\mu(X)\leq\lambda\big\}.
\]
By \emph{weak topology} in duality with \(C_b(X)\) we mean the
coarsest topology on \(\mathcal M_\lambda(X)\) such that the
function \(\mathcal M_\lambda(X)\ni\mu\to\int f\,\d\mu\in\R\)
is continuous for every \(f\in C_b(X)\). The weak topology is
metrised by the following distance:
\[
\delta(\mu,\nu)\coloneqq\sum_{n\in\N}\frac{1}{2^n}
\bigg|\int f_n\,\d(\mu-\nu)\bigg|,\quad\text{ for every }
\mu,\nu\in\mathcal M_\lambda(X),
\]
where \((f_n)_{n\in\N}\) is a suitably chosen sequence of
bounded, Lipschitz functions from \(X\) to \(\R\) satisfying
\(\|f_n\|_{C_b(X)}=1\) for all \(n\in\N\). For instance,
one can take as \((f_n)_n\) the countable family
\(\big\{g/\|g\|_{C_b(X)}\,\big|\,g\in\mathcal F,\,g\neq 0\big\}\),
where we define
\[
\mathcal F\coloneqq\big\{\pm\max\{g_1,\ldots,g_k\}\;\big|
\;k\in\N,\,g_1,\ldots,g_k\in\tilde{\mathcal F}\big\},
\]
while (for some given countable, dense subset \(D\) of \(X\)) we define
\[
\tilde{\mathcal F}\coloneqq\Big\{\max\big\{\alpha-\beta\,d(\cdot,y),
\gamma\big\}\;\Big|\;\alpha,\beta,\gamma\in\mathbb Q,\,y\in D\Big\}.
\]
Observe that, given any \(f\in C_b(X)\),
\(\mu\in\mathcal M_\lambda(X)\), and \(\varepsilon>0\),
there exist functions \(g,g'\in\mathcal F\) such that
\(g\leq f\leq g'\) and \(\int g'\,\d\mu-\varepsilon\leq\int f\,\d\mu
\leq\int g\,\d\mu+\varepsilon\).
\begin{remarkAPP}\label{rmk:Polish_P(X)}{\rm
Given any Polish metric space \((X,d)\), it holds that
\(\big(\mathscr P(X),\delta\big)\) is a Polish metric space,
where \(\mathscr P(X)\) stands for the space of all
Borel probability measures on \((X,d)\). See, for instance,
\cite[Theorem 15.15]{AliprantisBorder99}.
\fr}\end{remarkAPP}
\begin{propositionAPP}\label{prop:char_weak_ultraconv_meas}
Let \((X,d)\) be a Polish metric space. Let
\(\mu,\mu_i\in\mathcal M_\lambda(X)\) for some \(\lambda>0\).
Then it holds \(\mu=\lim_{i\to\omega}\mu_i\) with
respect to the weak convergence in \(\mathcal M_\lambda(X)\)
if and only if
\begin{equation}\label{eq:char_weak_ultraconv_meas}
\int f\,\d\mu=\lim_{i\to\omega}\int f\,\d\mu_i,
\quad\text{ for every }f\in C_b(X).
\end{equation}
\end{propositionAPP}
\begin{proof}
Suppose \(\mu=\lim_{i\to\omega}\mu_i\) with respect to
the weak convergence. Let \(f\in C_b(X)\) be fixed.
Given \(\varepsilon>0\), we can find \(g,g'\in\mathcal F\)
such that \(g\leq f\leq g'\) and \(\int g'\,\d\mu-\varepsilon
\leq\int f\,\d\mu\leq\int g\,\d\mu+\varepsilon\). Choose \(n,m\in\N\)
with \(f_n=g/\|g\|_{C_b(X)}\) and \(f_m=g'/\|g'\|_{C_b(X)}\).
Then for any \(i\in\N\) it holds
\[\begin{split}
\int f\,\d\mu_i-\int f\,\d\mu&\geq
\|g\|_{C_b(X)}\bigg(\int f_n\,\d(\mu_i-\mu)\bigg)-\varepsilon
\geq-2^n\,\|g\|_{C_b(X)}\,\delta(\mu_i,\mu)-\varepsilon,\\
\int f\,\d\mu_i-\int f\,\d\mu&\leq
\|g'\|_{C_b(X)}\bigg(\int f_m\,\d(\mu_i-\mu)\bigg)+\varepsilon
\leq 2^m\,\|g'\|_{C_b(X)}\,\delta(\mu_i,\mu)+\varepsilon.
\end{split}\]
All in all, for any \(i\in\N\) we have that
\[
\bigg|\int f\,\d\mu_i-\int f\,\d\mu\bigg|
\leq c\,\delta(\mu_i,\mu)+\varepsilon,\quad\text{ where }
c\coloneqq\max\big\{2^n\,\|g\|_{C_b(X)},2^m\,\|g'\|_{C_b(X)}\big\}.
\]
Therefore, we conclude that
\(\lim_{i\to\omega}\big|\int f\,\d\mu_i-\int f\,\d\mu\big|
\leq c\lim_{i\to\omega}\delta(\mu_i,\mu)+\varepsilon=\varepsilon\),
whence the desired property \eqref{eq:char_weak_ultraconv_meas}
follows thanks to the arbitrariness of \(\varepsilon\).

Conversely, suppose \eqref{eq:char_weak_ultraconv_meas}
is verified. Let \(\varepsilon>0\) be fixed. Pick \(N\in\N\)
such that \(1/2^N\leq\varepsilon/2\). It follows from
the assumption \eqref{eq:char_weak_ultraconv_meas} that
\[
S\coloneqq\bigg\{i\in\N\;\bigg|\;\bigg|\int f_n\,\d(\mu_i-\mu)\bigg|
\leq\frac{\varepsilon}{2}\;\text{ for every }n=1,\ldots,N\bigg\}
\in\omega.
\]
Therefore, we finally conclude that
\[
\delta(\mu_i,\mu)\leq\sum_{n=1}^N\frac{1}{2^n}
\bigg|\int f_n\,\d(\mu_i-\mu)\bigg|+\sum_{n>N}\frac{1}{2^n}
\leq\frac{\varepsilon}{2}+\frac{1}{2^N}\leq\varepsilon,
\quad\text{ for every }i\in S,
\]
so that \(\delta(\mu_i,\mu)\leq\varepsilon\) for
\(\omega\)-a.e.\ \(i\), thus \(\mu=\lim_{i\to\omega}\mu_i\)
with respect to the weak convergence.
\end{proof}

Alternatively, we could have proven Proposition
\ref{prop:char_weak_ultraconv_meas} by just observing that
for any given measure \(\mu\in\mathcal M_\lambda(X)\) the
collection \(\big\{U_{f,\eps}\,:\,f\in C_b(X),\,\eps>0\big\}\),
which is given by
\[
U_{f,\eps}\coloneqq\bigg\{\nu\in\mathcal M_\lambda(X)\;\bigg|\;
\bigg|\int f\,\d\mu-\int f\,\d\nu\bigg|<\eps\bigg\},
\]
forms a neighbourhood basis for \(\mu\) with respect to the weak topology.
\begin{corollaryAPP}\label{cor:sc_weak_ultraconv}
Let \((X,d)\) be a Polish metric space. Let \(\lambda>0\)
be given. Let \(\mu,\mu_i\in\mathcal M_\lambda(X)\) satisfy
\(\mu=\lim_{i\to\omega}\mu_i\) with respect to the weak
convergence in \(\mathcal M_\lambda(X)\). Then it holds that
\begin{subequations}\begin{align}
\label{eq:sc_weak_ultraconv1}
\mu(U)\leq\lim_{i\to\omega}\mu_i(U)&,\quad\text{ for every open set }
U\subset X,\\
\label{eq:sc_weak_ultraconv2}
\mu(C)\geq\lim_{i\to\omega}\mu_i(C)&,\quad\text{ for every closed set }
C\subset X.
\end{align}\end{subequations}
In particular, it holds that \(\mu(X)=\lim_{i\to\omega}\mu_i(X)\).
\end{corollaryAPP}
\begin{proof}
To prove \eqref{eq:sc_weak_ultraconv1}, pick a sequence
\((f_n)_{n\in\N}\subset C_b(X)\) with
\(f_n\nearrow\nchi_U\) everywhere on \(X\). Thanks to
Proposition \ref{prop:char_weak_ultraconv_meas}, we see that
\[
\int f_n\,\d\mu=\lim_{i\to\omega}\int f_n\,\d\mu_i
\leq\lim_{i\to\omega}\mu_i(U),\quad\text{ for every }n\in\N,
\]
whence by letting \(n\to\infty\) and using the monotone convergence
theorem we can conclude that
\[
\mu(U)=\lim_{n\to\infty}\int f_n\,\d\mu\leq\lim_{i\to\omega}\mu_i(U).
\]
This shows the validity of \eqref{eq:sc_weak_ultraconv1}.
To prove \eqref{eq:sc_weak_ultraconv2}, we can argue in a similar way.
Pick a sequence \((g_n)_{n\in\N}\subset C_b(X)\) such that
\(g_n\leq 1\) for all \(n\in\N\) and \(g_n\searrow\nchi_C\) everywhere
on \(X\). By using Proposition \ref{prop:char_weak_ultraconv_meas}
we get \(\int g_n\,\d\mu=\lim_{i\to\omega}\int g_n\,\d\mu_i\geq
\lim_{i\to\omega}\mu_i(C)\) for all \(n\in\N\), so by
dominated convergence theorem we conclude that \(\mu(C)=
\lim_{n\to\infty}\int g_n\,\d\mu\geq\lim_{i\to\omega}\mu_i(C)\).
This shows the validity of \eqref{eq:sc_weak_ultraconv2}.
The final statement immediately follows.
\end{proof}

A subset \(T\) of a metric space \((X,d)\) is said to be
\emph{\(\varepsilon\)-totally bounded} (for some \(\varepsilon>0\))
provided it is contained in the union of finitely many balls
of radius \(\varepsilon\). Observe that a subset
of \(X\) is totally bounded if and only if it is
\(\varepsilon\)-totally bounded for every \(\varepsilon>0\).
\begin{theoremAPP}[Prokhorov theorem for ultralimits]\label{thm:Prokhorov}
Let \((X,d)\) be a given Polish metric space. Let
\((\mu_i)_{i\in\N}\) be a sequence of Borel measures
on \((X,d)\) such that
\(\lambda\coloneqq\lim_{i\to\omega}\mu_i(X)<+\infty\).
Then the following conditions are equivalent:
\begin{itemize}
\item[\(\rm i)\)] Given any \(\varepsilon>0\), there exists an
\(\varepsilon\)-totally bounded, Borel subset \(T\) of \(X\)
such that
\[
\lim_{i\to\omega}\mu_i(X\setminus T)\leq\varepsilon.
\]
\item[\(\rm ii)\)] The ultralimit
\(\mu\coloneqq\lim_{i\to\omega}\mu_i\) exists
(with respect to the weak convergence in duality with $C_b(X)$).
\end{itemize}
\end{theoremAPP}
\begin{proof}
The proof goes along the lines of the proof of Prokhorov theorem,
see for example \cite{Bogachev07}.\\
{\color{blue}\({\rm ii)}\Longrightarrow{\rm i)}\).}
Suppose \(\rm ii)\) holds and fix \(\varepsilon>0\).
Choose any dense sequence \((x_j)_{j\in\N}\) in \(X\).
Then we claim that
\begin{equation}\label{eq:Prokhorov_aux1}
\lim_{i\to\omega}\mu_i\bigg(\bigcup_{j=1}^N
B(x_j,\varepsilon)\bigg)\geq\lambda-\varepsilon,
\quad\text{ for some }N\in\N.
\end{equation}
We argue by contradiction: suppose that \(\lim_{i\to\omega}
\mu_i\big(\bigcup_{j=1}^N B(x_j,\varepsilon)\big)<
\lambda-\varepsilon\) for every \(N\in\N\).  Hence, Corollary
\ref{cor:sc_weak_ultraconv} grants that
\(\mu\big(\bigcup_{j=1}^N B(x_j,\varepsilon)\big)\leq
\lim_{i\to\omega}\mu_i\big(\bigcup_{j=1}^N B(x_j,\varepsilon)\big)
<\lambda-\varepsilon\), whence by letting \(N\to\infty\) we get
\(\mu(X)\leq\lambda-\varepsilon\); we are using the
fact that \(X=\bigcup_{j\in\N}B(x_j,\varepsilon)\) by density
of the sequence \((x_j)_{j\in\N}\). On the other hand, we also
know from Corollary \ref{cor:sc_weak_ultraconv} that
\(\mu(X)=\lim_{i\to\omega}\mu_i(X)=\lambda\), thus
leading to a contradiction. This shows the validity
of \eqref{eq:Prokhorov_aux1}. Therefore, the Borel set
\(T\coloneqq\bigcup_{j=1}^N B(x_j,\varepsilon)\)
(where \(N\in\N\) is given as in \eqref{eq:Prokhorov_aux1})
is \(\varepsilon\)-totally bounded and satisfies
\(\lim_{i\to\omega}\mu_i(X\setminus T)=\lim_{i\to\omega}\mu_i(X)
-\lim_{i\to\omega}\mu_i(T)\leq\varepsilon\), thus yielding
\({\rm i})\).\\
{\color{blue}\({\rm i)}\Longrightarrow{\rm ii)}\).}
Suppose \({\rm i)}\) holds. If \(\lambda=0\), then
we trivially have \(\lim_{i\to\omega}\mu_i=0\) with respect
to the weak convergence, so that \(\rm ii)\) is verified.
Hence, we can assume that \(\lambda>0\). We claim that
\begin{equation}\label{eq:Prokhorov_aux2}
(\tilde\mu_i)_{i\in\N}\;\text{ is asymptotically Cauchy in }
\big(\mathscr P(X),\delta\big),\quad\text{ where we set }
\tilde\mu_i\coloneqq\frac{\mu_i}{\mu_i(X)}.
\end{equation}
In order to prove it, fix any \(\varepsilon\in(0,\lambda)\).
Pick some \(N\in\N\) for which \(1/2^N\leq\varepsilon\). Define
\[
\eta\coloneqq\frac{\varepsilon}{1+2\max\big\{{\rm Lip}(f_1),
\ldots,{\rm Lip}(f_N)\big\}}.
\]
By using \({\rm i})\), we find an \(\eta\)-totally bounded
set \(T\) such that \(\lim_{i\to\omega}\mu_i(X\setminus T)
/\mu_i(X)<\eta\leq\varepsilon\). Write \(T\) as a disjoint union
\(F_1\cup\ldots\cup F_k\) of Borel sets \(F_1,\ldots,F_k\neq\emptyset\)
having diameter smaller than \(2\,\eta\). Pick \(y_j\in F_j\) for any
\(j=1,\ldots,k\). Calling \(S'\coloneqq\big\{i\in\N\,:\,
\mu_i(X\setminus T)/\mu_i(X)\leq\varepsilon\big\}\in\omega\), we define
the Borel measures \((\nu_i)_{i\in S'}\) on \(X\) as
\[
\nu_i\coloneqq\frac{\sum_{j=1}^k\mu_i(F_j)\,\delta_{y_j}}{\mu_i(X)}
\in\mathcal M_1(X),\quad\text{ for every }i\in S'.
\]
If we set \(\lambda_j\coloneqq\lim_{i\to\omega}\mu_i(F_j)\) for
every \(j=1,\ldots,k\), then it holds that
\[
S\coloneqq\bigg\{i\in S'\;\bigg|\;\bigg|\frac{\mu_i(F_j)}{\mu_i(X)}
-\frac{\lambda_j}{\lambda}\bigg|\leq\frac{\varepsilon}{k}\;
\text{ for every }j=1,\ldots,k\bigg\}\in\omega.
\]
Let us now define
\[
\nu\coloneqq\frac{\sum_{j=1}^k\lambda_j\,\delta_{y_j}}{\lambda}
\in\mathcal M_1(X).
\]
Given any \(n=1,\ldots,N\) and \(i\in S\), we may estimate
\[\begin{split}
\bigg|\int f_n\,\d(\tilde\mu_i-\nu_i)\bigg|&\leq
\frac{1}{\mu_i(X)}\bigg|\int_{X\setminus T}f_n\,\d\mu_i\bigg|
+\frac{1}{\mu_i(X)}\sum_{j=1}^k\bigg|\int_{F_j}f_n\,\d\mu_i
-\mu_i(F_j)\,f_n(y_j)\bigg|\\
&\leq\frac{\mu_i(X\setminus T)}{\mu_i(X)}+
\frac{1}{\mu_i(X)}\sum_{j=1}^k\int_{F_j}\big|f_n(x)-f_n(y_j)\big|
\,\d\mu_i(x)\\
&\leq\varepsilon+\frac{{\rm Lip}(f_n)}{\mu_i(X)}\sum_{j=1}^k
\int_{F_j}d(x,y_j)\,\d\mu_i(x)\leq\varepsilon+
\frac{2\,\eta\,{\rm Lip}(f_n)}{\mu_i(X)}\sum_{j=1}^k\mu_i(F_j)\\
&\leq\varepsilon+\frac{\varepsilon\,\mu_i(T)}{\mu_i(X)}
\leq 2\,\varepsilon.
\end{split}\]
Moreover, it holds that
\[
\bigg|\int f_n\,\d(\nu_i-\nu)\bigg|=
\bigg|\sum_{j=1}^k f_n(y_j)\bigg(\frac{\mu_i(F_j)}{\mu_i(X)}
-\frac{\lambda_j}{\lambda}\bigg)\bigg|
\leq\sum_{j=1}^k\big|f_n(y_j)\big|
\bigg|\frac{\mu_i(F_j)}{\mu_i(X)}-\frac{\lambda_j}{\lambda}\bigg|
\leq\varepsilon.
\]
All in all, we obtain
\(\big|\int f_n\,\d(\tilde\mu_i-\nu)\big|\leq
\big|\int f_n\,\d(\tilde\mu_i-\nu_i)\big|+
\big|\int f_n\,\d(\nu_i-\nu)\big|\leq 3\,\varepsilon\)
for every \(i\in S\) and \(n=1,\ldots,N\), thus accordingly
we conclude that
\[\begin{split}
\delta(\tilde\mu_i,\tilde\mu_\ell)&=\sum_{n=1}^N
\frac{1}{2^n}\bigg|\int f_n\,\d(\tilde\mu_i-\tilde\mu_\ell)\bigg|+
\sum_{n>N}\frac{1}{2^n}\bigg|\int f_n\,\d(\tilde\mu_i
-\tilde\mu_\ell)\bigg|\\
&\leq\sum_{n=1}^N
\frac{1}{2^n}\bigg|\int f_n\,\d(\tilde\mu_i-\nu)\bigg|+
\sum_{n=1}^N
\frac{1}{2^n}\bigg|\int f_n\,\d(\nu-\tilde\mu_\ell)\bigg|
+\sum_{n>N}\frac{1}{2^n}\\
&\leq 6\,\varepsilon+\frac{1}{2^N}\leq 7\,\varepsilon,
\quad\text{ for every }i,\ell\in S.
\end{split}\]
This shows the claim \eqref{eq:Prokhorov_aux2}. Therefore,
by using Lemma \ref{lem:asympt_Cauchy} and Remark
\ref{rmk:Polish_P(X)} we get existence of a measure
\(\tilde\mu\in\mathscr P(X)\) such that \(\tilde\mu=
\lim_{i\to\omega}\tilde\mu_i\) with respect to the weak
convergence. Let us define \(\mu\coloneqq\lambda\,\tilde\mu
\in\mathcal M_\lambda(X)\). Then it can be readily checked that
\(\mu=\lim_{i\to\omega}\mu_i\) with respect to the weak
convergence, whence \(\rm ii)\) follows. The proof of the
statement is achieved.
\end{proof}
\begin{propositionAPP}[Sharpness of Prokhorov Theorem
\ref{thm:Prokhorov}]\label{prop:Prokh_sharp}
Let \(\omega\) be a non-principal ultrafilter on \(\N\)
satisfying the property that is described in Lemma \ref{lem:non_p-pt}.
Then there exists a Polish metric space \((X,d)\) and a sequence
\((\mu_i)_{i\in\N}\subset\mathscr P(X)\) of measures such that
\(\mu=\lim_{i\to\omega}\mu_i\) (with respect to the weak
convergence) for some \(\mu\in\mathscr P(X)\), but
\begin{equation}\label{eq:Prokh_sharp_cl2}
\lim_{i\to\omega}\mu_i(K)=0,\quad\text{ for every compact set }
K\subset X.
\end{equation}
In particular, this shows that Theorem \ref{thm:Prokhorov}
might fail if in its item \(\rm i)\) we replace
``\(\varepsilon\)-totally bounded'' with ``compact''.
\end{propositionAPP}
\begin{proof}
Let us define the subset \(X\) of \(\ell^\infty\) as
\[
X\coloneqq\{0\}\cup\bigg\{\frac{e_j}{n}\;\bigg|\;j,n\in\N\bigg\},
\]
where \(e_j\coloneqq(\delta_{jk})_{k\in\N}\) stands for the
\(j\)-th element of the canonical basis of \(\ell^\infty\).
Calling \(d\) the distance on \(X\) induced by the
\(\ell^\infty\)-norm, we have that \((X,d)\) is a Polish metric
space. Fix a partition \(\{N_n\}_{n\in\N}\) of \(\N\) into
infinite sets as in Lemma \ref{lem:non_p-pt}. Say that
each set \(N_n\) is written as \(\{m^n_k\}_{k\in\N}\), where
\(m^n_1<m^n_2<m^n_3<\ldots\). Let us then define
\(\mu_{m^n_k}\in\mathscr P(X)\) as
\[
\mu_{m^n_k}\coloneqq\frac{1}{k}\sum_{j=1}^k\delta_{e_j/n},
\quad\text{ for every }n,k\in\N.
\]
First of all, we claim that the resulting sequence
\((\mu_i)_{i\in\N}\) of measures satisfies
\begin{equation}\label{eq:Prokh_sharp_cl1}
\delta_0=\lim_{i\to\omega}\mu_i,\quad
\text{ with respect to the weak convergence in duality with }
C_b(X).
\end{equation}
In order to prove it, let \(f\in C_b(X)\) and \(\varepsilon>0\)
be fixed. By continuity, there exists \(r>0\) such that
\(\big|f(x)-f(0)\big|<\varepsilon\) for every \(x\in B(0,r)\).
Pick any \(\bar n\in\N\) such that \(1/\bar n\leq r\), so that we
have \(e_j/n\in B(0,r)\) for all \(n>\bar n\) and \(j\in\N\).
Consequently, for any \(n>\bar n\) and \(k\in\N\) it holds
\[
\bigg|\int f\,\d\mu_{m^n_k}-\int f\,\d\delta_0\bigg|
=\bigg|\frac{1}{k}\sum_{j=1}^k f(e_j/n)-f(0)\bigg|
\leq\frac{1}{k}\sum_{j=1}^k\big|f(e_j/n)-f(0)\big|
\leq\varepsilon.
\]
In other words, one has
\(\big|\int f\,\d\mu_i-\int f\,\d\delta_0\big|\leq\varepsilon\)
for every \(i\in\N\setminus N_{\bar n}\). Given that
\(\N\setminus N_{\bar n}\in\omega\) (recall Lemma
\ref{lem:non_p-pt}), thanks to the arbitrariness of
\(\varepsilon>0\) we can conclude that \eqref{eq:Prokh_sharp_cl1}
holds.

Let us now prove that \eqref{eq:Prokh_sharp_cl2} is verified.
Let \(K\subset X\) be a given compact set. It can be readily
checked that necessarily
\(\lim_{j\to\infty}{\rm diam}(K\cap\R e_j)=0\). We argue by
contradiction: suppose that \(\lambda\coloneqq\lim_{i\to\omega}
\mu_i(K)>0\). Then there exists \(S\in\omega\) such that
\(\mu_i(K)\geq\lambda/2\) for all \(i\in S\). By Lemma
\ref{lem:non_p-pt} we can find \(n\in\N\) for which
\(S\cap N_n\) is infinite (otherwise,
\(N_m\setminus(\N\setminus S)\) would be finite for all
\(m\in\N\) and thus \(\N\setminus S\in\omega\), contradicting
the fact that \(S\in\omega\)). Choose a subsequence
\((k_\ell)_{\ell\in\N}\) such that \(m^n_{k_\ell}\in S\) holds
for every \(\ell\in\N\). Now let us fix some \(\bar\ell\in\N\)
satisfying \(K\cap\R e_{k_\ell}\subset B(0,1/n)\) for all
\(\ell\geq\bar\ell\). Therefore, one has that
\(\mu_{m^n_{k_\ell}}(K)\leq k_{\bar\ell}/k_\ell\) for every
\(\ell\geq\bar\ell\). Since \(\lim_{\ell\to\infty}k_{\bar\ell}
/k_\ell=0\), we see that \(\inf_{i\in S}\mu_i(K)\leq
\inf_{\ell\in\N}\mu_{m^n_{k_\ell}}(K)=0\), which leads to
a contradiction with the fact that
\(\inf_{i\in S}\mu_i(K)\geq\lambda/2\).
Hence, \eqref{eq:Prokh_sharp_cl2} is proven.
\end{proof}
\chapter{Tangents to pointwise doubling metric measure spaces}
\label{s:tg_cone}
With the notion of wpmGH-convergence at our disposal, we can now introduce
a natural definition of \emph{tangent cone} to a Polish metric measure space
at a fixed point of its support.
\begin{definitionAPP}[wpmGH-tangent cone]
Let \((X,d,\mm,p)\) be a pointed Polish metric measure space
with \(p\in\spt(\mm)\). Then we define \({\rm Tan}(X,d,\mm,p)\) as the
family of all those pointed Polish metric measure spaces \((Y,d_Y,\mm_Y,q)\)
satisfying \(q\in\spt(\mm_Y)\) and
\[
(X,d/{r_i},\mm^p_{r_i},p)\longrightarrow(Y,d_Y,\mm_Y,q),\quad
\text{ in the wpmGH-sense,}
\]
for some sequence \(r_i\searrow 0\), where the Borel measure
\(\mm_{r_i}^p\geq 0\) over \((X,d/{r_i})\) is given by
\[
\mm_{r_i}^p\coloneqq\frac{\mm}{\mm\big(B_d(p,{r_i})\big)}.
\]
\end{definitionAPP}

Note that, trivially, the identity \(B_{d/r}(x,s)=B_d(x,rs)\)
holds for every \(x\in X\) and \(r,s>0\).
\begin{definitionAPP}[Pointwise doubling]\label{def:ptwse_doubl}
Let \((X,d,\mm)\) be a Polish metric measure space. Then we say that
\((X,d,\mm)\) is \emph{pointwise doubling} provided it holds that
\[
\lims_{r\searrow 0}\frac{\mm\big(B(x,2r)\big)}{\mm\big(B(x,r)\big)}<+\infty,
\quad\text{ for }\mm\text{-a.e.\ }x\in X.
\]
\end{definitionAPP}

By suitably adapting the proof of \cite[Theorem 1.6]{Heinonen01},
one can see that every pointwise doubling space \((X,d,\mm)\) is a
\emph{Vitali space} (in the sense of \cite[Remark 1.13]{Heinonen01}),
so that Lebesgue's differentiation theorem holds. In particular, given any
Borel set \(E\subset X\), we have that \(\mm\)-a.e.\ point
\(x\in E\) is of density \(1\) for \(E\), meaning that
\(\lim_{r\searrow 0}\mm\big(E\cap B(x,r)\big)/\mm\big(B(x,r)\big)=1\).
The following statement about pointwise doubling spaces is proven
in \cite[Lemma 8.3]{Bate12}.
\begin{lemmaAPP}\label{lem:ptwse_doubl_Lusin}
Let \((X,d,\mm)\) be a pointwise doubling Polish metric measure space.
Then there exists a family \((E_n)_{n\in\N}\) of pairwise disjoint Borel
subsets of \(X\) such that \(\mm\big(X\setminus\bigcup_{n\in\N}E_n\big)=0\) and
\(\big(E_n,d|_{E_n\times E_n}\big)\) is metrically doubling for every \(n\in\N\).
\end{lemmaAPP}
\begin{remarkAPP}\label{rmk:rescalings_omega-ubf}{\rm
Suppose \(\lims_{r\searrow 0}\frac{\mm(B(p,2r))}{\mm(B(p,r))}<+\infty\).
Then for any sequence \(r_i\searrow 0\) it holds that
\[
\lim_{i\to\omega}\mm_{r_i}^p\big(B_{d/r_i}(p,R)\big)\leq
\lims_{i\to\infty}\mm_{r_i}^p\big(B_d(p,Rr_i)\big)=
\lims_{i\to\infty}\frac{\mm\big(B_d(p,Rr_i)\big)}{\mm\big(B_d(p,r_i)\big)}
<+\infty,\quad\text{ for any }R>0.
\]
Therefore, the sequence \(\big((X,d/r_i,\mm_{r_i}^p,p)\big)\)
is \(\omega\)-uniformly boundedly finite.
\fr}\end{remarkAPP}
\begin{theoremAPP}[Existence of wpmGH tangents to pointwise doubling spaces]
\label{thm:tg_cone}
Let \((X,d,\mm)\) be a pointwise doubling Polish metric measure space. Let
\(r_i\searrow 0\) be given. Then for \(\mm\)-a.e.\ point \(p\in X\)
the sequence \(\big((X,d/r_i,\mm_{r_i}^p,p)\big)\) is asymptotically
boundedly \(\mm_\omega\)-totally bounded.

In particular, it holds that
\begin{equation}\label{eq:nonempty_Tan}
{\rm Tan}(X,d,\mm,p)\neq\emptyset,\quad\text{ for }\mm\text{-a.e.\ }p\in X.
\end{equation}
\end{theoremAPP}
\begin{proof}
Let \((E_n)_{n\in\N}\) be chosen as in Lemma \ref{lem:ptwse_doubl_Lusin}.
Fix a point \(p\in X\) that (belongs to and) is of density \(1\) for \(E_n\),
for some \(n\in\N\); as observed after Definition \ref{def:ptwse_doubl},
\(\mm\)-almost every point of \(X\) has this property. Now let
\(R,r,\varepsilon>0\) be fixed. Given that \(\big(E_n,d|_{E_n\times E_n}\big)\)
is metrically doubling by Lemma \ref{lem:ptwse_doubl_Lusin}, we can find a
constant \(M\in\N\) such that the following property is satisfied: given any
\(s\in(0,1)\), the ball (in \(E_n\)) of center \(p\) and
radius \(Rs\) can be covered by \(M\) balls of radius \(rs\). Hence,
for any \(i\in\N\) there exist \(x^i_1,\ldots,x^i_M\in E_n\cap B_d(p,Rr_i)\)
such that
\begin{equation}\label{eq:rescal_abmtb}
E_n\cap\bar B_{d/r_i}(p,R)=E_n\cap\bar B_d(p,Rr_i)\subset
\bigcup_{j=1}^M B_d(x^i_j,rr_i)=\bigcup_{j=1}^M B_{d/r_i}(x^i_j,r).
\end{equation}
By using the fact that \(p\) is of density \(1\) for \(E_n\), we thus conclude that
\[\begin{split}
\lim_{i\to\omega}\mm_{r_i}^p\big(\bar B_{d/r_i}(p,R)\setminus{\textstyle
\bigcup_{j=1}^M}B_{d/r_i}(x^i_j,r)\big)
&\overset{\phantom{\eqref{eq:rescal_abmtb}}}\leq
\lims_{i\to\infty}\mm_{r_i}^p\big(\bar B_{d/r_i}(p,R)\setminus{\textstyle
\bigcup_{j=1}^M}B_{d/r_i}(x^i_j,r)\big)\\
&\overset{\eqref{eq:rescal_abmtb}}\leq
\lims_{i\to\infty}\mm_{r_i}^p\big(\bar B_{d/r_i}(p,R)\setminus E_n\big)\\
&\overset{\phantom{\eqref{eq:rescal_abmtb}}}=
\lims_{i\to\infty}\frac{\mm\big(\bar B_d(p,Rr_i)\setminus E_n\big)}
{\mm\big(B_d(p,Rr_i)\big)}\frac{\mm\big(B_d(p,Rr_i)\big)}
{\mm\big(B_d(p,r_i)\big)}=0.
\end{split}\]
This shows that \(\big((X,d/r_i,\mm_{r_i}^p,p)\big)\) is
asymptotically boundedly \(\mm_\omega\)-totally bounded.
The last statement \eqref{eq:nonempty_Tan} now immediately follows
from the first one together with Theorem \ref{thm:Gromov_cpt_main}.
\end{proof}
\addtocontents{toc}{\protect\setcounter{tocdepth}{-1}}
\newpage

\begin{thebibliography}{10}

\bibitem{Ale51}
{\sc A.~D. Alexandrov}, {\em A theorem on triangles in a metric space and some
  of its applications}, Trudy Mat. Inst. Steklov., 38 (1951), pp.~5--23.

\bibitem{AliprantisBorder99}
{\sc C.~Aliprantis and K.~Border}, {\em Infinite {D}imensional {A}nalysis: {A}
  {H}itchhiker's {G}uide}, Studies in Economic Theory, Springer, 1999.

\bibitem{AmbrosioGigliMondinoRajala12}
{\sc L.~Ambrosio, N.~Gigli, A.~Mondino, and T.~Rajala}, {\em Riemannian {R}icci
  curvature lower bounds in metric measure spaces with $\sigma$-finite
  measure}, Trans. Amer. Math. Soc., 367 (2015), pp.~4661--4701.

\bibitem{AmbrosioGigliSavare11-2}
{\sc L.~Ambrosio, N.~Gigli, and G.~Savar{\'e}}, {\em Metric measure spaces with
  {R}iemannian {R}icci curvature bounded from below}, Duke Math. J., 163
  (2014), pp.~1405--1490.

\bibitem{Bate12}
{\sc D.~Bate}, {\em Structure of measures in {L}ipschitz differentiability
  spaces}, Journal of the American Mathematical Society, 28 (2012).

\bibitem{Bogachev07}
{\sc V.~I. Bogachev}, {\em Measure theory. {V}ol. {I}, {II}}, Springer-Verlag,
  Berlin, 2007.

\bibitem{BH99}
{\sc M.~R. Bridson and A.~Haefliger}, {\em Metric spaces of non-positive
  curvature}, vol.~319 of Grundlehren der Mathematischen Wissenschaften,
  Springer-Verlag, Berlin, 1999.

\bibitem{Brue20}
{\sc E.~Bru\`{e}}, {\em Structure of non-smooth spaces with {R}icci curvature
  bounded below}, 2020.
\newblock PhD Thesis, Scuola Normale Superiore di Pisa.

\bibitem{BBI01}
{\sc D.~Burago, Y.~Burago, and S.~Ivanov}, {\em A {C}ourse in {M}etric
  {G}eometry}, vol.~33, 2001.
\newblock Graduate Studies in Math.

\bibitem{BGP92}
{\sc Y.~D. Burago, M.~L. Gromov, and G.~Y. Perel'man}, {\em A. {D}.
  {A}leksandrov spaces with curvature bounded below}, Uspekhi Mat. Nauk, 47
  (1992), pp.~3--51.

\bibitem{Cheeger-Colding97I}
{\sc J.~Cheeger and T.~H. Colding}, {\em On the structure of spaces with
  {R}icci curvature bounded below. {I}}, J. Differential Geom., 46 (1997),
  pp.~406--480.

\bibitem{Cheeger-Colding97II}
\leavevmode\vrule height 2pt depth -1.6pt width 23pt, {\em On the structure of
  spaces with {R}icci curvature bounded below. {II}}, J. Differential Geom., 54
  (2000), pp.~13--35.

\bibitem{Cheeger-Colding97III}
\leavevmode\vrule height 2pt depth -1.6pt width 23pt, {\em On the structure of
  spaces with {R}icci curvature bounded below. {III}}, J. Differential Geom.,
  54 (2000), pp.~37--74.

\bibitem{CheegerKleiner13}
{\sc J.~Cheeger and B.~Kleiner}, {\em Realization of {M}etric {S}paces as
  {I}nverse {L}imits, and {B}ilipschitz {E}mbedding in ${L}^1$}, Geometric and
  Functional Analysis, 23 (2013), pp.~96--133.

\bibitem{Conley-Kechris-Tucker-Drob}
{\sc C.~T. Conley, A.~S. Kechris, and R.~D. Tucker-Drob}, {\em Ultraproducts of
  measure preserving actions and graph combinatorics}, Ergodic Theory Dynam.
  Systems, 33 (2013), pp.~334--374.

\bibitem{Cutland00}
{\sc N.~J. Cutland}, {\em Loeb measures in practice: {R}ecent advances},
  vol.~1751 of Lecture Notes in Mathematics, Springer-Verlag Berlin Heidelberg,
  2000.

\bibitem{Edwards75}
{\sc D.~A. Edwards}, {\em The {S}tructure of {S}uperspace},  (1975).
\newblock Published in: Studies in Topology, Academic Press.

\bibitem{Elek}
{\sc G.~Elek}, {\em Samplings and observables. {I}nvariants of metric measure
  spaces}, 2012.
\newblock Preprint, arXiv:1205.6936.

\bibitem{Fukaya87}
{\sc K.~Fukaya}, {\em Collapsing of {R}iemannian manifolds and eigenvalues of
  {L}aplace operator}, Inventiones mathematicae, 87 (1987), pp.~517--547.

\bibitem{Gigli12}
{\sc N.~Gigli}, {\em On the differential structure of metric measure spaces and
  applications}, Mem. Amer. Math. Soc., 236 (2015), pp.~vi+91.

\bibitem{Gigli-Mondino-Savare}
{\sc N.~Gigli, A.~Mondino, and G.~Savar\'{e}}, {\em Convergence of pointed
  non-compact metric measure spaces and stability of {R}icci curvature bounds
  and heat flows}, Proc. Lond. Math. Soc. (3), 111 (2015), pp.~1071--1129.

\bibitem{Gromov81}
{\sc M.~Gromov}, {\em Groups of polynomial growth and expanding maps},
  Publications Math\'{e}matiques de l'Institut des Hautes \'{E}tudes
  Scientifiques, 53 (1981), pp.~53--78.

\bibitem{Gromov07}
\leavevmode\vrule height 2pt depth -1.6pt width 23pt, {\em Metric structures
  for {R}iemannian and non-{R}iemannian spaces}, Modern Birkh\"auser Classics,
  Birkh\"auser Boston Inc., Boston, MA, english~ed., 2007.
\newblock Based on the 1981 French original, With appendices by M. Katz, P.
  Pansu and S. Semmes, Translated from the French by Sean Michael Bates.

\bibitem{GP91}
{\sc K.~Grove and P.~Petersen}, {\em Manifolds near the boundary of existence},
  J. Differential Geom., 33 (1991), pp.~379--394.

\bibitem{Heinonen01}
{\sc J.~Heinonen}, {\em Lectures on Analysis on Metric Spaces}, Universitext
  (Berlin. Print), Springer New York, 2001.

\bibitem{Jansen17}
{\sc D.~Jansen}, {\em Notes on {P}ointed {G}romov-{H}ausdorff {C}onvergence}.
\newblock Preprint, arXiv:1703.09595, 2017.

\bibitem{Lohr}
{\sc W.~L\"{o}hr}, {\em Equivalence of {G}romov-{P}rohorov- and {G}romov's
  {$\underline\square_\lambda$}-metric on the space of metric measure spaces},
  Electron. Commun. Probab., 18 (2013), pp.~no. 17, 10.

\bibitem{Lott-Villani09}
{\sc J.~Lott and C.~Villani}, {\em Ricci curvature for metric-measure spaces
  via optimal transport}, Ann. of Math. (2), 169 (2009), pp.~903--991.

\bibitem{MacLane98}
{\sc S.~Mac~Lane}, {\em Categories for the working mathematician}, vol.~5 of
  Graduate Texts in Mathematics, Springer-Verlag, New York, second~ed., 1998.

\bibitem{Roe03}
{\sc J.~Roe}, {\em Lectures on coarse geometry}, 2003.
\newblock American Mathematical Society.

\bibitem{Semola20}
{\sc D.~Semola}, {\em Recent developments about {G}eometric {A}nalysis on $\rm
  {R}{C}{D}({K},{N})$ spaces}, 2020.
\newblock PhD Thesis, Scuola Normale Superiore di Pisa.

\bibitem{Sturm06I}
{\sc K.-T. Sturm}, {\em On the geometry of metric measure spaces. {I}}, Acta
  Math., 196 (2006), pp.~65--131.

\bibitem{Sturm06II}
\leavevmode\vrule height 2pt depth -1.6pt width 23pt, {\em On the geometry of
  metric measure spaces. {II}}, Acta Math., 196 (2006), pp.~133--177.

\end{thebibliography}
\end{document}